\documentclass[11pt,a4paper]{article}
\usepackage[all,2cell]{xy}
\UseAllTwocells
\usepackage{amssymb,amsmath,amsthm, euscript}
\usepackage{bbm}
\usepackage{enumitem}
\usepackage{fullpage}
\usepackage{leftidx}
\usepackage{hyperref}
\usepackage{youngtab}
\usepackage{tikz}
\usepackage{tikz-cd}
\usepackage{bm}
\usetikzlibrary{decorations.markings}
\usetikzlibrary{calc}

\DeclareMathOperator{\id}{id}

\DeclareMathOperator{\cone}{cone}

\DeclareMathOperator{\tr}{tr}
\DeclareMathOperator{\Ch}{Ch}
\DeclareMathOperator{\Sing}{Sing}
\DeclareMathOperator{\proj}{proj}
\DeclareMathOperator{\Hom}{Hom}
\DeclareMathOperator{\Span}{Span}
\DeclareMathOperator{\Map}{Map}
\DeclareMathOperator{\colim}{colim}
\DeclareMathOperator{\N}{N}
\DeclareMathOperator{\dg}{dg}
\DeclareMathOperator{\h}{h}
\DeclareMathOperator{\Ndg}{N_{dg}}
\DeclareMathOperator{\Spanf}{Span^f}

\DeclareMathOperator{\Aut}{Aut}

\DeclareMathOperator{\GL}{GL}
\DeclareMathOperator{\ob}{ob}

\DeclareMathOperator{\iso}{iso}

\renewcommand{\labelenumi}{(\theenumi)}
\theoremstyle{definition}
\numberwithin{equation}{section}
\newtheorem{exa}[equation]{Example}
\newtheorem{defi}[equation]{Definition}
\newtheorem{prop}[equation]{Proposition}
\newtheorem{theo}[equation]{Theorem}
\newtheorem{lem}[equation]{Lemma}
\newtheorem{cor}[equation]{Corollary}

\newtheorem{rem}[equation]{Remark}

\newtheorem{question}[equation]{Question}

\newcommand{\cn}{\langle n \rangle}
\newcommand{\cm}{\langle m \rangle}
\newcommand{\A}{\EuScript{A}}
\def\S{\EuScript{S}}
\newcommand{\X}{\EuScript{X}}
\newcommand{\Y}{\EuScript{Y}}
\newcommand{\Z}{\EuScript{Z}}
\newcommand{\B}{\EuScript{B}}
\newcommand{\C}{\EuScript{C}}
\def\P{\EuScript{P}}

\newcommand{\T}{\EuScript{T}}
\newcommand{\D}{\EuScript{D}}
\newcommand{\F}{\EuScript{F}}
\newcommand{\Zt}{\mathbb Z/(2)}
\newcommand{\one}{\mathbbm 1}

\def\cross{+}
\def\frame{\Box}
\def\square{\boxplus}
\def\Triv{\operatorname{\EuScript{T}riv}}
\def\Split{\operatorname{\EuScript{S}plit}}
\def\GExt{\operatorname{\EuScript{E}xt}}

\def\on{\operatorname}

\def\op{\on{op}}
\def\Ch{\mathbf{Ch}}
\def\k{\mathbf{k}}

\def\Grpd{\mathbf{Grpd}}

\def\Top{\mathbf{Top}}

\def\Vect{\mathbf{Vect}}
\def\dgcat{\mathbf{dgcat}}
\def\Vectone{\mathbf{Vect}_{\mathbb F_1}}
\def\Vectq{\mathbf{Vect}_{\mathbb F_q}}
\def\im{\on{im}}

\def\Ext{\on{Ext}}
\def\Fun{\on{Fun}}

\def\Hall{\on{Hall}}

\def\O{ {\mathcal O}}
\def\ZZ{ {\mathbb Z}}
\def\NN{ {\mathbb N}}
\def\FF{ {\mathbb F}}
\def\Fq{ {\mathbb F}_q}
\def\Fone{ {\mathbb F}_1}

\def\RR{ {\mathbb R}}
\def\QQ{ {\mathbb Q}}
\def\CC{ {\mathbb C}}
\def\lra{ \longrightarrow}
\newcommand{\m}{\frak{m}}

\title{Higher categorical aspects of Hall Algebras}
\author{T. Dyckerhoff}

\begin{document}

\maketitle

\begin{abstract}
	These are extended notes for a series of lectures on Hall algebras given at the CRM
	Barcelona in February 2015. The basic idea of the theory of Hall algebras is that the
	collection of flags in an exact category encodes an associative multiplication law. While
	introduced by Steinitz and Hall for the category of abelian p-groups, it has since become
	clear that the original construction can be applied in much greater generality and admits
	numerous useful variations. These notes focus on higher categorical aspects based on the 
	relation between Hall algebras and Waldhausen's S-construction.
\end{abstract}

\tableofcontents
 
\thispagestyle{empty}

\renewcommand{\theenumi}{\arabic{enumi}}
\renewcommand{\labelenumi}{(\theenumi)}

\section*{Introduction}
\addcontentsline{toc}{section}{Introduction}

While studying the combinatorics of flags
\[
	M = M_0 \supset M_1 \supset \dots \supset M_n = 0
\]
of abelian $p$-groups, Hall \cite{hall} (and, in fact, more than $50$ years earlier
Steinitz \cite{steinitz}) had the striking insight that the numbers obtained by counting
flags in the various abelian $p$-groups $M$ form the structure constants of an associative algebra: 
the Hall algebra. The combinatorics of these structure constants is quite subtle and relates
beautifully to the theory of symmetric functions \cite{macdonald}.

Ringel \cite{ringel-hall} noticed that Hall's construction can be applied to the category of
representations of a quiver over a finite field $\Fq$. He showed that, in the case of a simply laced
Dynkin quiver, the resulting associative algebra realizes the upper triangular part of the quantum
group classified by the graph underlying the quiver. In the same context, Lusztig
\cite{lusztig:quivers} defined a geometric variant of Hall's algebra using perverse 
sheaves on moduli spaces of flags. These groundbreaking discoveries created a surge
of activity providing new perspectives on the theory of quantum groups (see the survey papers \cite{ringel-lectures},
\cite{schiffmann}, \cite{hubery}, \cite{rouquier}).
More recently, a host of variants of Hall algebras have been used by various authors to study
counting invariants: Reineke \cite{reineke-harder,reineke-counting} used Ringel's Hall algebra to
count rational points and compute Betti numbers of quiver moduli, Joyce \cite{joyce-motivic}
introduced motivic Hall algebras to study generalized Donaldson-Thomas invariants,
Kontsevich-Soibelman \cite{kontsevich-soibelman:cohomological} introduce cohomological Hall algebras
for a similar purpose. To\"en \cite{toen} has introduced Hall algebras for derived categories.

The point of view taken in this course is based on the observation that the collections $\S_n$ of
flags of varying length $n \ge 0$ naturally organize into a simplicial object 
\[
	\S_{\bullet} : \Delta^{\op} \lra \D,\; [n] \mapsto \S_n
\]
where the target category $\D$ depends on the context. The above variants of Hall algebras $H$ can
be obtained as various specializations of the simplicial object $\S_{\bullet}$: 
\begin{enumerate}
	\item \label{h1} Ringel's Hall algebra: $\S_{\bullet}$ is a simplicial groupoid and $H$ is obtained
		by passing to functions on isomorphism classes
	\item \label{h2} Lusztig's geometric Hall algebra: $\S_{\bullet}$ is a simplicial stack and $H$ is obtained by passing to
		the Grothendieck group of perverse sheaves 
	\item \label{h3} Joyce's motivic Hall algebra: $\S_{\bullet}$ is a simplicial stack and $H$ is obtained
		by passing to relative Grothendieck groups
	\item \label{h4} Kontsevich-Soibelman's cohomological Hall algebra: $\S_{\bullet}$ is a simplicial stack and $H$ is obtained
		by passing to equivariant cohomology
	\item \label{h5} To\"en's derived Hall algebra: $\S_{\bullet}$ is a simplicial space and $H$ is obtained by passing to
		locally constant functions
\end{enumerate}
Remarkably, the simplicial object $\S_{\bullet}$ also plays a central role in the work of Waldhausen
\cite{waldhausen} on algebraic K-theory where it is simply called the {\em $\S_{\bullet}$-construction}. We
adopt this terminology even though the relation to algebraic K-theory will not play any role in
these notes.\\

The following question will serve as a guide for this course:

\begin{question}\label{question} To what extent can we study Hall algebras ``universally'' in 
terms of the simplicial object $\S_{\bullet}$ without passing to any of the above specializations?
\end{question}
A first answer is that associativity of the Hall algebra can be seen on a universal level: it
corresponds to the fact that $\S_{\bullet}$ satisfies certain natural conditions called 
{\em $2$-Segal conditions} in \cite{dk-segal}.\\

We outline the contents of these notes:\\

In Section \ref{section:classical}, after giving the historical definition of Hall's algebra of
partitions, we define Hall algebras for {\em finitary proto-abelian categories}. This allows us to
treat categories linear over $\Fq$ and categories ``linear over $\Fone$'' on the same footing. Via
two examples, we demonstrate the idea of \cite{szczesny} to interpret, in some cases, the Hall
algebra of an $\Fq$-linear category as a $q$-analog of the Hall algebra of a suitably defined
$\Fone$-linear category. The methods used in this section pay tribute to combinatorial aspects such
as statistics. The main reference is \cite{macdonald}.\\

The remaining sections are devoted to the analysis of Question \ref{question} for various target
categories $\D$:\\

Section \ref{section:groupoids}: $\D = \{ \text{groupoids} \}$. We introduce the simplicial groupoid
$\S_{\bullet}$ of flags in a proto-abelian category $\C$ and establish the $2$-Segal conditions. We
construct from $\S_{\bullet}$ an {\em abstract Hall algebra} in the monoidal category of spans of
groupoids. This is an instance of the Baez-Hoffnung-Walker {\em groupoidification} program
\cite{baez-groupoidification,walker}. We attempt to advertize the benefits of this point of view by
lifting the main part of the proof of Green's theorem \cite{green,ringel-green} into this language,
following proposals of Baez and Kapranov. The Hall algebra of Section \ref{section:classical} can,
for finitary $\C$, be obtained from the abstract Hall algebra by passing to groupoid functions as
indicated in (1). We conclude with the observation that the simplicial $2$-Segal groupoid
$\S_{\bullet}$ encodes structure of higher categorical nature which is lost by passing to functions:
we define {\em Hall monoidal categories} via a construction which fits into Day's \cite{day} theory
of monoidal convolution. For the category $\Vectone$, we recover a classical monoidal category:
Schur's \cite{schur} category of polynomial functors. For $\Vectq$, we obtain Joyal-Street's
\cite{joyal-street} category of representations of the general linear groupoid over a finite
field.\\

Section \ref{section:derived}: $\D = \{ \text{$\infty$-groupoids} \}$. We introduce the Waldhausen
$\S_{\bullet}$-construction of a pretriangulated differential graded category and establish the
$2$-Segal conditions. We explain how to obtain To\"en's derived Hall algebra by passing to locally 
constant functions. This is a reformulation of the results of \cite{toen} using the language of
$\infty$-categories which makes the constructions entirely parallel to the ones of Section
\ref{section:classical}. The material in this section is taken from \cite{dk-segal}.\\ 

Section \ref{section:triangulated}: $\D = \{ \text{differential $\Zt$-graded categories} \}$. One of
the motiviations for formulating Question \ref{question} is as follows: the $2$-Segal simplicial
space $\S_{\bullet}$ from Section \ref{section:derived} can be constructed for any stable
$\infty$-category $\C$ -- finiteness conditions on $\C$ are only needed when passing to functions.
For example, the $\S_{\bullet}$-construction of a $2$-periodic dg category makes perfect sense while
it is not clear how to pass to functions since the finiteness conditions are violated. In this example, we give
an answer to Question \ref{question} which can {\em only} be seen universally: the
$\S_{\bullet}$-construction has a canonical cyclic structure. Observed heuristically in
\cite{dk-segal}, this statement has since been established in \cite{nadler,dk-triangulated,
lurie-cyclic}. 

The cyclic structure should be regarded as a symmetric Frobenius structure on the abstract Hall
algebra of a $2$-periodic dg category. Instead of making this statement precise, we give an
application: using a categorification of the state sum formalism from $2$-dimensional topological
field theory, we obtain invariants of oriented surfaces (\cite{dk-triangulated}). Remarkably, a
universal variant of this construction recovers Kontsevich's construction \cite{kontsevich-fukaya}
of a version of the Fukaya category of a noncompact Riemann surface. 

A nonlinear generalization of the results discussed in this section, and beyond, has been given by Lurie
\cite{lurie-cyclic} who also provides an interpretation in terms of a rotation invariance statement for
algebraic $K$-theory.\\

Throughout this text, we consider categories which are small without explicitly mentioning this.
For example, when we speak about the category of all sets then we really mean the category of $U$-small sets for
a chosen Grothendieck universe which we leave implicit. Other examples of categories like groupoids,
vector spaces, etc, are defined as small categories in a similar way.\\

\noindent
{\bf Acknowledgements.} I would like to thank the organizers of the ``Advanced Course: (Re)emerging
methods in commutative algebra and representation theory'' at the CRM Barcelona for the invitation
to give this lecture series. I thank Mikhail Kapranov for the many inspiring discussions on the
subject. Much of the material presented in these notes originates either directly or indirectly from
our joint work. Many thanks to Pranav Pandit for influential comments and for leading the working
sessions. Further, I would like to thank Joachim Kock for interesting discussions. Last but not
least, I thank Malte Leip for correcting many typos in a first draft of these notes.

\newpage
\section{Classical Hall algebras}
\label{section:classical}

\subsection{Hall's algebra of partitions}

We outline the original context in which Hall algebras first appeared. Let $p$ be
a prime number, and let $M$ be a finite abelian $p$-group. By the classification theorem for
finitely generated abelian groups the group $M$ decomposes into a direct sum of cyclic $p$-groups.
Therefore, we have
\[
	M \cong \bigoplus_{i = 1}^{r} \ZZ/(p^{\lambda_i})
\]
where we may assume $\lambda_1 \ge \lambda_2 \ge \dots \ge \lambda_r$ so that the sequence
\[
	\lambda = (\lambda_1, \lambda_2, \dots, \lambda_r, 0, \dots)
\]
is a {\em partition}, i.e., a weakly decreasing sequence of natural numbers with finitely many
nonzero components. We call the partition $\lambda$ the {\em type} of $M$. 
The association
\[
	M \mapsto \text{type of $M$}
\]
provides a bijective correspondence between isomorphism classes of finite abelian $p$-groups 
and partitions.

Given partitions $\mu^{(1)}, \mu^{(2)}, \dots, \mu^{(s)}, \lambda$, we define
\[
	g_{\mu^{(1)} \mu^{(2)} \dots \mu^{(s)}}^{\lambda}(p)
\]
to be the number of flags
\[
	M = M_0 \supset M_1 \supset \dots \supset M_{s-1} \supset M_{s} = 0
\]
such that $M_{i-1}/M_i$ has type $\mu^{(i)}$ where $M$ is a fixed group of type $\lambda$. In this
context, Hall had the following insight:

\begin{theo}\label{theo:hall1} The numbers $g_{\mu \nu}^\lambda(p)$ form the structure constants of a unital
	associative algebra with basis $\{u_\lambda\}$ labelled by the set of all partitions. 
	More precisely, the $\ZZ$-linear extension of the formula
	\[
		u_{\mu} u_{\nu} = \sum_{\lambda} g_{\mu \nu}^\lambda(p) u_{\lambda}
	\]
	defines a unital associative multiplication on the abelian group $\bigoplus_{\lambda} \ZZ
	u_{\lambda}$. 
\end{theo}
\begin{proof} One shows that the product 
	\[
		u_{\mu} u_{\nu} u_{\lambda},
	\] 
	with any chosen bracketing, is equal to 
	\[
		\sum_{\pi} g_{\mu \nu \lambda}^{\pi}(p) u_{\pi}. 
	\]
	We will provide a proof of this statement in greater generality in Section
	\ref{section:groupoids} so that, at
	this point, we leave the details as an exercise.
\end{proof}

The resulting associative algebra is called {\em Hall's algebra of partitions}.
We compute some examples of products $u_{\mu} u_{\nu}$. Fixing an abelian $p$-group $M$ of type
$\lambda$, the number $g_{\mu \nu}^{\lambda}(p)$
is the number of subgroups $N \subset M$ such that $N$ has type $\nu$ and $M/N$ has type $\mu$.
In particular, we obtain that $g_{\mu \nu}^{\lambda}(p)$ is nonzero if and only if $M$ is an
extension of a $p$-group $N'$ of type $\mu$ by a $p$-group $N$ of type $\nu$.
\begin{enumerate}
	\item We compute 
		\[
			u_{(1)} u_{(1)} = g_{(1)(1)}^{(1,1)}(p) u_{(1,1)} + g_{(1)(1)}^{(2)}(p) u_{(2)}.
		\]
		Further, $g_{(1)(1)}^{(1,1)}(p)$ is the number of subgroups 
		\[
			N \subset M = \ZZ/(p) \oplus \ZZ/(p) 
		\]
		such that $N \cong \ZZ/(p)$ and $M/N \cong \ZZ/(p)$. This coincides with the number
		of $1$-dimensional subspaces in the $\FF_p$-vector space $(\FF_p)^2$ (here $\FF_p$
		denotes the field with $p$ elements) of which there are $p+1$. The number
		$g_{(1)(1)}^{(2)}(p)$ is the number of subgroups
		\[
			N \subset M = \ZZ/(p^2)
		\]
		such that $N \cong \ZZ/(p)$ and $M/N \cong \ZZ/(p)$. Any such $N$ must lie in the
		$p$-torsion subgroup of $M$ which is $p \ZZ/(p^2)$. But $p \ZZ/(p^2) \cong \ZZ/(p)$
		and so $N = p\ZZ/(p^2)$ which implies $g_{(1)(1)}^{(2)}(p) = 1$. In conclusion, we have
		\[
			u_{(1)} u_{(1)} = (p+1) u_{(1,1)} + u_{(2)}.
		\]
	\item We compute 
		\[
			u_{(1,1)} u_{(1)} = g_{(1,1)(1)}^{(1,1,1)}(p) u_{(1,1,1)} +
			g_{(1,1)(1)}^{(2,1)}(p) u_{(2,1)}.
		\]
		 To compute $g_{(1,1)(1)}^{(1,1,1)}(p)$ we have to determine the number of subgroups
		 \[
			 N \subset M = \ZZ/(p) \oplus \ZZ/(p) \oplus \ZZ/(p)
		 \]
		 such that $N \cong \ZZ/(p)$ and $M/N \cong \ZZ/(p) \oplus \ZZ/(p)$. This is
		 equivalent to counting $1$-dimensional subspaces of $\FF_p^3$ of which there are
		 $p^2 + p + 1$. Further, the number $g_{(1,1)(1)}^{(2,1)}(p)$ is the number of
		 subgroups
		 \[
			 N \subset M = \ZZ/(p^2) \oplus \ZZ/(p)
		 \]
		 such that $N \cong \ZZ/(p)$ and $M/N \cong \ZZ/(p) \oplus \ZZ/(p)$. As above, the
		 subgroup $N$ must be contained in the $p$-torsion subgroup of $M$ which is
		 $p\ZZ/(p^2) \oplus \ZZ/(p)$. There are $p+1$ such subgroups, but only one of them satisfies the
		 condition $M/N \cong \ZZ/(p) \oplus \ZZ/(p)$: $N = p \ZZ/(p^2)$, contained in the
		 first summand of $M$. Therefore, we have $g_{(1,1)(1)}^{(2,1)}(p) = 1$ so that 
		\[
			u_{(1,1)} u_{(1)} = (p^2 + p + 1) u_{(1,1,1)} + u_{(2,1)}.
		\]
\end{enumerate}

\subsection{Proto-abelian categories}

Our goal in this section will be to introduce a certain class of categories to which Hall's
associative multiplication law can be generalized.

\begin{defi}\label{defi:proto} A category $\C$ is called {\em proto-abelian} if the following conditions hold.
	\begin{enumerate}
		\item The category $\C$ is pointed. 
		\item 
	\begin{enumerate}
		\item Every diagram in $\C$ of the form
		\[
			\xymatrix{
				A \ar@{^{(}->}[r] \ar@{->>}[d] & B\\
				C  & 
			}
		\]
		can be completed to a pushout square of the form
		\[
			\xymatrix{
				A \ar@{^{(}->}[r] \ar@{->>}[d] & B\ar@{->>}[d] \\
				C \ar@{^{(}->}[r] &  D.
			}
		\]
		\item Every diagram in $\C$ of the form
		\[
			\xymatrix{
					& B \ar@{->>}[d]\\
				C \ar@{^{(}->}[r]  & D
			}
		\]
		can be completed to a pullback square of the form
		\[
			\xymatrix{
				A \ar@{^{(}->}[r] \ar@{->>}[d] & B\ar@{->>}[d] \\
				C \ar@{^{(}->}[r] &  D.
			}
		\]
		\end{enumerate}
		\item A commutative square in $\C$ of the form
		\[
			\xymatrix{
				A \ar@{^{(}->}[r] \ar@{->>}[d] & B\ar@{->>}[d] \\
				C \ar@{^{(}->}[r] &  D
			}
		\]
		is a pushout square if and only if it is a pullback square. We also call such a square
		{\em biCartesian}.
\end{enumerate}
\end{defi}

\begin{exa} A pushout diagram of the form
		\[
			\xymatrix{
				A \ar@{^{(}->}[r] \ar@{->>}[d] & B\ar@{->>}[d] \\
				0 \ar@{^{(}->}[r] &  A'
			}
		\]
		is called a {\em short exact sequence}, or, an {\em extension of $A'$ by $A$}.
\end{exa}

\begin{exa} 
	\begin{enumerate}
		\item Abelian categories are proto-abelian. 
		\item We introduce a category $\Vectone$ of {\em finite dimensional vector spaces
			over $\Fone$} as follows. An object of $\Vectone$ is a
			finite set $K$ equipped with a marked point $\ast \in K$ (a {\em pointed set}). A morphism $(K,\ast) \to
			(L,\ast)$ is a map $f: K \to L$ of underlying sets such that $f(\ast) = \ast$ and the restriction
			$f|_{K \setminus f^{-1}(\ast)}$ is injective. The category $\Vectone$ is
			proto-abelian.
		\item Let $\C$ be a proto-abelian category. Then the opposite category $\C^{\op}$ is
			proto-abelian. Given a category $I$, the category $\Fun(I, \C)$ of $I$-diagrams in
			$\C$ is proto-abelian.
	\end{enumerate}
\end{exa}

To define the Hall algebra of a proto-abelian category $\C$, we have to impose additional
finiteness conditions.
Two extensions $A \hookrightarrow B \twoheadrightarrow A'$
and $A \hookrightarrow C \twoheadrightarrow A'$ of $A'$ by $A$ are called equivalent if there exists a
commutative diagram
\[
	\xymatrix{ A \ar[d]_{\id} \ar@{^{(}->}[r] & B\ar[d]^{\cong} \ar@{->>}[r]& A' \ar[d]^{\id}\\
	A\ar@{^{(}->}[r] & C \ar@{->>}[r]& A'.}
\]
We denote by $\Ext_{\C}(A',A)$ the set of equivalence classes of extensions of $A'$ by $A$. 

\begin{defi}\label{defi:finitary} A proto-abelian category $\C$ is called {\em finitary} if, for every pair of objects
	$A$,$A'$, the sets $\Hom_{\C}(A',A)$ and $\Ext_{\C}(A',A)$ have finite cardinality.
\end{defi}

\begin{theo} Let $\C$ be a finitary proto-abelian category. Consider the free abelian group 
	\[
		\Hall(\C) = \bigoplus_{[M] \in \iso(\C)} \ZZ [M]
	\]
	on the set of isomorphism classes of objects in $\C$. Then the
	bilinear extension of the formula
	\[
		[N] \cdot [L] = \sum_{[M] \in \iso(\C)} g^M_{N,L} [M]
	\]
	where $g^M_{N,L}$ denotes the number of subobjects $A \subset M$ such that $A \cong L$ and
	$M/A \cong N$ defines a unital associative multiplication law on $\Hall(\C)$.  
\end{theo}
\begin{proof} The axioms for a proto-abelian category are chosen so that the argument for Theorem
	\ref{theo:hall1} generalizes verbatim.
\end{proof}

\subsection{A first example} 
\label{sec:first}

\subsubsection{The categories $\Vectq$ and $\Vectone$}

Let $\Vectq$ denote the category of finite dimensional vector spaces over the field $\Fq$
where $q$ is some prime power. The category $\Vectq$ is finitary proto-abelian, and we have
\[
	\Hall(\Vectq) = \bigoplus_{n \in \NN} \ZZ [\Fq^n]
\]
with multiplication given by 
\[
	[\Fq^n][\Fq^m] = g_{n,m}^{n+m}(q) [\Fq^{n+m}]
\]
where
\begin{align*}
	g_{n,m}^{n+m}(q) 
	& = |\{ V \subset \Fq^{n+m}\; | \; V \cong \Fq^m, \Fq^{n+m}/V \cong \Fq^n \}|\\
	& = \frac{|\{ \Fq^m \hookrightarrow \Fq^{n+m} \}|}{|\GL_n(\Fq)|}\\
	& = \frac{(q^{n+m} - 1)(q^{n+m}-q) \cdots (q^{n+m}-q^{m-1})}{(q^{m} - 1)(q^{m}-q) \cdots (q^{m}-q^{m-1})}\\
	& = \frac{(q^{n+m} - 1)(q^{n+m-1}-1) \cdots (q^{n+1}-1)}{(q^{m} - 1)(q^{m-1}-1) \cdots (q-1)}\\
	& = \left[ \begin{array}{c} n+m\\ m \end{array} \right]_q.
\end{align*}
The symbol appearing in the last line of the calculation denotes a $q$-binomial coefficient which is
defined as follows: We first define the $q$-analog of a natural number $n \in \NN$ to be
\[
	[n]_q = \frac{q^n - 1}{q-1} = 1 + q + q^2 + \dots + q^{n-1}.
\]
The $q$-factorial of $n$ is defined as
\[
	[n]_q! = [n]_q [n-1]_q \dots [1]_q
\]
and we finally set
\[
	\left[ \begin{array}{c} n+m\\ m \end{array} \right]_q = \frac{[n+m]_q!}{[m]_q! [n]_q!}.
\]
We have an isomorphism of $\ZZ$-algebras
\begin{align*}
	\Hall(\Vectq) & \overset{\cong}{\lra} \ZZ[x,\frac{x^2}{[2]_q!},\frac{x^3}{[3]_q!}, \dots]
	\subset \QQ[x],\\
	[\Fq^n] & \mapsto \frac{x^n}{[n]_q!}.
\end{align*}
On the other hand, consider the category $\Vectone$ which is easily seen to be finitary and
proto-abelian. We have 
\[
	\Hall(\Vectone) = \bigoplus_{n \in \NN} \ZZ [\{\ast, 1,2,\dots,n\}]
\]
with multiplication
\[
	[\{\ast, 1,2,\dots,n\}][\{\ast, 1,2,\dots,m\}] = \lambda_{n,m}^{n+m} [\{\ast, 1,2,\dots,n+m\}]
\]
where 
\begin{align*}
	\lambda_{n,m}^{n+m} 
	& = \frac{|\{ \{\ast, 1,2,\dots,m\} \hookrightarrow \{\ast, 1,2,\dots,n+m\}\}|}{|S_m|}\\
	& = \left( \begin{array}{c} n+m\\ m \end{array} \right).
\end{align*}
We have an isomorphism of algebras
\begin{align*}
	\Hall(\Vectone) & \overset{\cong}{\lra} \ZZ[x,\frac{x^2}{2!},\frac{x^3}{3!}, \dots]
	\subset \QQ[x],\\
	[\Fone^n] & \mapsto \frac{x^n}{n!}.
\end{align*}
Therefore, the algebra $\Hall(\Vectone)$ is the free divided power algebra on one generator.
Note that, interpreting $q$ as a formal variable, we have 
\begin{equation}\label{eq:q-analog}
		\lambda_{n,m}^{n+m} = g_{n,m}^{n+m}(1)
\end{equation}
so that $\Hall(\Vectq)$ can be regarded as a $q$-analog of $\Hall(\Vectone)$. 

\begin{rem} We can give a natural explanation for the commutativity of both Hall algebras computed
	in this section: Assume that a finitary proto-abelian category $\C$ is equipped with an exact
	equivalence $D: \C^{\op} \to \C$ such that, for every object $M$, we have $D(M) \cong M$.
	Then $\Hall(\C)$ is a commutative algebra. Now we have:
	\begin{enumerate}
		\item The category $\Vectq$ is equipped with the exact duality $V \mapsto V^* = \Hom_{\Vectq}(V,
			\Fq)$ which satisfies $V^* \cong V$.
		\item The category $\Vectone$ is equipped with the exact duality $K \mapsto K^* =
			\Hom_{\Vectone}(K,\{1,\ast\})$. Curiously, in contrast to $\Vectq$, the dual $K^*$
			can be canonically identified with $K$.
	\end{enumerate}
\end{rem}

\subsubsection{Statistical interpretation of $q$-analogs}
\label{subsec:statistic}

We introduce terminology to discuss the phenomenon of equation \eqref{eq:q-analog} somewhat
more systematically.  
Let $S$ be a finite set. A {\em statistic} on $S$ is a function
\[
	f: S \lra \NN.
\]
Given a statistic $f$, we define the corresponding {\em partition function} to be
\[
	Z(q) = \sum_{s \in S} q^{f(s)}.
\]

\begin{rem}
Evaluation of the partition function at $q=1$ yields the cardinality of the set $S$ so that $Z(q)$
can be interpreted as a $q$-analog of $|S|$. Note that, any $q$-analog obtained in this way from a statistic will
therefore, by construction, be polynomial in $q$.
\end{rem}

\begin{exa}
	\begin{enumerate}
\item Consider the set $S = \{1,\dots,n\}$. We define a statistic on $S$ via
	\[
		f: S \lra \NN,\; i \mapsto i-1.
	\]
	The corresponding partition function is
	\[
		Z(q) = 1 + q + \dots + q^{n-1} = [n]_q.
	\]
\item Consider the set $S_n$ underlying the symmetric group on $n$ letters. We define the {\em
	inversion statistic} on $S_n$ as
	\[
		\on{inv}: S_n \lra \NN,\; \sigma = \left( \begin{array}{cccc} 1 & 2 & \dots & n\\ \sigma_1 &
			\sigma_2 & \dots & \sigma_n \end{array} \right) \mapsto |\{ (i,j)\; |\; i
			<j, \; \sigma_i > \sigma_j \}|
	\]
	We claim that we have
	\[
		\sum_{\sigma \in S_n} q^{\on{inv}(\sigma)} = [n]_q!.
	\]
	To show this, we interpret the summands in the expansion of the product
	\[
		[n]_q! = 1 (1+q) (1+q+q^2) \dots (1 + q + \dots + q^{n-1}).
	\]
	Given a summand $q^a$ of the product, it must arise as a product $q^{i_1}q^{i_2}\dots
	q^{i_n}$ with $0 \le i_k \le k-1$. We produce a corresponding permutation $\sigma$ by providing an
	algorithm to write the list $\sigma_1, \dots, \sigma_n$. In step $1$, we start by writing the number $1$. In step $2$, we write the number $2$ to the
	left of $1$ if $i_2 = 1$ or to the right of $1$ if $i_2 = 0$. At step $k$, there will be
	$k-1$ numbers and we label the gaps between the numbers by $0, \dots, k-1$ from right to left. We
	fill in the number $k$ into the gap with label $i_k$. This algorithm produces a permutation
	$\sigma$ with $\on{inv}(\sigma) = a$. The claim follows immediately from this construction.
\end{enumerate}
\end{exa}

Finally, we would like to find a statistical interpretation of the $q$-binomial coefficient. We will
achieve this by defining a statistic on the set $P(m,n+m)$ of subsets of $\{1,\dots,n+m\}$ of
cardinality $m$. A lattice path in the rectangle of size $(n,m)$ is path in $\RR^2$ which begins at
$(0,0)$, ends at $(n,m)$, and is obtained by sequence of steps moving either one integer step to the
east or to the north. Given $K \in P(m,n+m)$, we can construct a lattice path by the following rule:
at step $i$, we move east if $i \notin K$, and north if $i \in K$. It is immediate that this
construction provides a bijective correspondence between $P(m,n+m)$ and lattice paths in the
rectangle of size $(n,m)$. We now define the statistic as
\[
	a: P(m,n+m) \lra \NN, \; K \mapsto a(K)
\]
where $a(K)$ denotes the area of the part of the rectangle of size $(n,m)$ which lies above the
lattice path corresponding to $K$.

\begin{prop} We have 
	\[
		\sum_{K \in P(m,n+m)} q^{a(K)} = \left[ \begin{array}{c} n+m\\ m \end{array} \right]_q.
	\]
\end{prop}
\begin{proof} We prove this equality by showing that both sides satisfy the recursion
	\[
		Q(m,n)(q) = q^n Q(m-1,n)(q) + Q(m,n-1).
	\]
	On the left hand side, the term $q^n Q(m-1,n)(q)$ is the contribution from subsets $K$ such
	that $n+m \in K$, the term $Q(m,n-1)$ is the contribution from subsets such that $n+m \notin
	K$. The right hand side satisfies the recursion by a straightforward calculation.
\end{proof}

In conclusion, we obtain that the structure constants of the Hall algebra of $\Vectq$ have a
statistical interpretation and are hence polynomial in $q$. The value at $q = 1$ is realized by the
structure constants of the Hall algebra of $\Vectone$.

The phenomenon observed seems to be a general feature: the Hall algebra of a finitary $\FF_q$-linear
abelian category often describes a $q$-analog of an algebra of more classical nature obtained by
specializing at $q=1$. The idea to describe this latter algebra as the Hall algebra of a
category of combinatorial nature is quite recent and due to M. Szczesny
\cite{szczesny}. 

In the next section, we will discuss this phenomenon in the classical context:
We will see that Hall's algebra of partitions is isomorphic to the ring $\Lambda$ of symmetric functions
and equips it with an interesting $\ZZ$-basis given, up to some renormalization, by the so-called
{\em Hall-Littlewood symmetric functions}. This basis is a $q$-analog of the classical basis of
$\Lambda$ given by the {\em monomial symmetric functions} which can be obtained via the Hall
algebra of a proto-abelian category of combinatorial nature: the category of vector spaces over
$\FF_1$ equipped with a nilpotent endomorphism. 

\subsection{Hall's algebra of partitions and symmetric functions}

Let $R$ be a discrete valuation ring with (a principal ideal domain with exactly one nonzero
maximal ideal $\m$) with residue field of finite cardinality $q$. It turns out that the Hall algebra 
of the category of finite $R$-modules only depends on $q$. For the $p$-adic integers $\ZZ_p$,
we recover the category of abelian $p$-groups and thus Hall's classical algebra of partitions.
For $R = \FF_q[ [t]]$ we obtain another construction of the same Hall algebra.

We will analyze the algebra $\Hall(\FF_q[ [t]])$ by establishing an interpretation similar to the one
provided in Section \ref{sec:first}: The case $q=1$ corresponds to the Hall algebra of a
proto-abelian category of combinatorial nature and $\Hall(R)$ is studied as a $q$-analog by
giving a statistical interpretation of (some of) its structure constants.

\subsubsection{The Hall algebra of $\FF_1[ [t]]$} 
\label{subsec:hallf1}

\begin{defi}
A {\em finite $\FF_1[ [t]]$-module} is an object of $\Vectone$ equipped with a nilpotent endomorphism.  
\end{defi}

The isomorphism classes of finite $\FF_1[ [t]]$-modules are naturally labelled by
the set of all partitions: a pointed set $K$ equipped with a nilpotent endomorphism $T$, 
corresponds to rooted tree with vertices given by the elements of $K$ and an edge from $k$ to $k'$
if $T(k) = k'$. We obtain a partition by reordering the tupel of lengths of the branches.

\begin{exa} The finite $\FF_1[ [t]]$-module corresponding to the partition $(3,2,2,1,0,\dots)$ is
	represented by the rooted tree
	
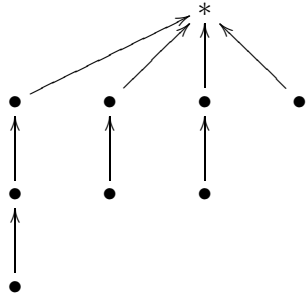
\[
		\xymatrix{ & & \ast &  \\
			\bullet \ar[urr] & \bullet \ar[ur] & \bullet \ar[u] & \bullet \ar[ul]. \\
			\bullet \ar[u] & \bullet \ar[u] & \bullet \ar[u] &  \\
		\bullet \ar[u] &  &  &  }
	\]
\end{exa}

We write  
\[
	\Hall(\FF_1[ [t]]) = \bigoplus_{\lambda} \ZZ u_\lambda
\]
for the Hall algebra of the category of finite $\FF_1[ [t]]$-modules where $\lambda$ runs over all
partitions.

\begin{exa} An example of a short exact sequence of finite $\FF_1[ [t]]$-modules is given by 
	
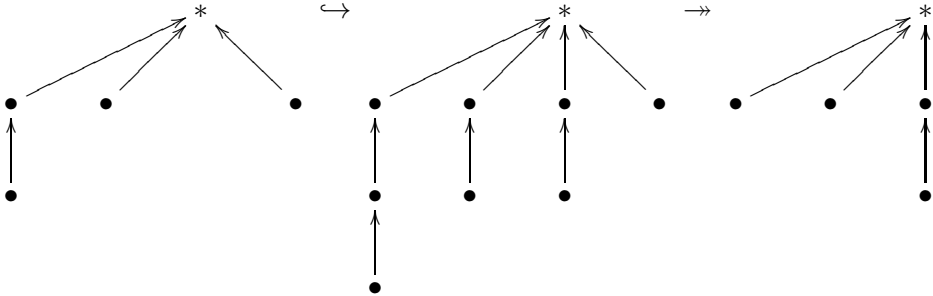
\[
		\xymatrix{ & & \ast &  \\
			\bullet \ar[urr] & \bullet \ar[ur] &  & \bullet \ar[ul] \\
			\bullet \ar[u] &  & &  \\
		&  &  &  } \hookrightarrow
		\xymatrix{ & & \ast &  \\
			\bullet \ar[urr] & \bullet \ar[ur] & \bullet \ar[u] & \bullet \ar[ul] \\
			\bullet \ar[u] & \bullet \ar[u] & \bullet \ar[u] &  \\
		\bullet \ar[u] &  &  &  } \twoheadrightarrow
		\xymatrix{ & & \ast &  \\
			\bullet \ar[urr] & \bullet \ar[ur] & \bullet \ar[u] &  \\
			&& \bullet \ar[u] &  \\
		&  &  &  }
	\]
\end{exa}

Given a finite $\FF_1[ [t]]$-module $K$ with nilpotent endomorphism $T$, the dual
$D(K) = \Hom(K, \{1,\ast\})$ is naturally equipped with a nilpotent endomorphism obtained by
precomposing with $T$. The corresponding rooted tree $\Gamma_{D(K)}$ is obtained from
$\Gamma_{K}$ by removing $\ast$, reversing the orientation of all edges, and adding $\ast$
as a root. This description implies that $D(K)$ has the same branch lengths and is hence
isomorphic to $K$. We deduce that $\Hall(\FF_1[ [t]])$ is commutative.

\begin{prop}\label{prop:hallf1} Let $\lambda$ be a partition of length $s$. Then, in $\Hall(\FF_1[ [t]])$, we have
	\begin{equation}\label{eq:formula}
		u_{(1^{\lambda_1})}u_{(1^{\lambda_2})} \cdots u_{(1^{\lambda_s})}= \sum_{\mu} a_{\lambda \mu} u_{\mu}
	\end{equation}
	where $a_{\lambda \mu}$ denotes the number of $\NN$-by-$\NN$ matrices with entries in
	$\{0,1\}$ with column sums $\lambda$ and row sums $\mu$.
\end{prop}
\begin{proof}
	By definition, the coefficient $a_{\lambda \mu}$ is the number of flags
	\[
		K = K_0 \supset K_1 \supset \dots \supset K_{s-1} \supset K_s = \{\ast\}
	\]
	where $K$ is fixed of type $\mu$ and $K_{i-1}/K_i$ has type $(1^{\lambda_i})$. We represent
	$K$ as an oriented graph with branches labelled by $1, \dots, s$. Let $K_1
	\subset K$ be a submodule such that $K/K_1$ has type $(1^{\lambda_1})$. Then the set $K
	\setminus K_1$ consists of exactly $\lambda_1$ elements which form the tips of pairwise
	disjoint branches. We may encode this in a vector $v_1$ in $\{0,1\}^s$ where we mark those branches of
	$K$ which contain a point in $K \setminus K_1$ by $1$ and all remaining branches by $0$.
	note that the sum over all entries in $v_1$ equals $\lambda_1$. We
	repeat this construction for each $K_i \subset K_{i-1}$ and organize the resulting vectors
	$v_1, v_2, \dots, v_s$ as the columns of a matrix. By construction, this matrix has column sums
	$\lambda$ and row sums $\mu$. This construction establishes a bijection between flags of
	the above type and $\{0,1\}$-matrices with column sums $\lambda$ and row sums $\mu$.
\end{proof}

To analyze the nature of the matrix $(a_{\lambda \mu})$ indexed by the set of all partitions, we
introduce some terminology: We define two orders on the set $P_n$ of partitions of a natural number $n$:
\begin{enumerate}
	\item The {\em lexicographic order}: $\mu \le^l \lambda$ if $\mu = \lambda$ or the first
		nonzero difference $\mu_i - \lambda_i$ is negative. 
	\item The {\em dominance order}: $\mu \le^d \lambda$ if, for all $i \ge 1$, 
		\[
			\mu_1 + \dots + \mu_i \le \lambda_1 + \dots + \lambda_i.
		\]
\end{enumerate}
Observe that $\mu \le^d \lambda$ implies $\mu \le^l \lambda$. 

\begin{prop} Define $\widetilde{a}_{\lambda \mu} = a_{\lambda' \mu}$. Then, we have $\widetilde{a}_{\lambda \lambda} = 1$, and
	$\widetilde{a}_{\lambda \mu} = 0$ unless $|\lambda| = |\mu|$ and $\mu \le^d \lambda$.
\end{prop}
\begin{proof} First assume that $(A_{ij})$ is a $\{0,1\}$-matrix with column sums $\lambda'$ and
	row sums $\mu$ so that $(A_{ij})$ has no {\em gaps}. Here a gap is a $0$ entry in a column
	which is followed below by an entry $1$ in the same column. The condition that
	$(A_{ij})$ has no gaps means that the $1$-entries of the matrix constitute a Young diagram
	of the partition $\mu$ whose transpose is $\lambda'$. This implies $\lambda = \mu$. Vice
	versa, it is easy to see that this ``Young'' matrix is the unique $\{0,1\}$-matrix with
	column sums $\lambda'$ and row sums $\lambda$. Therefore, we obtain $\widetilde{a}_{\lambda\lambda} = 1$.
	Now suppose that $(A_{ij})$ is a $\{0,1\}$-matrix with column sums $\lambda'$ and row sums
	$\mu$ which has a gap. Pick a column with a gap and swap the $0$ forming the gap with the
	lowest $1$ in the same column, thus obtaining a new matrix $(\widetilde{A}_{ij})$. The
	column sums of the new matrix have not changed. The sequence of row sums $\alpha =
	(\alpha_1,\alpha_2,\dots)$ has changed, in particular, it may not form a partition. But,
	enlarging the definition of the dominance order from partitions to arbitrary sequences with
	entries in $\NN$, it is immediate to verify $\alpha \ge^d \mu$. We obtain a modified matrix
	$(\widetilde{A}_{ij})$ with less gaps, column sums $\lambda'$, and row sums given by $\alpha
	\ge^d \mu$. 
	If the matrix $(\widetilde{A}_{ij})$ has no gaps, then the above argument shows
	that $\alpha = \lambda$. Otherwise we iterate, producing a totally ordered chain
	of sequences in $\NN$
	\[
		\mu \le^d \alpha \le^d \cdots \le^d \lambda
	\]
	showing that $\mu \le^d \lambda$.
\end{proof}

The proposition implies that the matrix $(\widetilde{a}_{\lambda \mu})$ is upper unitriangular with
respect to the lexicographic order so that it is invertible over $\ZZ$. We deduce the invertibility
of the matrix $(a_{\lambda \mu})$ and conclude:

\begin{cor} The set $\{u_{(1^r)} |\; r > 0\}$ is algebraically independent and generates
	$\Hall(\FF_1[ [t]])$ as a $\ZZ$-algebra. In other words, 
	\[
		\Hall(\FF_1[ [t]]) = \ZZ[ u_{(1)}, u_{(1^2)}, \dots]
	\]
	is a polynomial ring in countably many variables.
\end{cor}

\subsubsection{Relation to symmetric functions}

Let $\ZZ[x_1,\dots,x_n]$ denote the ring of polynomials with integer coefficients. The symmetric
group $S_n$ acts by permuting the variables and the polynomials which are invariant under this action
are called {\em symmetric polynomials}. They form a ring which we denote by
\[
	\Lambda_n = \ZZ[x_1,\dots,x_n]^{S_n}.
\]
The ring $\Lambda_n$ is graded by total degree so that we have
\[
	\Lambda_n^k = \bigoplus_{k \ge 0} \Lambda_n^k.
\]
Given a tupel $\alpha = (\alpha_1,\alpha_2,\dots,\alpha_n) \in \NN^n$, we obtain a monomial
\[
	x^\alpha = x_1^{\alpha_1} x_2^{\alpha_2} \cdots x_n^{\alpha_n}.
\]
For a partition $\lambda$ of length $l(\lambda) \le n$, we define
\[
	m_\lambda(x_1,\dots,x_n) = \sum_\alpha x^\alpha
\]
where the sum ranges over all distinct permutations $\alpha$ of $(\lambda_1, \lambda_2, \dots, \lambda_n)$.
For example, we have
	\begin{align*}
		m_{(1,1,\dots,1)}(x_1,\dots,x_n) & = x_1 x_2 \cdots x_n\\
		m_{(1,1,0,\dots,0)}(x_1,\dots,x_n) & = \sum_{i < j} x_i x_j\\
		m_{(k,0,0,\dots,0)}(x_1,\dots,x_n) & = x_1^k + x_2^k + \dots + x_n^k.
	\end{align*}
It is immediate that the collection of polynomials $\{m_\lambda(x_1,\dots,x_n) | l(\lambda) \le n
\}$ forms a $\ZZ$-basis of $\Lambda_n$. In particular, the set $\{m_\lambda(x_1,\dots,x_n) |
l(\lambda) \le n,\; |\lambda| = k\}$ forms a $\ZZ$-basis of $\Lambda_n^k$. 

Many statements and formulas involving symmetric polynomials hold independently of the number of
variables $n$. This can be naturally incorporated by introducing {\em symmetric functions} which formalize the
notion of symmetric polynomials in countably many variables. Fix $k \ge 0$ and consider the projective 
system of abelian groups
\[
	\dots \lra \Lambda_{n+1}^k \overset{\rho_{n+1}}{\lra} \Lambda_n^k \overset{\rho_{n}}{\lra} \Lambda_{n-1}^k \lra \dots
\]
where the map $\rho_{n+1}: \Lambda_{n+1}^k \lra \Lambda_n^k$ is obtained by sending $x_{n+1}$ to $0$. We denote
the inverse limit of the projective system by
\[
	\Lambda^k = \underset{\longleftarrow}{\lim}\; \Lambda_{\bullet}^k.
\]

\begin{exa}\label{ex:monomials} Let $\lambda$ be a partition and let $n > l(\lambda)$. Then we have 
	\[
		m_\lambda(x_1,\dots,x_{n-1},0) = 
		m_\lambda(x_1,\dots,x_{n-1}). 
	\]
	Therefore, the sequence $\{ m_\lambda(x_1,\dots,x_n) | n > l(\lambda)\}$ defines an element
	of $\Lambda^k$ which we denote by $m_\lambda$. We call the symmetric functions
	$\{m_\lambda\}$
	the {\em monomial symmetric functions}.
\end{exa}

We finally define 
\[
	\Lambda = \bigoplus_{k \ge 0} \Lambda^k
\]
called the {\em ring of symmetric functions}. There is an apparent bilinear multiplication map $\Lambda^k
\times \Lambda^{k'} \lra \Lambda^{k+k'}$ which makes $\Lambda$ a graded ring. The set $\{ m_\lambda\}$ where $\lambda$ ranges over all partitions
forms a basis of the ring $\Lambda$ of symmetric functions.

We introduce another family of symmetric functions. For $r>0$, let 
\[
	e_r = m_{(1^r)} = \sum_{i_1 < i_2 < \dots < i_r} x_{i_1} x_{i_2} \cdots x_{i_r}
\]
and for a partition $\lambda$, we let
\[
	e_\lambda  = e_{\lambda_1}e_{\lambda_2} \cdots e_{\lambda_s}
\]
where $s = l(\lambda)$. The symmetric functions $\{e_{\lambda}\}$ are called {\em elementary
symmetric functions}. 

\begin{prop}\label{prop:key} 
	Let $\lambda$ be a partition. We have
	\[
			e_{\lambda} = \sum_{\mu} a_{\lambda \mu} m_{\mu}
	\]
	where, as above, $a_{\lambda \mu}$ denotes the number of $\{0,1\}$-matrices
	with column sums $\lambda$ and row sums $\mu$.
\end{prop}
\begin{proof} Let $x^\mu$, where $\mu$ is a partition, be a monomial which appears in the
	product expansion of
	\[
		e_{\lambda} = e_{\lambda_1}e_{\lambda_2} \cdots e_{\lambda_s}.
	\]
	This means that $x^\mu$ must be of the form
	\begin{equation}\label{eq:mu}
			x^\mu = y_1 y_2 \dots y_s
	\end{equation}
	where $y_j$ is a monomial term of $e_{\lambda_j}$. We write each monomial $y_j$ as
	\[
		y_j = \prod_i x_i^{A_{i j}}
	\]
	with $A_{ij} \in \NN$. We then observe that the condition that $y_j$ be a
	monomial term of $e_{\lambda_j}$ simply translates into the condition that the $j$th column of
	the matrix $(A_{ij})$ has entries in $\{0,1\}$ and satisfies $\sum_i A_{ij} = \lambda_j$.
	Similarly, the condition that equation \eqref{eq:mu} holds translates into the condition
	that, for every $i$, the $i$th row of the matrix $(A_{ij})$ sums to $\mu_i$. This implies
	that the number of terms $x^\mu$ (and hence, since the result is symmetric,
	the number of terms $m_\mu$) in $e_{\lambda}$ is given by $a_{\lambda \mu}$.
\end{proof}

We obtain immediate corollaries:

\begin{cor}[\cite{szczesny}]\label{cor:hallf1} There is a $\ZZ$-algebra isomorphism
	\[
		\varphi: \Hall(\FF_1[ [t]]) \overset{\cong}{\lra} \Lambda
	\]
	determined by $\varphi(u_{(1^r)}) = e_r$. We further have $\varphi(u_{\lambda}) =
	m_{\lambda}$ for an arbitrary partition $\lambda$.
\end{cor}

\begin{cor} The set $\{e_{\lambda}\}$ forms a $\ZZ$-basis of $\Lambda$. In particular, the set
	$\{e_1, e_2, \dots\}$ is algebraically independent and generates $\Lambda$ as a ring so that 
	\[
		\Lambda = \ZZ[e_1, e_2, \dots]
	\]
	is a polynomial ring in countably many variables.
\end{cor}

The statement of this last corollary is known as the {\em fundamental theorem on symmetric
functions}.

\subsubsection{Zelevinsky's statistic}

Let $R$ is a discrete valutation ring with residue field of cardinality $q$. Our discussion of
$\Hall(R)$ will be based on $q$-analogs of the arguments in Section \ref{subsec:hallf1} which are
obtained via a statistic introduced by Zelevinsky.

We start with some preparatory remarks. Recall that the length $l(M)$ of an
$R$-module $M$ is defined to be the length of a composition series of $M$. For example, the cyclic
module $R/\m^n$ has composition series
\[
	R/\m^n \supset \m/\m^n \supset \m^2/\m^n \supset \cdots \supset \m^{n-1}/\m^n \supset 0
\]
so that we have $l(R/\m^n) = n$. From this, we deduce that a finite module $M$ of type $\lambda$ has length
given by $l(M) = |\lambda|$. 
Further, given a finite module $M$, the sequence $(\mu_1,\mu_2,\dots)$
with
\[
	\mu_i = \dim_k (\m^{i-1}M/\m^iM)
\]
is the conjugate partition of $\lambda$. Finally, recall that the length is additive in short exact
sequences so that, for a submodule $N \subset M$, we have $l(M) = l(N) + l(M/N)$.

Let $\lambda$ be a partition of length $s$. We introduce numbers $b_{\lambda \mu}$ via the
equation
\[
u_{(1^{\lambda_1})}u_{(1^{\lambda_2})} \cdots u_{(1^{\lambda_s})} = \sum_{\mu} b_{\lambda \mu}
u_{\mu}.
\]
We will interpret the numbers $b_{\lambda \mu}$ as $q$-analogs of the numbers
$a_{\lambda \mu}$ from Proposition \ref{prop:hallf1} by introducing a statistic on the set
$\A_{\lambda \mu}$ of $\{0,1\}$-matrices with column sums $\lambda$ and row sums $\mu$. To this end,
we have to introduce some terminology.
For compositions $\alpha$, $\beta$, we define an array $A$ of shape $\alpha$ and weight $\beta$ to be a
labelling of the squares of the diagram of $\alpha$ such that the number $i$ appears $\beta_i$
times. We express an array as a function
\[
	A: \alpha \lra \NN^+
\]
where we consider (the diagram of) $\alpha$ as a subset of $\NN^+ \times \NN^+$. We will extend $A$ to all of
$\NN^+ \times \NN^+$ by letting elements in the complement of $\alpha$ have value $\infty$. For $x =
(i,j) \in \NN^+ \times \NN^+$, we denote $x^{\to} = (i,j+1)$. We call an array row-ordered (resp.
row-strict) if, for every $x \in \alpha$, $A(x^{\to}) \ge A(x)$ (resp. $A(x^{\to}) > A(x)$). Analogously, we define
column-ordered and column-strict arrays. 

\begin{rem}
	Given a row-strict array of shape $\mu$ and weight $\lambda$, we introduce a
	$\{0,1\}$-matrix which has entry $1$ at the positions $\{(i,A(x))\}$ where $x = (i,j)$
	runs over all $x \in \alpha$, and entry $0$ elsewhere. It is immediate that this construction provides
	a bijection between row-strict arrays of shape $\mu$ and weight $\lambda$ and the set
	$\A_{\lambda \mu}$ of $\{0,1\}$-matrices with column sums $\lambda$ and row sums $\mu$. 
\end{rem}

We define a lexicographic
order on $\NN^+ \times \NN^+$ via
\[
	(i,j) < (i',j') \quad \Longleftrightarrow \quad \text{either $j < j'$, or $j = j'$ and $i >
	i'$}.
\]
For a row-strict array $A$ of shape $\alpha$, we define
\[
	d(A) = |\{(x,y) \in \alpha \times \alpha | \; \text{$y < x$ and $A(x) < A(y) < A(x^{\to})$}
\} |.
\]
By the above remark, we obtain a statistic
\[
	d: \A_{\lambda \mu} \lra \NN, A \mapsto d(A)
\]
on the set of $\{0,1\}$-matrices with column sums $\lambda$ and row sums $\mu$ 
which we call {\em Zelevinsky's statistic}. It provides the following statistical 
interpretation of the coefficients $b_{\lambda \mu}$:

\begin{theo}[Zelevinsky] \label{theo:stat} We have
	\[
		b_{\lambda \mu} = \sum_A q^{d(A)}
	\]
	where $A$ ranges over all row-strict arrays of shape $\mu$ and weight $\lambda$.
\end{theo}

Before giving a proof, we analyze consequences of the theorem. We introduce the polynomials
$a_{\lambda\mu}(t) = \sum_A t^{d(A)}$ so that $b_{\lambda \mu} = a_{\lambda \mu}(q)$. We have:
\begin{enumerate}
	\item The polynomial $a_{\lambda\mu}(t)$ has nonnegative integral coefficients.
	\item $a_{\lambda\mu}(1)$ is the number of $\{0,1\}$-matrices with column sums $\lambda$ and
		row sums $\mu$.
	\item $a_{\lambda \mu}(t) = 0$ unless $\mu \le^d \lambda'$. Moreover, $a_{\lambda' \lambda}(t) = 1$.
\end{enumerate}
The statements of (1) and (2) are immediate. To see (3), note that we have $a_{\lambda \mu}(1) = 0$
unless $\mu \le^d \lambda'$ which implies the first statement.  $a_{\lambda' \lambda}(t) = 1$
follows from direct computation, using the fact that there is precisely one $\{0,1\}$-matrix with
column sums $\lambda'$ and row sums $\lambda$ and the corresponding array $A$ satisfies $d(A) = 0$.

We deduce that the matrix $(a_{\lambda\mu}(t))$ has an inverse over $\ZZ[t]$ and have the
following immediate consequences.

\begin{theo}\label{theo:hallfq} Let $R$ be a discrete valuation ring with residue field of cardinality $q$. Then the
	following hold:
	\begin{enumerate}
		\item The set $\{u_{(1^r)} |\; r > 0\}$ is algebraically independent and generates 
			$\Hall(R)$ as a $\ZZ$-algebra. In particular, there is a $\ZZ$-algebra isomorphism 
			\[
				\psi: \Hall(R) \overset{\cong}{\lra} \Lambda
			\]
			determined by $\psi(u_{(1^r)}) := m_{(1^r)}$. 
		\item The structure constants of $\Hall(R)$ are polynomial in $q$ so that 
			there exist polynomials $g_{\mu \nu}^\lambda(t)$ such that
			\[
				u_\mu u_\nu = \sum_\lambda g_{\mu \nu}^\lambda(q) u_\lambda.
			\]
			The numbers $g_{\mu \nu}^\lambda(1)$ are the structure constants of
			$\Hall(\FF_1[ [t]])$ and, hence, the structure constants of the $\ZZ$-basis
			of $\Lambda$ given by the monomial symmetric functions.
	\end{enumerate}
\end{theo}

From the theorem, we deduce that we have constructed a family of $\ZZ$-bases $\{\psi(u_\lambda)\}$
of $\Lambda$ which varies polynomially in $q$ and specializes for $q=1$ to the basis given by the
monomial symmetric functions. Up to some renormalization, the symmetric functions comprising this
family are known as the {\em Hall-Littlewood symmetric functions}. Remarkably, this family of bases also relates to another natural
basis of $\Lambda$: the {\em Schur functions}. This can be seen from a more refined analysis of
Zelevinsky's statistic relating it to {\em Kostka numbers} (\cite[Appendix A]{macdonald}). 
We only mention the final result of this discussion which leads to the following augmentation of Theorem \ref{theo:hallfq}:
\begin{enumerate}
	\item[(3)] The polynomial $g_{\mu \nu}^\lambda(t)$ has degree $\le n(\lambda) - n(\mu) - n(\nu)$ and
	the various coefficients of $t^{n(\lambda) - n(\mu) - n(\nu)}$ are the structure constants of
	$\Lambda$ with respect to the Schur basis $\{s_\lambda\}$.
\end{enumerate}
The results of Theorem \ref{theo:hallfq} have been known independently to Hall (1959) and Steinitz
(1901). 
We conclude with a proof of Zelevinsky's theorem: 
\begin{proof}
	We prove a slightly stronger version which will allow for an inductive argument:
	Let $\mu$, $\lambda$ be partitions and let $\alpha \sim \mu$ be a composition which is a
	permutation of $\mu$. Then we will show that
	\[
		b_{\lambda \mu} = \sum_A q^{d(A)}
	\]
	where $A$ ranges over all arrays of shape $\alpha$ and weight $\lambda$.\\

{\bf Step 1.} We reformulate the formula in terms of sequences of compositions: For $\alpha$,
$\beta$ compositions, we write $\beta \dashv \alpha$ if, for all $i$, $\alpha_i -1 \le \beta_i \le
\alpha_i$. For $\beta \dashv \alpha$, we define
\[
	d(\alpha,\beta) = | \{ (i,j) | \; \text{$\beta_i = \alpha_i$, $\beta_j = \alpha_j -1$, and
	$(j,\alpha_j) < (i,\alpha_i)$} \} |
\]
We then observe that, given $\alpha$ composition and $\lambda$ partition of length $\le s$,
we have a natural bijection between 
\[
	\left\{ \text{row-strict arrays $A$ of shape $\alpha$ and weight $\lambda$} \right\}
\]
and 
\[
	\left\{ \text{sequences $0 = \alpha^{(0)} \dashv \alpha^{(1)} \dashv \cdots \dashv
	\alpha^{(s)} = \alpha$ with $|\alpha^{(i)}| - |\alpha^{(i-1)}| = \lambda_i$} \right\}.
\]
Under this correspondence, we have
\[
	d(A) = \sum_{i \ge 0} d(\alpha^{(i)}, \alpha^{(i-1)}).
\]

{\bf Step 2.} Induction: Assume that 
\[
	u_{(1^{\lambda_1})} \cdots u_{(1^{\lambda_{s-1}})} = \sum_\nu \left(\sum_{0 = \beta^{(0)}
	\dashv \cdots \dashv \beta^{(s-1)} = \beta} \prod_{i \ge 1} q^{d(\beta^{(i)},
	\beta^{(i-1)})}\right) u_\nu
\]
where, for each $\nu$, $\beta$ is a fixed permutation of $\nu$. Then to show that  
\[
	u_{(1^{\lambda_1})} \cdots u_{(1^{\lambda_{s}})} = \sum_\mu \left(\sum_{0 = \alpha^{(0)}
	\dashv \cdots \dashv \alpha^{(s)} = \alpha} \prod_{i \ge 1} q^{d(\alpha^{(i)},
	\alpha^{(i-1)})}\right) u_\mu
\]
where, for each $\mu$, $\alpha$ is a fixed permutation of $\mu$, it suffices to show 
\[
	u_\nu u_{(1^r)} = \sum_\mu \left(\sum_{\beta \dashv \alpha, |\alpha| - |\beta| = r, \beta \sim
	\nu} q^{d(\alpha,\beta)}\right) u_\mu
\]
where, for each $\mu$, $\alpha$ is any fixed permutation of $\mu$.

{\bf Step 3.} In other words, we have to verify for the structure constant $g_{\nu (1^r)}^\mu$ of
$\Hall(R)$, the formula 
\[
	g_{\nu (1^r)}^\mu = \sum_{\beta} q^{d(\alpha,\beta)}
\] 
where $\alpha$ is is a fixed permutation of $\mu$ and $\beta$ runs through all permutations of $\nu$
such that $\beta \dashv \alpha$, and $|\alpha| - |\beta| = r$. Recall that, for a fixed $R$-module
$M$ of type $\mu$ the number $g_{\nu (1^r)}^\mu$ is the number of submodules $N \subset M$ such that
$N$ has type $(1^r)$ and $M/N$ has type $\nu$. In particular, we have $\m N = 0$, so that $N$ is a
$r$-dimensional $k$-subspace of the $k$-vector space given by the socle $S = \{x \in M | \m x = 0\}$
of $M$. We denote by $G_r(S)$ the set of all $r$-dimensional subspaces of $S$. For every choice of a
basis $\{v_i | i \in I\}$ of $S$ together with a total order of $I$, the set $G_r(S)$ is the
disjoint union of {\em Schubert cells} defined as follows: We have one Schubert cell $C_J$ for every
$r$-subset $J$ of $I$. The elements of $C_J$ have coordinates $(c_{ij} \in k | j \in J, i \in I
\setminus J, i >j)$ where the subspace of $S$ corresponding to a coordinate $(c_{ij})$ has basis
$\{v_j + \sum_{i} c_{ij} v_i\}_{j \in J}$.

Therefore, we have
\[
	|C_J| = q^{d(J)}
\]
where $d(J)$ is the number of pairs $(i,j)$ such that $j \in J$, $i \in J \setminus I$, $i >j$.
Suppose
\[
	M \cong \bigoplus_{i \in I} R/ \m^{\alpha_i}
\]
We order $I$ so that $j < i $ iff $(j, \alpha_j) < (i, \alpha_i)$. Then we have a bijective
correspondence between subsets $J \subset I$ and compositions $\beta \dashv \alpha$ (where $\beta_i
= \alpha_{i-1}$ iff $i \in J$). Under this correspondence, we have $d(J) = d(\alpha, \beta)$.
Further, for all $k$-subspaces $N \subset S \subset M$ which lie in a fixed Schubert cell $C_J$, the
quotient $M/N$ has the same type $\lambda \sim \beta$. Therefore, only those Schubert cells so that
$M/N$ has type $\nu$ contribute to the count and we obtain precisely the claimed formula.
\end{proof}

\newpage
\section{Hall algebras via groupoids}
\label{section:groupoids}

\subsection{Groupoids of flags}
\label{subsec:flags}

A {\em groupoid} is a category in which all morphisms are invertible. 

\begin{exa} 
	\begin{enumerate}
		\item Let $G$ be a group. Then we can define a groupoid $BG$ which has one object
			$\ast$ and $\Hom(\ast, \ast) = G$ where the composition of morphisms is
			given by the group law. 
		\item Let $X$ be a topological space. Then we can define the fundamental groupoid
			$\Pi(X)$ whose objects are the points of $X$ and a morphism between $x$ and
			$y$ is a homotopy class of continuous paths connecting $x$ to $y$.
		\item Let $\C$ be a category. Then we can form the {\em maximal groupoid}
			$\C^{\cong}$ in $\C$ by simply discarding all noninvertible morphisms in
			$\C$.
	\end{enumerate}
\end{exa}

We introduce a family of groupoids which will be of central relevance for this section. Let $\C$ be a
proto-abelian category and let $\S_n = \S_n(\C)$ denote the maximal groupoid in the category of
diagrams of the form
\begin{equation}\label{eq:sdiag}
		\xymatrix{ 0 \ar@{^{(}->}[r] & A_{0,1} \ar@{^{(}->}[r] \ar@{->>}[d] & A_{0,2}
		\ar@{^{(}->}[r] \ar@{->>}[d] & \dots & A_{0,n-1} \ar@{^{(}->}[r] \ar@{->>}[d]&
		A_{0,n} \ar@{->>}[d]\\
		&0 \ar@{^{(}->}[r] & A_{1,2} \ar@{^{(}->}[r] \ar@{->>}[d] &\dots & A_{1,n-1} \ar@{^{(}->}[r] \ar@{->>}[d]
		& A_{1,n} \ar@{->>}[d]\\
		&& 0 & \ddots &\vdots & \vdots &\\
		&&&& A_{n-2,n-1} \ar@{^{(}->}[r]\ar@{->>}[d] & A_{n-2,n}\ar@{->>}[d]\\
		&& & & 0 \ar@{^{(}->}[r] & A_{n-1,n}\ar@{->>}[d]\\
		&& & & & 0
	}
\end{equation}
in $\C$ where $0$ is a fixed zero object in $\C$ and all squares are required to be biCartesian. 
The crucial observation is that the various groupoids $\S_{\bullet}$ are related to one another: 
for every $0 \le k \le n$, we have a functor
\[
	\partial_k : \S_n \to \S_{n-1}
\]
obtained by omitting in the diagram \eqref{eq:sdiag} the objects in the $k$th row and $k$th column and forming the composite of the
remaining morphisms. Similarly, for every $0 \le k \le n$, we have functors
\[
	\sigma_k : \S_{n} \to \S_{n+1}
\]
given by replacing the $k$th row by two rows connected via identity maps and replacing the $k$th
column by two columns connected via identity maps.

Recall the definition of the {\em simplex category} $\Delta$. The objects are given by the standard
ordinals $[n] = \{0,1,\dots,n\}$, $n \ge 0$, and a morphism from $[m]$ to $[n]$ is a map
\[
\{0,1,\dots,m\} \to \{0,1,\dots,n\}
\]
of underlying sets which preserves $\le$. 
\begin{prop}
	The collection of groupoids $\S_{\bullet}$ naturally forms a simplicial groupoid, i.e., a
	functor
	\[
		\S_{\bullet}: \Delta^{\op} \lra \Grpd
	\]
	with values in the category of groupoids.
\end{prop}

\begin{rem}\label{rem:notation} For convenience reasons, we sometimes implicitly replace the simplex category by the
larger category of {\em all} finite nonempty linearly ordered sets. This is harmless since every
such a set is isomorphic to some standard ordinal $[n]$ via a unique isomorphism. For example, we
write $\S_{\{0,1,2\}} \to \S_{\{0,2\}}$ to refer to the face map $\partial_1: \S_2 \to \S_1$
leaving the natural inclusion $\{0,2\} \to \{0,1,2\}$ and the identification with $[1] \to [2]$
implicit.
\end{rem}

Given a diagram
\[
	\xymatrix{ \A \ar[r]^F & \C & \ar[l]_G \B}.
\]
of groupoids, we introduce the {\em $2$-pullback} to be the groupoid 
\[
	\A \times_{\C}^{(2)} \B 
\]
with objects given by triples $(a,b,\varphi)$ where $a \in \A$, $b \in \B$ and $\varphi: F(a)
\overset{\cong}{\to} G(b)$. A morphism $(a,b,\varphi) \to (a',b',\varphi')$ is given by a pair of
morphisms $f: a \to a'$, $g: b \to b'$ such that the diagram
\[
	\xymatrix{
		F(a) \ar[r]^{\varphi}\ar[d]^{F(f)} & G(b)\ar[d]^{G(g)}\\ 
		F(a') \ar[r]^{\varphi'} & G(b') }
\]
commutes. We call a diagram
\[
	\xymatrix{
		\X \ar[r]^{F'}\ar[d]_{G'} & \B \ar[d]^{G} \\
		\A \ar[r]^F & \C, \ultwocell<\omit>{\eta}
	} 
\]
where $\eta: F \circ G' \Rightarrow G \circ F'$ is a natural isomorphism, a {\em $2$-pullback diagram} if the functor 
\[
	\X \to \A \times^{(2)}_{\C} \B, x \mapsto (G'(x),F'(x), \eta(x))
\]
is an equivalence. The property which makes $2$-pullbacks
better-behaved than ordinary ones is that an equivalence of diagrams of groupoids induces an equivalence of their $2$-pullbacks.

It turns out that in some cases, ordinary pullback squares of categories are actually $2$-pullback
squares: A functor $F: \A \to \B$ is called {\em isofibration} if, for every object $a \in \A$ and every
isomorphism $\varphi: F(a) \overset{\cong}{\to} b$ in $\B$, there exists an isomorphism
$\widetilde{\varphi}: a \to a'$ in $\A$ such that $F(\widetilde{\varphi}) = \varphi$. 

\begin{prop} Let 
	\[
	\xymatrix{
		\X \ar[r]\ar[d] & \B \ar[d]\\
	\A \ar[r]^F & \C }
	\]
	be a commutative square of categories and assume that $F$ is an isofibration. Then the
	square is a $2$-pullback square.
\end{prop}

We come to a key property of the simplicial groupoid $\S_{\bullet}$: it satisfies the {\em $2$-Segal
conditions} which we now introduce. Let $\X_{\bullet}$ be a simplicial object in $\Grpd$.
\begin{enumerate}
	\item Consider a planar $n+1$-gon $P$ with vertices labelled cyclically by the set $\{0,1,\dots,n\}$. Let $i <
		j$ be the vertices of a diagonal of $P$ which subdivides the polygon into two polygons with labels
		$\{0,1,\dots,i,j,j+1,\dots,n\}$ and $\{i,i+1,\dots,j\}$. We obtain a corresponding
		commutative square of groupoids
		\begin{equation}\label{eq:segal}
			\xymatrix{ \X_{\{0,1,\dots,n\}} \ar[d]\ar[r] & \X_{\{0,1,\dots,i,j,\dots,n\}}\ar[d]\\
			\X_{\{i,i+1,\dots,j\}} \ar[r] & \X_{\{i,j\}}. }
		\end{equation}
	\item For every $0 \le i < n$, there is commutative square
		\begin{equation}\label{eq:unital}
			\xymatrix{ \X_{\{0,1,\dots,n-1\}} \ar[d]^{\sigma_i}\ar[r] & \X_{\{i\}}\ar[d]\\
			\X_{\{0,1,\dots,n\}} \ar[r] & \X_{\{i,i+1\}} }
		\end{equation}
		where $\sigma_i$ denotes the $i$th degeneracy map.
\end{enumerate}

The following definition was introduced in \cite{dk-segal} (an equivalent condition was independently discovered and studied in \cite{kock-decomposition}).

\begin{defi}\label{defi:2segal} A simplicial object $\X_{\bullet}$ in $\Grpd$ is called {\em $2$-Segal} if, 
	\begin{enumerate}
		\item for every polygonal subdivision, the corresponding square $\eqref{eq:segal}$
			is a $2$-pullback square,
		\item all squares \eqref{eq:unital} are $2$-pullback squares.
	\end{enumerate}
\end{defi}

\begin{rem}\label{rem:segal} To explain the terminology, recall that a simplicial set $K$ is called {\em Segal} 
	if, for every $0 < k < n$, the square
	\[
		\xymatrix{
			K_{\{0,1,\dots,n\}} \ar[r] \ar[d] & K_{\{0,1,\dots,k\}} \ar[d]\\
			K_{\{k,k+1,\dots,n\}} \ar[r] & K_{\{k\}}}
	\]
	is a pullback square. We can interpret the datum $0 < k < n$ as a subdivision of the
	interval $[0,n]$ into intervals $[0,k]$ and $[k,n]$ so that the $2$-Segal condition is a
	$2$-dimensional analog of the Segal condition.
\end{rem}

\begin{theo} The simplicial groupoid $\S_{\bullet}$ of flags in a proto-abelian category $\C$ is $2$-Segal.
\end{theo}
\begin{proof} We note that all maps in the diagram \eqref{eq:segal} are isofibrations so that it
	suffices to check that the squares are ordinary pullback squares. The functor
	\[
		\S_{\{0,1,\dots,n\}} \lra \S_{\{0,1,\dots,i,j,\dots,n\}}\times_{\S_{\{i,j\}}} \S_{\{i,i+1,\dots,j\}}
	\]
	is a forgetful functor which forgets those objects in the diagram \eqref{eq:sdiag} whose
	indices $(x,y)$ correspond to diagonals of $P_n$ which cross the diagonal $(i,j)$. But
	these objects can be filled back in by forming pullbacks or pushouts, using the axioms of a
	proto-abelian category.
	We leave the verification of \eqref{eq:unital} to the reader.
\end{proof}

Below, we will interpret the lowest $2$-Segal conditions \eqref{eq:segal} and \eqref{eq:unital} as associativity and
unitality, respectively, of the {\em abstract Hall algebra} of a proto-abelian category $\C$.

\subsection{Spans of groupoids and the abstract Hall algebra}
\label{subsec:spans}

We introduce the category $\Span(\Grpd)$ of {\em spans in groupoids}. The objects are given by groupoids. 
The set of morphisms between groupoids $\A,\B$, is defined to be
\[
	\Hom(\A,\B) = \left\{ \vcenter{\xymatrix{ & \ar[dl] \X \ar[dr] &\\ \A & & \B}} \right\}/\sim
\]
where two {\em spans} $\A \leftarrow \X \rightarrow \B$ and $\A \leftarrow \X' \rightarrow \B$ are
considered equivalent if there exists a diagram
\[
	\xymatrix{
		& \ar[dl] \ddltwocell<\omit> \X \ar[dd]|F \ar[dr] & \\
		\A & & \B\\
	&\X'\ar[ul] \ar[ur] & \uultwocell<\omit> }
\]
where $F$ is an equivalence. The composition of
morphisms is given by forming $2$-pullbacks: Given morphisms $f:\A \to \B$ and $g:\B \to \C$ in
$\Span(\Grpd)$, we
represent them by spans $\A \leftarrow \X \rightarrow \B$ and $\B \leftarrow \Y \rightarrow \C$ and
form the diagram
\[
	\xymatrix{
		& &\ar[dl] \Z\ar[dr] & &\\ 
		& \X \ar[dr] \ar[dl]& & \Y \ar[dl] \ar[dr]& \\
	\A & & \B\uutwocell<\omit> & & \C }
\]
where the square is a $2$-pullback square. We then define the composite $g \circ f$ to be the morphism from $\A$
to $\C$ represented by the span
\[\xymatrix{ & \ar[dl] \Z \ar[dr] &\\ \A & & \B.}
\]
It follows from the invariance properties of $2$-pullbacks that this operation is
well-defined. 

The category $\Span(\Grpd)$ has a natural monoidal structure: The tensor
product is defined on objects via $\A \otimes \B = \A \times \B$ and on morphisms via
\[
	\vcenter{\xymatrix{ & \ar[dl] \X \ar[dr] &\\ \A & & \B}} \otimes \vcenter{\xymatrix{ & \ar[dl] \X' \ar[dr] &\\
	\A' & & \B'}} = \vcenter{\xymatrix{ & \ar[dl] \X \times \X' \ar[dr] &\\ \A \times \A' & & \B \times
	\B'.}}
\]

\begin{theo}\label{theo:abshall}
	Let $\C$ be a proto-abelian category and let $\S_{\bullet}$ be the corresponding groupoids
	of flags in $\C$. 
	The morphisms in $\Span(\Grpd)$ represented by the spans
		\[
			\mu: \vcenter{\xymatrix{ & \ar[dl]_G \S_{\{0,1,2\}} \ar[dr]^F & \\
			\S_{\{0,1\}} \times \S_{\{1,2\}} & & \S_{\{0,2\}} }}
		\]
	and 
		\[
			e: \vcenter{\xymatrix{ & \ar[dl]_{\id} \S_0 \ar[dr]^{0 \mapsto 0}
		&\\ \S_0 & & \S_1}}
		\]
	make $\S_1$ an algebra object in $\Span(\Grpd)$. We call $(\S_1,\mu,e)$ the
	{\em abstract Hall algebra} of $\C$.
\end{theo}
\begin{proof}
	We have to verify $\mu \circ (\mu \otimes \id) = \mu \circ (\id \otimes \mu)$. To
	compute the left-hand side, we consider the commutative diagram 
	\[
		\xymatrix{
			\S_{\{0,1,2,3\}} \ar[d] \ar[r]& \S_{\{0,2,3\}} \ar[r]^F\ar[d]^G &
			\S_{\{0,3\}}\\
			\S_{\{0,1,2\}} \times \S_{\{2,3\}} \ar[r]^{F \times \id} \ar[d]^{G \times
			\id}& \S_{\{0,2\}} \times \S_{\{2,3\}} & \\
			\S_{\{0,1\}} \times \S_{\{1,2\}} \times \S_{\{2,3\}} & & }.
	\]
	We claim that the left-hand square is a pullback square: There is a sequence of natural
	functors
	\[
		\S_{\{0,1,2,3\}} \lra (\S_{\{0,1,2\}} \times \S_{\{2,3\}})
		\times^{(2)}_{\S_{\{0,2\}} \times \S_{\{2,3\}}} \S_{\{0,2,3\}}
		\lra \S_{\{0,1,2\}} \times^{(2)}_{\S_{\{0,2\}}} \S_{\{0,2,3\}}
	\]
	where the composite is an equivalence by the $2$-Segal condition \eqref{eq:segal}
	corresponding to the triangulation $\{0,1,2\}$,$\{0,2,3\}$ of the square, and the second
	functor is an equivalence by direct verification. It follows that the first functor is an
	equivalence which shows the claim. Therefore, we have 
	\[
		\mu \circ (\mu \otimes \id) = \vcenter{\xymatrix{ & \S_{\{0,1,2,3\}} \ar[dl]
		\ar[dr] &\\ \S_{\{0,3\}} \times
		\S_{\{0,1\}} \times \S_{\{1,2\}}& & \S_{\{2,3\}}.}}
	\]
	Similarly, using the $2$-Segal condition for the second triangulation of the square, we obtain
	that 
	\[
		\mu \circ (\id \otimes \mu) = \vcenter{\xymatrix{ & \S_3 \ar[dl]_R \ar[dr]^T &\\ \S_1 \times
		\S_1 \times \S_1 & & \S_1.}}
	\]
	so that we deduce associativity. Unitality follows similarly from the lowest two conditions
	\eqref{eq:unital}.
\end{proof}

\subsection{Groupoid functions}
\label{subsec:groupfunc}

We explain how, assuming $\C$ to be finitary, we can obtain the usual Hall algebra from the abstract
one by passing to groupoid functions.

Let $K$, $S$ be finite sets. We have 
\[
	|K \amalg S|  = |K|+|S|
\]
and 
\[
	|K \times S|  = |K| |S|
\]
so that the categorical operations $\amalg$ and $\times$ yield upon application of $|-|$ the
numerical operations of addition and multiplication. One may wonder if there is a categorical
analog of division (or subtraction). 

\begin{exa} 
	Consider the set $K = \{1,2,3,4\}$ equipped with the action of the cyclic group $C_2 =
	<\tau>$ of order $2$ where $\tau$ acts via the permutation $(14)(23)$. Then we can form the
	orbit set $K/C_2$ and
	have
	\[
		| K / C_2 | = 2 = |K| / |C_2|
	\]
	so that categorical construction of forming the quotient yields division by $2$ upon
	application of $|-|$. However, this interpretation fails as soon as the group action has
	nontrivial stabilizers: Letting the group $C_2$ act on the set $S  =\{1,2,3\}$ via the
	permutation $(13)$, we obtain a quotient $S / C_2$ of cardinality $2 \ne \frac{3}{2}$.
\end{exa}

We will now define a notion of cardinality for groupoids which solves the issue of the example. From
now on we use the notation $\pi_0(\A)$ for the set of isomorphism classes of objects in $\A$.
A groupoid $\A$ is called {\em finite} if
\begin{enumerate}
	\item the set $\pi_0(\A)$ of isomorphism classes of objects is finite,
	\item for every object $a \in \A$, the set of automorphisms of $a$ is finite.
\end{enumerate}
Given a finite groupoid $\A$, we introduce the {\em groupoid cardinality}
\[
	| \A | = \sum_{[a] \in \pi_0(\A)} \frac{1}{|\Aut(a)|}.
\]

\begin{rem} It is immediate from the definition, that groupoid cardinality is invariant under
	equivalences of finite groupoids.
\end{rem}

\begin{exa}
	Any finite set $K$ can be interpreted as a {\em discrete} groupoid with $K$ as
	its set of objects and morphisms given by identity morphisms only. The
	groupoid cardinality of the discrete groupoid associated with $K$ agrees
	with the cardinality of the set $K$.
\end{exa}

Let $K$ be a finite set equipped with a right action of a finite group $G$. We define
the {\em action groupoid $K // G$} to have $K$ as its set of objects and morphisms
between two elements $k$ and $k'$ given by elements $g \in G$ such that $k.g
= k'$. 

\begin{prop}
	We have 
	\[
		| K // G | = |K|/|G|.
	\]
\end{prop}
\begin{proof}
	The set of isomorphism classes $\pi_0(K // G)$ can be identified with the set of orbits of the
	action of $G$ on $K$. The automorphism group of an object $k$ of $K // G$ coincides with the
	stabilizer group $G_k = \{g \in G | k.g = k\} \subset G$. The orbit of an element $k$ under
	$G$ can be identified with the quotient set $G / G_k$ of cardinality $|G|/|G_k|$. We compute
	\begin{align*}
		| K // G | & = \sum_{[k] \in \pi_0(K // G)} \frac{1}{|G_k|}\\
		& = \frac{1}{|G|} \sum_{[k] \in \pi_0(K // G)} \frac{|G|}{|G_k|}\\
		& = \frac{|K|}{|G|} 
	\end{align*}
	where the last equality follows since the disjoint union of the orbits yields the set $K$.
\end{proof}

Given a set $K$ and a function $\varphi: K \to \QQ$ with finite support, we can introduce the integral
\[
	\int_K \varphi = \sum_{k \in K} \varphi(k).
\]
If $K$ is finite, then we have $\int_K \one = |K|$ where $\one$ denotes the constant function on $K$ with value $1$.  We
give a generalization to groupoids. 

Given a groupoid $\A$, we define $\F(\A)$ to be the $\QQ$-vector
space of functions $\varphi: \ob \A \to \QQ$ which are 
\begin{enumerate}
	\item constant on isomorphism classes, 
	\item nonzero on only finitely many isomorphism classes.
\end{enumerate}

We call $\A$ {\em locally finite} if every connected component $\A(a)$ is finite.
Given a locally finite groupoid $\A$ and $\varphi \in \F(\A)$, we define the {\em groupoid integral}
\[
	\int_{\A} \varphi = \sum_{[a] \in \pi_0(\A)} \frac{\varphi(a)}{|\Aut(a)|}.
\]
Note that, if $\A$ is finite, then we have $\int_{\A} \one = |\A|$.

We further introduce a relative version of the groupoid integral given by integration along the
fibers: A functor $F: \A \to \B$ of groupoids is called 
\begin{enumerate}
	\item {\em finite} if every $2$-fiber of $F$ is finite,
	\item {\em locally finite} if, for every $a \in \A$, the restriction of $F$ to $\A(a)$ is
		finite, 
	\item {\em $\pi_0$-finite} if the induced map of sets $\pi_0(\A) \to \pi_0(\B)$ has finite fibers.
\end{enumerate}

Given a locally finite functor $F: \A \to \B$ and a function $\varphi \in \F(\A)$, we define the
{\em pushforward} $F_! \varphi \in \F(\B)$ via
\[
	F_!\varphi(b) := \int_{\A_b} \varphi_{|\A_b}
\]
where $\A_b$ is the $2$-fiber of $F$ at $b$ and $\varphi_{|\A_b}$ denotes the pullback of $\varphi$ along the
natural functor $\A_b \to \A$.
We obtain a $\QQ$-linear map
\[
	F_!: \F(\A) \lra \F(\B).
\]

\begin{exa}
	For a locally finite groupoid $\A$ the constant functor $F: \A \to \{\ast\}$ is locally finite and we have
	$F_! \varphi (\ast) = \int_{\A} \varphi$.
\end{exa}

Given a $\pi_0$-finite functor $F: \A \to \B$ and a function $\varphi \in \F(\B)$, we define the
{\em pullback} $F^*\varphi \in \F(\A)$ via
\[
	F^*\varphi(a) := \varphi(F(a)).
\]
We obtain a $\QQ$-linear map
\[
	F^*: \F(\B) \lra \F(\A).
\]

The central properties of the pullback and pushforward operations are captured in the following Proposition.

\begin{prop}\label{prop:basefun}
	\begin{enumerate}
		\item {\em Functoriality.} 
			\begin{enumerate}
				\item Let $F: \A \to \B$ and $G: \B \to \C$ be $\pi_0$-finite functors
			of groupoids. Then the composite $G \circ F$ is $\pi_0$-finite and we have
			\[
				(G \circ F)^* = F^* \circ G^*.
			\]
				\item Let $F: \A \to \B$ and $G: \B \to \C$ be locally finite functors
			of groupoids. Then the composite $G \circ F$ is locally finite and we have
			\[
				(G \circ F)_! = G_! \circ F_!.
			\]
			\end{enumerate}
		\item {\em Base change.} Let 
			\begin{equation}\label{eq:2pulkan}
					\xymatrix{ \X \ar[d]_{G'}\ar[r]^{F'} & \B\ar[d]^G \\
					\A \ar[r]^F & \C \ultwocell<\omit>}
			\end{equation}
			be a $2$-pullback square with $F$ locally finite and $G$ $\pi_0$-finite.
			Then $F'$ is locally finite, $G'$ is $\pi_0$-finite, and we have
			\[
				(F')_! \circ (G')^* = G^* \circ F_!.
			\]
	\end{enumerate}
\end{prop}

We define a subcategory of $\Spanf(\Grpd) \subset \Span(\Grpd)$ with the same objects but
morphisms given by spans
\[\xymatrix{ & \ar[dl]_L \X \ar[dr]^R &\\ \A & & \B.}
\]
such that $R$ locally finite and $L$ $\pi_0$-finite. The composition of such spans is well-defined by 
Proposition \ref{prop:basefun}. 

\begin{prop}
	Let $\C$ be a finitary proto-abelian category. 
	\begin{enumerate}
		\item The abstract Hall algebra $(\S_1,\mu,e)$ defines an algebra object in the
			monoidal category $\Spanf(\Grpd) \subset \Span(\Grpd)$.
		\item The association 
			\begin{align*}
				\A & \mapsto \quad \F(\A)\\
				\vcenter{\xymatrix{ & \ar[dl]_L \X \ar[dr]^R &\\ \A & & \B}} & \mapsto \quad \F(\A) \overset{R_! \circ
				L^*}{\lra} \F(\B)
			\end{align*}
			defines a monoidal functor
			\[
				\F: \Spanf(\Grpd) \to \Vect_{\QQ}.
			\]
			The resulting algebra object $\F(\S_1,\mu,e)$ in $\Vect_{\QQ}$ is isomorphic to the (opposite) Hall algebra of $\C$.
	\end{enumerate}
\end{prop}

\subsection{Green's theorem}

Green's theorem (cf. \cite{green}) states that, under certain assumptions on an abelian category $\C$, we can introduce
a {\em coproduct} on the Hall algebra $\Hall(\C)$ making it a bialgebra up to a certain twist.
In this section, we will use the abstract Hall algebra introduced in Section \ref{subsec:spans} to provide a
proof of this statement. 

Let $\C$ be a proto-abelian category and $\S_{\bullet}$ the corresponding groupoids of flags.
Instead of considering the span 
\begin{equation}\label{eq:mult}
		\mu: \vcenter{\xymatrix{ & \ar[dl]_G \S_2 \ar[dr]^F &\\ \S_1 \times
		\S_1 & & \S_1}}
\end{equation}
which represents the multiplication on the abstract Hall algebra, we may form the {\em reverse} span
\begin{equation}\label{eq:comult}
	\Delta: \vcenter{\xymatrix{ & \ar[dl]_F \S_2 \ar[dr]^{G} &\\ \S_1  & & \S_1 \times \S_1}}
\end{equation}
Also taking into account the reverse $c$ of the unit morphism $e$, it is immediate that $(\S_1,
\Delta, c)$ forms a {\em coalgebra object} in $\Span(\Grpd)$. The
proof consists of reading all span diagrams involved in the proof of Theorem \ref{theo:abshall} in
reverse direction.
	
Given an associative $k$-algebra $A$, equipped with a coproduct $\Delta: A \to A
\otimes A$, we may ask if multiplication and comultiplication are compatible in the sense
\begin{equation}\label{eq:compat}
		\Delta(a b) = \Delta(a) \Delta(b).
\end{equation}
In other words, introducing on $A \otimes A$ the algebra structure $(a \otimes b) (a' \otimes b) = a
a' \otimes b b'$, we ask if the coproduct $\Delta$ is a homomorphism of algebras.

\begin{exa} Let $G$ be a group and let 
	\[
		k[G] = \bigoplus_{g \in G} k g
	\]
	denote the group algebra over the field $k$.
	The $k$-linear extension of the formula $\Delta(g) = g \otimes g$ defines a coproduct
	\[
		\Delta: k[G] \to k[G] \otimes k[G] 
	\]
	on $k[G]$. It is immediate to verify \eqref{eq:compat}.
\end{exa}

We address the analogous compatibility question for the abstract Hall algebra. 

\subsubsection{Squares, frames, and crosses}

To analyze whether or not the equation \eqref{eq:compat} holds for the abstract Hall algebra, we
explicitly compute both sides using \eqref{eq:mult} and \eqref{eq:comult}. 
The left-hand side is given by the composite
\[
	\xymatrix{ \cross \ar[r]\ar[d] & \S_2 \ar[r]^G\ar[d]_F & \S_1 \times \S_1\\
		\S_2 \ar[r]^F\ar[d]_G & \S_1 & \\
		\S_1 \times \S_1. &&}
\]
which yields
\begin{equation}\label{eq:lhs}
		\xymatrix{ & \ar[dl]_L \cross \ar[dr]^R &\\ \S_1 \times
		\S_1 & & \S_1 \times \S_1}
\end{equation}
where $\cross$ denotes the groupoid of {\em exact crosses} in $\C$: diagrams
\[
		\xymatrix{ & B \ar[d] & \\
			A' \ar[r] & B' \ar[d]\ar[r] & C'\\
		& B'' & }
\]
consisting of two exact sequences $\C$ with common middle term. The functor $L$ associates to such a
cross the pair of objects $(B, B'')$, the functor $R$ assigns the pair $(A', C')$. Note that we can
compute the $2$-pullback $\cross$ as an ordinary pullback since the functor $F$ is an isofibration.
The right-hand side of \eqref{eq:compat} is given by 
\[
	\xymatrix{ \frame \ar[r]\ar[d] & \S_2 \times \S_2 \ar[r]^{F \times F} \ar[d]_{P} & \S_1 \times \S_1\\
		\S_2 \times \S_2 \ar[r]^-{G \times G}\ar[d]_{F \times F} & \S_1 \times \S_1 \times \S_1 \times \S_1 & \\
		\S_1 \times \S_1 & & }
\]
where the functor $P$ assigns to a pair $(A \to A' \to A'', C \to C' \to C'')$ of short exact
sequences the $4$-tupel $(A,C,A'',C'')$ of objects in $\S_1$. The composite is represented by the span
\begin{equation}\label{eq:rhs}
		\xymatrix{ & \ar[dl]_M \frame \ar[dr]^N &\\ \S_1 \times
		\S_1 & & \S_1 \times \S_1}
\end{equation}
where $\frame$ denotes the groupoid of {\em exact frames} in $\C$: diagrams
\[
	\xymatrix{A \ar[r]\ar[d] & B \ar[r] & C\ar[d] \\
		A' \ar[d] &  & C'\ar[d] \\
	A'' \ar[r] & B'' \ar[r] & C''}
\]
where the two complete rows and columns are short exact sequences. The functor $M$ assigns to such a
frame the pair $(B,B'')$ while the functor $N$ assigns the pair $(A',C')$. To compare the groupoids $\cross$
and $\frame$, we introduce another groupoid: the groupoid $\square$ of {\em exact $3$-by-$3$ squares} given by
commutative diagrams in $\C$ of the form
\[
	\xymatrix{A \ar[r]\ar[d] & B \ar[r] \ar[d] & C\ar[d] \\
		A' \ar[r]\ar[d] & B' \ar[d]\ar[r] & C'\ar[d]\\
	A'' \ar[r] & B''\ar[r] & C'' }
\]
where all rows and columns are required to be short exact sequences. We obtain a commutative diagram
\[
	\xymatrix{
		& \ar[dl]\cross \ar[dr] & \\
		\S_1 \times \S_1 &\ar[l]\ar[u] \square\ar[d] \ar[r]& \S_1 \times \S_1\\
	& \ar[ul]\ar[ur] \frame & }
\]
of groupoids. In what follows, we will assume that the category $\C$ is {\em abelian}. 

\begin{lem} \label{lem:eq1} The forgetful functor $F: \square \to \cross$ is an equivalence of groupoids.
\end{lem}
\begin{proof} Fully faithfulness is clear since the missing objects needed to complete a cross to
	square are characterized by universal properties.
	To show essential surjectivity, one shows that one may always complete a cross to a square by
	first forming the top-left corner as a pullback, then the bottom-right corner as a pushout.
	The remaining objects are obtained by forming kernels/cokernels. A few applications of the
	snake lemma are needed to verify that the resulting $3$-by-$3$ square is exact.
\end{proof}

Using Lemma \ref{lem:eq1}, we may replace \eqref{eq:lhs} by the span
\[
		\xymatrix{ & \ar[dl] \square \ar[dr] &\\ \S_1 \times
		\S_1 & & \S_1 \times \S_1}
\]
which represents the same morphism in $\Span(\Grpd)$.
We obtain a diagram 
\[
	\xymatrix{
		& \ar[dl]\square \ar[dd]\ar[dr] & \\
		\S_1 \times \S_1 && \S_1 \times \S_1\\
	& \ar[ul]\ar[ur] \frame & }
\]
of groupoids. If the forgetful functor $\pi: \square \to \frame$ were an equivalence, this would imply the
compatibility of the multiplication and comultiplication for the abstract Hall algebra. As it turns
out, this is {\em not} the case -- the situation is more subtle. We analyze the discrepancy of
the functor $\pi$ from being an equivalence by calculating its $2$-fibers. Note that, $\pi$ being an
isofibration, we may calculate the $2$-fibers as ordinary fibers of $\pi$. Suppose the frame $f$
is given by the diagram
\[
	\xymatrix{A \ar[r]\ar[d] & B \ar[r] & C\ar[d] \\
		A' \ar[d] &  & C'\ar[d] \\
		A'' \ar[r] & B'' \ar[r] & C''.}
\]
The fiber $\square_f$ is the groupoid of diagrams of the form
\[
	\xymatrix{A \ar[r]\ar[d] & B\ar[d] \ar[r] & C\ar[d] \\
		A' \ar[d]\ar[r] & Y \ar[r]\ar[d]& C'\ar[d] \\
		A'' \ar[r] & B'' \ar[r] & C''.}
\]
with morphisms inducing the identity on the fixed outer frame $f$. Note that associated to the frame
$f$, there is a long exact sequence
\begin{equation}\label{eq:yonedasum}
		\xi_f: 0 \lra A \lra A' \amalg_{A} B \lra B'' \times_{C''} C' \lra C'' \lra 0
\end{equation}
given as the Baer sum of the two outer long exact sequences of the frame. 

\begin{defi}\label{defi:triv} Let 
\[
	\xi: 0 \lra Q \lra R  \lra S \lra T \lra 0
\]
be an exact sequence in an abelian category $\A$. We introduce a corresponding groupoid $\Triv(\xi)$
of commutative diagrams of the form
\begin{equation}\label{eq:triv}
		\xymatrix{ Q \ar[r]  & R \ar[rr]\ar@{^{(}->}[dr] &  & S \ar[r] & T \\
		& & Y \ar@{->>}[ur]& & }
\end{equation}
such that the natural maps $Y/Q \to S$ and $Y/R \to T$ are isomorphisms. The morphisms in
$\Triv(\xi)$ are given by isomorphisms of diagrams which induce the identity on $\xi$.
\end{defi}

\begin{lem}\label{lem:2fib} There is an equivalence of groupoids
	\[
		\square_f \lra \Triv(\xi_f).
	\]
\end{lem}
\begin{proof} It is clear that, given an exact $3$-by-$3$ square, we obtain a canonical diagram of
	the form \eqref{eq:triv}. Various applications of the snake lemma show that this 
	association yields a well-defined functor which is an equivalence.
\end{proof}

It turns out that the groupoid $\Triv(\xi)$ has a beautiful interpretation in the context of
Yoneda's theory of extensions. We review some aspects of this theory.

\subsubsection{Yoneda's theory of extensions}
\label{sec:yoneda}

Let $\C$ be an abelian category and let $A,B$ be objects in $\C$. An $n$-extension of $B$ by $A$ is
an exact sequence
\[
	\xi: 0 \lra A \lra X_{n-1} \lra X_{n-2} \lra \dots \lra X_0 \lra B \lra 0
\]
in $\C$. Given another extension 
\[
	\xi': 0 \lra A \lra X_{n-1}' \lra X_{n-2}' \lra \dots \lra X_0' \lra B \lra 0
\]
we say that $\xi$ and $\xi'$ are {\em Yoneda equivalent} if there exists a commutative diagram
\begin{equation}\label{eq:yonedaequiv}
		\xymatrix{& & X_{n-1} \ar[r] & X_{n-2} \ar[r] & \dots \ar[r] & X_0 \ar[dr] & & \\
		0 \ar[r] & A \ar[r] \ar[ur] \ar[dr] & Y_{n-1} \ar[r]\ar[u] \ar[d] & Y_{n-2} \ar[u] \ar[d]\ar[r] & \dots
		\ar[r] & Y_0\ar[u] \ar[d] \ar[r] & B \ar[r] & 0\\
		& & X_{n-1}' \ar[r] & X_{n-2}' \ar[r] & \dots \ar[r] & X_0' \ar[ur] & &}
	\end{equation}
with exact rows. Yoneda equivalence in fact defines an equivalence relation and 
we denote by $\Ext^n(B,A)$ the set of equivalence classes of $n$-extensions of $B$ by $A$. The
association 
\[
	(B,A) \mapsto \Ext^n(B,A)
\]
is functorial in both arguments: Given a morphism $f: A \to A'$, and an extension $\xi$ of $B$ by $A$ as
above, we obtain an $n$-extension
\[
	f_*(\xi): 0 \lra A' \lra X_{n-1} \amalg_{A} A' \lra X_{n-2} \lra \dots \lra X_0 \lra B \lra 0
\]
of $B$ by $A'$ called the {\em Yoneda pushout} of $\xi$ along $f$.
Dually, given a morphism $g: B' \to B$, we obtain an $n$-extension
\[
	g^*(\xi): 0 \lra A \lra X_{n-1} \lra X_{n-2} \lra \dots \lra X_0 \times_{B} B' \lra B' \lra 0
\]
of $B'$ by $A$ called the {\em Yoneda pullback} of $\xi$ along $g$.
Further, the set $\Ext^n(B,A)$ is equipped with an addition law: Given two extensions
$\xi$ and $\xi'$ as above, we first define
\[
	\xi \oplus \xi': 0 \lra A \oplus A \lra X_{n-1} \oplus X_{n-1}'  \lra \dots \lra
	\lra X_0 \oplus X_0' \lra B \oplus B \lra 0
\]
and then the {\em Baer sum}
\[
	\xi + \xi' = (\Delta_B)^*((\nabla_A)_*(\xi \oplus \xi')) 
\]
where $\Delta_B: B \to B \oplus B$ and $\nabla_A: A \oplus A \to A$ denote diagonal and codiagonal,
respectively. The Baer sum defines an abelian group structure on the set $\Ext^n(B,A)$.  

\begin{exa} In the case $n=2$, we obtain 
	\[
		\xi + \xi':  0 \lra A \lra X_{1} \amalg_A X_{1}'  \lra X_0 \times_B X_0' \lra B \lra
		0
	\]
	which is the operation used to produce the exact sequence \eqref{eq:yonedasum}.
\end{exa}

Due to the complicated nature of the equivalence relation \eqref{eq:yonedaequiv}, it is hard to
decide whether two given extensions $\xi$ and $\xi'$ are equivalent, let alone to compute the group
$\Ext^n(B,A)$. The situation simplifies greatly if the abelian category $\A$ has enough projectives
which we assume from now on (what follows can alternatively be done via dual arguments assuming that $\A$ has enough
injectives).

Given an extension 
\[
	\xi: 0 \lra A \lra X_{n-1} \lra X_{n-2} \lra \dots \lra X_0 \lra B \lra 0
\]
we choose a projective resolution
\[
	\dots \lra P_1 \lra P_0 \lra B
\]
of $B$. Using the projectivity of the objects $P_i$ we can construct a lift of the identity 
map $B \to B$ to a morphism of complexes
\begin{equation}\label{eq:lift}
		\xymatrix{\dots \ar[r] & P_{n+1} \ar[r]^d\ar[d] & P_n \ar[r]^d\ar[d]^{f_n} & \dots \ar[r] &
		P_1 \ar[r]^d\ar[d]^{f_1} & P_0 \ar[r]\ar[d]^{f_0} & B \ar[d]^{\id} \\
	\dots \ar[r] &	0 \ar[r] & A \ar[r] & \dots \ar[r] & X_1 \ar[r] & X_0 \ar[r] & B }
\end{equation}
so that we obtain an element $f_n \in \Hom(P_n, A)$ satisfying $f_n \circ d = 0$. This element
therefore defines a $n$-cocycle in the complex $\Hom(P_{\bullet}, A)$. We have: 

\begin{lem} The association $\xi \mapsto f_n$ defines an isomorphism of abelian groups
	\[
		\Ext^n(B,A) \overset{\cong}{\lra} H^n(\Hom(P_{\bullet}, A)).
	\]
\end{lem}

Since $\Ext^n(B,A)$ forms an abelian group, there exists a distinguished equivalence class of
$n$-extensions which are {\em trivial} in the sense that they represent the neutral element. We will
now provide a detailed study of trivial extensions in the cases $n=1$ and $n=2$. 

Let 
\[
	\xi: 0 \lra A \overset{i}{\lra} X_0 \lra B \lra 0
\]
be an extension of $B$ by $A$. A {\em splitting} of $\xi$ is a morphism $s: X_0 \to A$ such that $si =
\id_A$. We denote by $\Split(\xi)$ the set of splittings of $\xi$ which we will now analyze
explicitly: Fix a projective resolution $P_{\bullet}$ of $B$, and a lift of $\id:
B \to B$ as in \eqref{eq:lift}. In particular, we obtain a corresponding cocycle $f_1 \in
\Hom(P_1,A)$. Consider the differential 
\[
	d: \Hom(P_0, A) \to \Hom(P_1, A).
\]

\begin{prop} There is a canonical bijection of sets
	\[
		d^{-1}(f_1) \overset{\cong}{\lra} \Split(\xi).
	\]
	In particular,
	\begin{enumerate}
		\item A splitting exists if and only if the class of $\xi$ in $\Ext^1(B,A)$ is trivial.
		\item If the class of $\xi$ is trivial, then the set of different splittings admits
			a simply transitive action of the abelian group $\Hom(B,A)$.
	\end{enumerate}
\end{prop}
\begin{proof}
	The diagram \eqref{eq:lift} induces a commutative diagram
	\[
		\xymatrix{0 \ar[r] & P_1/\im P_2 \ar[r]\ar[d]^{\overline{f_1}} & P_0 \ar[r]\ar[d]^{f_0} & B \ar[d]^{\id}\ar[r] & 0 \\
		0 \ar[r] & A \ar[r] &  X_0 \ar[r] & B \ar[r] & 0}
	\]
	with exact rows. Forming the pushout of the top-left square, we obtain a commutative diagram
	\[
		\xymatrix{0 \ar[r] & A \ar[r]\ar[d]^{\id} & A \amalg_{P_1} P_0 \ar[r] \ar[d]^g & B \ar[d]^{\id}\ar[r] & 0 \\
		0 \ar[r] & A \ar[r] &  X_0 \ar[r] & B \ar[r] & 0}
	\]
	with exact rows so that, by the snake lemma, the morphism $g$ is an isomorphism. Therefore,
	splittings of $\xi$ are canonically identified with splittings of the short exact sequence
	\begin{equation}\label{eq:suffices}
			\xymatrix{0 \ar[r] & A \ar[r] & A \amalg_{P_1} P_0 \ar[r]  & B
		\ar[r] & 0. }
	\end{equation}
	We now provide the claimed bijection. Let $\varphi \in \Hom(P_0,A)$ such that $\varphi \circ
	d = f_1$. Then we obtain a morphism
	\[
		A \amalg_{P_1} P_0 \lra A, (a,p) \mapsto a + \varphi(p)
	\]
	which defines a splitting of \eqref{eq:suffices}. Vice versa, given a splitting $s$, we pull
	back via the canonical morphism $P_0 \to A \amalg_{P_1} P_0$ to obtain a morphism $\varphi: P_0 \to
	A$ which, by the relations defining the pushout, satisfies $\varphi \circ d = f_1$. It is
	immediate to verify that these two assignments define inverse maps.
\end{proof}

We will now describe an analogous point of view on trivial $2$-extensions. Let 
\[
	\xi: 0 \lra A \lra X_1 \lra X_0 \lra B \lra 0
\]
be a $2$-extension of $B$ by $A$. A {\em trivialization} of $\xi$ is a commutative diagram
\begin{equation}\label{eq:trivy}
		\xymatrix{ A \ar[r]  & X_1 \ar[rr]\ar@{^{(}->}[dr] &  & X_0 \ar[r] & B \\
		& & Y \ar@{->>}[ur]& & }
\end{equation}
such that the natural maps $Y/A \to S$ and $Y/X_1 \to B$ are isomorphisms. Note that, in contrast to
the case $n=1$, where the collection of splittings forms a {\em set}, the collection of trivializations
naturally organizes into a {\em groupoid}: the groupoid $\Triv(\xi)$ introduced in Definition
\ref{defi:triv}. We will now argue that the groupoid $\Triv(\xi)$ is the $n=2$ analog of the set
$\Split(\xi)$: Fix a projective resolution $P_{\bullet}$ of $B$, and a lift of $\id:
B \to B$ as in \eqref{eq:lift}. We obtain a corresponding cocycle $f_2 \in
\Hom(P_2,A)$ and consider the complex
\[
	\Hom(P_0, A) \overset{d}{\to} \Hom(P_1, A) \overset{d}{\to} \Hom(P_2, A).
\]

\begin{prop}\label{prop:triv} There is a canonical equivalence of groupoids
	\[
		T: d^{-1}(f_2)//\Hom(P_0,A) \overset{\simeq}{\lra} \Triv(\xi)
	\]
	where the left-hand side denotes the action groupoid corresponding to the action of the
	abelian group $\Hom(P_0,A)$ on $d^{-1}(f_2)$ via $(t,g) \mapsto t + g \circ d$.
	In particular,
	\begin{enumerate}
		\item A trivialization exists, i.e., $\Triv(\xi) \ne \emptyset$, if and only if the class of $\xi$ in $\Ext^2(B,A)$ is
			trivial.
		\item Assume that the class of $\xi$ is trivial. Then
			\begin{enumerate}[label=(\roman{*})]
					\item The set of isomorphism classes $\pi_0(\Triv(\xi))$ is acted
					upon simply transitively by the group $\Ext^1(B,A)$. 
				\item The automorphism group of any object of $\Triv(\xi)$ is isomorphic to the
					abelian group $\Hom(B,A)$.
			\end{enumerate}
	\end{enumerate}
\end{prop}
\begin{proof}
	Let $t$ be an object of the action groupoid, i.e., an element $t \in \Hom(P_1, A)$ such that
	$t \circ d = f_2$. From the chosen diagram 
	\[
		\xymatrix{ P_3 \ar[r] \ar[d]& P_2 \ar[r]\ar[d]^{f_2} & P_1 \ar[r]\ar[d]^{f_1} & P_0 \ar[r]\ar[d]^{f_0} & B \ar[d]^{\id} \\
	0 \ar[r] & A \ar[r]^i & X_1 \ar[r] & X_0 \ar[r] & B }
	\]
	we obtain a commutative diagram
	\begin{equation}\label{eq:new}
			\xymatrix{ 0\ar[r] \ar[d] & P_1/\im P_2 \ar[r]\ar[d]^{f_1 - i \circ t} & P_0 \ar[r]\ar[d]^{f_0} & B \ar[d]^{\id}\ar[r] & 0 \\
		A  \ar[r]^i & X_1 \ar[r] &  X_0 \ar[r] & B \ar[r] & 0}
	\end{equation}
	with exact rows. We form the pushout
	\[
		Y_t := X_1 \amalg_{P_1 / \im P_2} P_0 \cong X_1 \amalg_{P_1} P_0
	\] 
	to obtain a commutative diagram 
	\begin{equation}\label{eq:triv2}
			\xymatrix{ & & Y_t \ar[dr] & \\
			A  \ar[r]^i & X_1 \ar[rr]\ar[ur] & &  X_0 \ar[r] & B. }
	\end{equation}
	The sequence
	\[
		0 \lra X_1 \lra Y_t \lra B \lra 0
	\]
	is a Yoneda pushout of the top exact sequence in \eqref{eq:new} and therefore exact. We
	further obtain from \eqref{eq:triv2} a commutative diagram
	\[
		\xymatrix{ 0\ar[r] &  X_1/A \ar[r]\ar[d]^{\id} & Y_t/A \ar[r]\ar[d] &  B \ar[r] \ar[d]^{\id}& 0\\
	0\ar[r] & X_1/A \ar[r] &  X_0 \ar[r] & B \ar[r] & 0}
	\]
	where the top row is exact by the third isomorphism theorem, and the bottom row is trivially
	exact. The snake lemma implies that $Y_t/A \to X_0$ is an isomorphism so that the diagram
	\eqref{eq:triv2} defines an object of $\Triv(\xi)$. This defines the functor $T$ 
	on objects. Given a morphism between objects $t$ and $t'$, i.e., an element $g \in
	\Hom(P_0,A)$ such that $t' = t + g \circ d$, we obtain an induced morphism $Y_t \to Y_t'$
	via the formula
	\[
		X_1 \amalg_{P_1}^t P_0 \lra X_1 \amalg_{P_1}^{t'} P_0,\; (x,p) \mapsto (x + i(g(p)), p)
	\]
	This association defines $T$ on morphisms. Fully faithfulness and essential surjectivity are
	straightforward to verify.
\end{proof}

\subsubsection{Proof of Green's theorem}

In the previous sections we have seen that the compatibility
\begin{equation}\label{eq:compat2}
		\Delta(ab) = \Delta(a) \Delta(b)
\end{equation}
of multiplication and comultiplication fails for the abstract Hall algebra since the forgetful functor
\[
	\pi: \square \to \frame
\]
from exact $3$-by-$3$ squares to exact frames is not an equivalence. However, the language of
groupoids gives us a precise measure for the failure of \eqref{eq:compat2}: the
$2$-fibers of the functor $\pi$. As we have seen in Lemma \ref{lem:2fib}, the $2$-fiber $\square_f$
over a fixed frame $f$ is given by the groupoid $\Triv(\xi_f)$ of trivializations of the
$2$-extension $\xi_f$ obtained as the Baer sum of the two $2$-extensions of $C''$ by $A$ which form
the frame $f$. 

We will now make sufficient assumptions on the category $\C$ so that we can work around the 
fact that $\pi$ is not an equivalence after passing from groupoids to functions by means of the 
monoidal functor
\[
	\F: \Span^f(\Grpd) \lra \Vect_{\QQ}.
\]
Namely, we will assume that the abelian category $\C$ is
\begin{enumerate}
	\item {\em finitary} in the sense of Definition \ref{defi:finitary},
	\item {\em cofinitary}: every object of $\C$ has only finitely many subobjects,
	\item {\em hereditary}: for every pair of objects $A,B$ of $\C$, we have $\Ext^i(A,B) \cong
		0$ for $i>1$.
\end{enumerate}
We have seen that the condition on $\C$ to be finitary implies that the abstract Hall algebra
defines an algebra object in $\Span^f(\Grpd) \subset \Span(\Grpd)$. 

\begin{prop}\label{prop:assumptions} Let $\C$ be a finitary abelian category. 
	\begin{enumerate}
		\item The condition on $\C$ to be cofinitary implies that the object $(\S_1,\Delta, c)$ 
			defines a coalgebra object in $\Span^f(\Grpd)$. 
		\item The condition on $\C$ to be hereditary implies that all $2$-fibers of $\pi$
			are nonempty.
	\end{enumerate}
\end{prop}

Consider the commutative diagram
\[
	\xymatrix{ \square \ar[dr]^{\pi} \ar@/_/[ddr]_{L'} \ar@/^/[drr]^{R'}& \\
	& \frame \ar[r]^{R}\ar[d]_{L} & \S_2 \times \S_2 \ar[r]^{F \times F} \ar[d]_{P} & \S_1 \times \S_1\\
		& \S_2 \times \S_2 \ar[r]^-{G \times G}\ar[d]_{F \times F} & \S_1 \times \S_1 \times \S_1 \times \S_1 & \\
		& \S_1 \times \S_1 & & }
\]
The failure of the equality \eqref{eq:compat2} after passing to functions is given by 
\[
	(R')_!(L')^* \overset{?}{\ne} (R)_!(L)^*.
\]
We compute the left-hand side explicity: letting 
\[
	\varphi = \one_{(A \to B \to C,  A'' \to B'' \to C'')} \in \F(\S_2 \times \S_2),
\]
we have
\begin{equation}\label{eq:scale}
	(R')_!(L')^*(\varphi)  = R_!\pi_!\pi^*L^*(\varphi)
	= \frac{|\Ext^1(C'',A)|}{|\Hom(C'',A)|} R_!L^*(\varphi).
\end{equation}
The last equality follows from Lemma \ref{lem:transfer} below since, using Proposition
\ref{prop:triv} together with the assumption on $\C$ to be hereditary, we have
\[
	|\square_f| = |\Triv(\xi_f)| = \frac{|\Ext^1(C'',A)|}{|\Hom(C'',A)|}.
\]

\begin{lem} \label{lem:transfer} Let $F: \A \to \B$ be a functor which is both $\pi_0$-finite and locally finite. Let $b$
	be an object of $\B$. Then, for every $\varphi \in \F(\B)$, we have
	\[
		(F_!F^*\varphi)(b) = |\A_b| \varphi(b)
	\]
	so that the effect of $F_!F^*$ on the function $\varphi$ is given by rescaling with the
	groupoid cardinalities of the $2$-fibers of $F$.
\end{lem}
\begin{proof} This follows immediately from the definitions.
\end{proof}

To compensate for the rescaling factor 
\[
\frac{|\Ext^1(C'',A)|}{|\Hom(C'',A)|}
\]
appearing in \eqref{eq:scale}, we use
the following modifications:
\begin{enumerate}
	\item Instead of the comultiplication $\Delta$, we use the comultiplication
		represented by the span
		\[
			\Delta': \vcenter{\xymatrix{ & \ar[dl]_F \S_2 \ar[dr]^{G'} &\\ \S_1  & & \S_1 \times \S_1}}
		\]
		where $G'$ assigns to a short exact sequence $A \to B \to C$ the pair of objects
		$(C,A)$ (instead of $(A,C)$).
	\item Letting $H = \F(\S_1,\mu,e)$ we define on $H \otimes H$ the {\em twisted}
		algebra structure
		\[
			\mu^t: (H \otimes H) \otimes (H \otimes H) \lra H \otimes H 
		\]
		given by setting
		\[
			(\one_A \otimes \one_B) (\one_{A'} \otimes \one_{B'}) :=  \frac{|\Ext^1(A',B)|}{|\Hom(A',B)|}(\one_A \one_{A'})
			\otimes (\one_B \one_{B'}).
		\]
\end{enumerate}

We arrive at the main result of this section.

\begin{theo}[Green] Let $\C$ be a finitary, cofinitary, and hereditary abelian category and consider
	the datum $H  = \F(\S_1, \mu, e, \Delta', c)$. Then the diagram
	\[
		\xymatrix{ 
			H \otimes H \ar[r]^\mu \ar[d]^{\Delta' \otimes \Delta'} &  H
			\ar[d]^{\Delta'}\\
			(H \otimes H) \otimes (H \otimes H) \ar[r]_-{\mu^t}  &   H \otimes H }
	\]
	commutes. 
\end{theo}

\subsubsection{Example}

Note, that we have investigated the compatibility of product and coproduct without determining an
explicit formula for the coproduct. We now provide a formula and analyze the compatibility of multiplication
and comultiplication of product and coproduct for the category $\Vect_{\FF_q}'$ of finite
dimensional $\FF_q$-vector spaces. 

Let $\C$ be a finitary and cofinitary abelian category. 
For objects $A,A'$ in $\C$, we introduce the groupoid $\GExt(A',A)$ with objects given by short exact
sequences
\[
	0 \lra A \lra X \lra A' \lra 0
\]
in $\C$ and morphisms given by diagrams
\[
	\xymatrix{
		A \ar[r]\ar[d]^{\id} & X \ar[r] \ar[d]^{\cong}& A' \ar[d]^{\id} \\
	A \ar[r] & X' \ar[r] & A' }
\]
We denote by $\GExt(A',A)^B$ the full subgroupoid of $\GExt(A',A)$ consisting of those short exact
sequences where $X \cong B$.
With this notation, we have
\[
	\Delta'(\one_{[B]}) = (G')_!F^*(\one_{[B]}) = \sum_{[A'],[A]} |\GExt(A',A)^B| \one_{[A']}
	\otimes \one_{[A]}. 
\]

\begin{exa} For the category of finite dimensional $\FF_q$-vector spaces we have
	\[
		\F(\S_1) \cong \bigoplus_{n \ge 0} \QQ \one_n
	\]
	where we set $\one_n :=
	\one_{[\FF_q^n]}$. Further, we have seen
	\[
		\one_n \one_m = \left[ \begin{array}{c} n+m\\ m \end{array} \right]_q \one_{m + n}.
	\]
	We compute
	\begin{align*}
		\Delta'(\one_n) & = \sum_{k+l = n} |\GExt(\FF_q^k,\FF_q^l)^{\FF_q^n}| \one_k \otimes
		\one_l \\
		& = \sum_{k+l = n} q^{-kl} \one_k \otimes
		\one_l. 
	\end{align*}
	Thus, we have 
	\begin{align*}
	\Delta'(\one_m \one_n) & = \left[ \begin{array}{c} n+m\\ n \end{array} \right]_q \sum_{x+y = n+m} q^{-xy} \one_x \otimes \one_y
	\end{align*}
	and
	\begin{align*}
		\Delta'(\one_m) \Delta'(\one_n)&= (\sum_{k+l = m} q^{-kl} \one_k \otimes
		\one_l)(\sum_{r+s = n} q^{-rs} \one_r \otimes \one_s)\\
		& = \sum_{n = r+s, m = k+l} q^{-kl-rs-rl} \one_{k+r} \otimes \one _{l+s}\\
		& = \sum_{n+m = x+ y} q^{-xy} \sum_{k \le n} q^{k(n-l)} \left[ \begin{array}{c} x\\ k \end{array} \right]_q
		\left[ \begin{array}{c} y\\ n-k \end{array} \right]_q.
	\end{align*}
	The compatibility of product and coproduct up to twist therefore amounts to the formula
	\[
		\left[ \begin{array}{c} n+m\\ n \end{array} \right]_q = \sum_{k \le n} q^{k(n-k)} 
		\left[ \begin{array}{c} x\\ k \end{array} \right]_q
		\left[ \begin{array}{c} y\\ n-k \end{array} \right]_q
	\]
	where $x+y = n+m$.
\end{exa}

\newpage

\subsection{Hall monoidal categories}
\label{sec:hallmonoidal}

The construction of the abstract Hall algebra in Section \ref{subsec:spans} only makes use of the groupoids
$\S_{\le 3}$ and the $2$-Segal conditions involving them. In this section, we 
use {\em functors} instead of functions to define the {\em Hall monoidal category} of a finitary
proto-abelian category $\C$. This construction utilizes the $2$-Segal conditions involving $\S_{\le
4}$. 

Given a groupoid $\A$, we denote by $\Fun(\A)$ the category of functors from $\A$ to the
category $\Vect_{\CC}$ of finite dimensional complex vector spaces which are nonzero on only finitely 
many isomorphism classes of $\A$. Let $F:\A \to \B$ be a functor of groupoids. We have:
\begin{itemize}
	\item if $F$ is $\pi_0$-finite, then we have a corresponding pullback functor
		\[
			F^*: \Fun(\B) \lra \Fun(\A), \varphi \mapsto \varphi \circ F.
		\]
	\item if $F$ is locally finite, then we have a pushforward functor 
		\[
			F_!: \Fun(\A) \lra \Fun(\B)
		\]
		which is defined as a left Kan extension functor. By the pointwise formula for Kan
		extensions, we have
		\[
			F_!(\varphi)(b) = \colim_{\A_b} \varphi_{|\A_b}
		\]
		where $\A_b$ denotes the $2$-fiber of $F$ over $b$.
\end{itemize}

These operations satisfy the following compatibility conditions (in analogy to Proposition
\ref{prop:basefun}).

\begin{prop}\label{prop:basefunkan}
	\begin{enumerate}
		\item {\em Functoriality.} 
			\begin{enumerate}
				\item Let $F: \A \to \B$ and $G: \B \to \C$ be $\pi_0$-finite functors
			of groupoids. Then we have 
			\[
				(G \circ F)^* = F^* \circ G^*.
			\]
				\item Let $F: \A \to \B$ and $G: \B \to \C$ be locally finite functors
			of groupoids. Then we have a
			canonical isomorphism
			\[
				(G \circ F)_! \cong G_! \circ F_!.
			\]
			\end{enumerate}
		\item {\em Base change.} Let 
			\begin{equation}\label{eq:2pul}
					\xymatrix{ \X \ar[d]_{G'}\ar[r]^{F'} & \B\ar[d]^G \\
					\A \ar[r]^F & \C \ultwocell<\omit>}
			\end{equation}
			be a $2$-pullback square with $F$ locally finite and $G$ $\pi_0$-finite.
			Then we have a canonical isomorphism
			\[
				(F')_! \circ (G')^* \cong G^* \circ F_!.
			\]
	\end{enumerate}
\end{prop}

Let $\C$ be a finitary proto-abelian category. Given a span of groupoids
\[
	\vcenter{\xymatrix{ & \ar[dl]_L \X \ar[dr]^{R} &\\ \A  & & \B}}
\]
with $L$ $\pi_0$-finite and $R$ locally finite, we obtain a corresponding functor
\[
	R_! \circ L^*: \Fun(\A) \lra \Fun(\B).
\]
Applying this to the span
\[
	\vcenter{\xymatrix{ & \ar[dl]_G \S_2 \ar[dr]^F &\\ \S_1\times \S_1  & & \S_1}}
\]
yields a functor $\Fun(\S_1 \times \S_1) \to \Fun(\S_1)$ which we precompose with the pointwise
tensor product $\Fun(\S_1) \times \Fun(\S_1) \to \Fun(\S_1 \times \S_1)$ to obtain 
\[
	\otimes: \Fun(\S_1) \times \Fun(\S_1) \lra \Fun(\S_1).
\]
Further, from the span
\[
	\vcenter{\xymatrix{ & \ar[dl]_{\id} \S_0 \ar[dr]^{\sigma_0} &\\ \S_0  & & \S_1}}
\]
we obtain a functor $\Fun(\S_0) \lra \Fun(\S_1)$ which we evaluate on $k$ to obtain 
\[
	I \in \Fun(\S_1).
\]

\begin{theo} Let $\C$ be a finitary proto-abelian category. The datum $(\Fun(\S_1), \otimes, I)$ naturally extends to a
	monoidal structure on the category $\Fun(\S_1)$.
\end{theo}
\begin{proof} 
	We sketch the basic idea of the proof, the main point being the derivation of MacLane's pentagon. 
	There are five $2$-Segal conditions involving $\S_4$, 
	corresponding to the five possible subdivisions of a planar pentagon. We denote by $\P_{ij}$
	the subdivision $0 \le i < j \le 4$ of the pentagon. We further introduce the
	notation
	\begin{equation}\label{eq:subdiv}
		\S(\P_{ij}) = \S_{\{0,\dots,i,j,\dots,4\}} \times_{\S_{\{i,j\}}} \S_{\{i,\dots,j\}}
	\end{equation}
	for the corresponding pullback so that, for every subdivision, we have, by the $2$-Segal
	property, an equivalence
	\[
		\S_4 \lra \S(\P_{ij}).
	\]
	For example, we have
	\[
		\S(\P_{13}) = \S_{\{0,1,3,4\}} \times_{\S_{\{1,3\}}} \S_{\{1,2,3\}} \cong \S_3
		\times_{\S_1} \S_2
	\]
	Similarly, we label the five different triangulations $\T_{ij,kl}$ of the pentagon
	via their internal edges $i \to j$ and $k \to l$. We use the notation
	$\S(\T_{ij,kl})$ analogous to \eqref{eq:subdiv} so that we have, for example,
	\[
		\S(\T_{13,14}) = \S_{\{0,1,4\}} \times_{\S_{\{1,4\}}} \S_{\{1,3,4\}}
		\times_{X_{\{1,3\}}} \S_{\{1,2,3\}} \cong \S_2
		\times_{\S_1} \S_2\times_{\S_1} \S_2.
	\]
	We obtain a commutative diagram of groupoids 
	\begin{equation}
		\label{eq:huge-pentagon}
		{ 
			\begin{tikzpicture}

				\node (origin) at (0,0) (origin) {$\S_4$}; 
				\node (P0) at (0:4cm)  { $\S(\T_{02,03}) $}; 
				\node (P2) at (1*72:4cm) {$\S(\T_{02,24})$};
				\node (P4) at (2*72:4cm) {$\S(\T_{14,24})$}; 
				\node (P1) at (3*72:4cm) {$\S(\T_{13,14})$}; 
				\node (P3) at (4*72:4cm) {$\S(\T_{03,13})$}; 
				\path (P0) -- (P2)
				node [midway] (P02) {$\S(\P_{02})$}; 
				\path (P2) -- (P4) node [midway] (P24) {$\S(\P_{24})$}; 
				\path (P4) -- (P1) node [midway] (P14) {$\S(\P_{14})$ };
				\path (P1) -- (P3) node [midway] (P13) {$\S(\P_{13})$};
				\path (P0) -- (P3) node [midway] (P03) {$\S(\P_{03})$}; 
				\draw [-latex,  thick] (P02) -- (P0);
				\draw   [-latex,  thick] (P02) -- (P2);
				\draw  [-latex, thick] (origin) -- (P0);
				\draw  [-latex,  thick] (origin) -- (P1);
				\draw  [-latex,  thick] (origin) -- (P2);
				\draw  [-latex,  thick] (origin) -- (P3);
				\draw  [-latex,  thick] (origin) -- (P4);
				\draw  [-latex,  thick] (origin) -- (P02);
				\draw  [-latex,  thick] (origin) -- (P24);
				\draw  [-latex,  thick] (origin) -- (P14);
				\draw  [-latex,  thick] (origin) -- (P13);
				\draw  [-latex,  thick] (origin) -- (P03);
				\draw [-latex,  thick] (P24) -- (P2); 
				\draw [-latex,  thick] (P24) -- (P4); 
				\draw [-latex,  thick] (P14) -- (P4); 
				\draw [-latex,  thick] (P14) -- (P1); 
				\draw [-latex,  thick] (P13) -- (P3); 
				\draw [-latex,  thick] (P13) -- (P1); 
				\draw [-latex,  thick] (P03) -- (P0); 
				\draw [-latex,  thick] (P03) -- (P3); 

			\end{tikzpicture}
		}
	\end{equation}
	in which, by the various $2$-Segal conditions, all functors are equivalences.

	The basic idea of the argument is now as follows: Let $\varphi,\psi,\xi$ be objects in
	$\Fun(\S_1)$. We have a diagram of groupoids
	\[
		\xymatrix{
			\S_{\{0,1,2\}} \times_{\S_{\{0,2\}}} \S_{\{0,2,3\}} & \ar[l]
			\S_{\{0,1,2,3\}} \ar[r] & \S_{\{0,1,3\}}\times_{\S_{\{1,3\}}}
			\S_{\{1,2,3\}}} 
	\]
	where, by the lowest $2$-Segal conditions corresponding to the two triangulations of a
	square, both functors are equivalences. They are responsible for isomorphisms
	\[
		\xymatrix{
			(\varphi \otimes \psi) \otimes \xi &\ar[l]_{\cong}^{\alpha_1} \varphi \otimes \psi \otimes \xi
			\ar[r]^{\cong}_{\alpha_2}  & \varphi \otimes (\psi \otimes \xi) }
	\]
	which we use to define the associator of the monoidal structure as $\alpha = \alpha_2 \circ
	\alpha_1^{-1}$. Here, the middle term $\varphi \otimes \psi \otimes \xi$ is defined as a
	pull-push along the span
	\[
		\xymatrix{
			\S_{\{0,1\}} \times \S_{\{1,2\}} \times \S_{\{2,3\}} & \ar[l]
			\S_{\{0,1,2,3\}} \ar[r] & \S_{\{0,3\}}}.
	\]
	Given four objects $\varphi,\psi,\xi,\varepsilon$, the commutative diagram
	\eqref{eq:huge-pentagon} is responsible for a commutative diagram of isomorphisms
	\begin{equation}
		\label{eq:pentagon2}
		{ 
			\begin{tikzpicture}
				{\scriptsize
					\node (origin) at (0,0) (origin) {$\varphi \otimes \psi \otimes \xi \otimes \varepsilon$}; 
					\node (P0) at (0:4cm)  { $((\varphi \otimes \psi) \otimes \xi) \otimes \varepsilon$}; 
					\node (P2) at (1*72:4cm) {$(\varphi \otimes \psi) \otimes (\xi
					\otimes \varepsilon)$};
					\node (P4) at (2*72:4cm) {$\varphi \otimes (\psi \otimes (\xi \otimes
					\varepsilon))$}; 
					\node (P1) at (3*72:4cm) {$\varphi \otimes ((\psi \otimes \xi)
					\otimes \varepsilon)$}; 
					\node (P3) at (4*72:4cm) {$(\varphi \otimes (\psi \otimes \xi)) \otimes \varepsilon$}; 
					\path (P0) -- (P2)
					node [midway] (P02) {$(\varphi \otimes \psi) \otimes \xi \otimes \varepsilon$}; 
					\path (P2) -- (P4) node [midway] (P24) {$\varphi \otimes \psi
					\otimes (\xi \otimes \varepsilon)$}; 
					\path (P4) -- (P1) node [midway] (P14) {$\varphi \otimes (\psi
					\otimes \xi \otimes \varepsilon)$ };
					\path (P1) -- (P3) node [midway] (P13) {$\varphi \otimes (\psi
					\otimes \xi) \otimes \varepsilon$};
					\path (P0) -- (P3) node [midway] (P03) {$(\varphi \otimes \psi
					\otimes \xi) \otimes \varepsilon$}; 
					\draw [-latex,  thick] (P02) -- (P0);
					\draw   [-latex,  thick] (P02) -- (P2);
					\draw  [-latex, thick] (origin) -- (P0);
					\draw  [-latex,  thick] (origin) -- (P1);
					\draw  [-latex,  thick] (origin) -- (P2);
					\draw  [-latex,  thick] (origin) -- (P3);
					\draw  [-latex,  thick] (origin) -- (P4);
					\draw  [-latex,  thick] (origin) -- (P02);
					\draw  [-latex,  thick] (origin) -- (P24);
					\draw  [-latex,  thick] (origin) -- (P14);
					\draw  [-latex,  thick] (origin) -- (P13);
					\draw  [-latex,  thick] (origin) -- (P03);
					\draw [-latex,  thick] (P24) -- (P2); 
					\draw [-latex,  thick] (P24) -- (P4); 
					\draw [-latex,  thick] (P14) -- (P4); 
					\draw [-latex,  thick] (P14) -- (P1); 
					\draw [-latex,  thick] (P13) -- (P3); 
					\draw [-latex,  thick] (P13) -- (P1); 
					\draw [-latex,  thick] (P03) -- (P0); 
					\draw [-latex,  thick] (P03) -- (P3); 
				}
			\end{tikzpicture}
		}
	\end{equation}
	Along the boundary of \eqref{eq:pentagon2}, we extract the commutative MacLane pentagon for
	the tensor product on $\Fun(\S_1)$. To make this argument formally precise is somewhat
	tedious: it is best done by interpreting monoidal structures in terms of Grothendieck fibrations over the category $\Delta^{\op}$. 
\end{proof}

We call the monoidal category $(\Fun(\S_1), \otimes, I)$ the {\em Hall monoidal category}
$\Hall^{\otimes}(\C)$ of $\C$.

\begin{rem} The construction of the Hall monoidal category can be understood as an
	instance of {\em Day convolution} \cite{day}: From the simplicial groupoid of flags we can construct a
	promonoidal structure on the groupoid $\S_1$ which is then turned into a monoidal one by
	passing to functors.
\end{rem}

We discuss two examples:

\subsubsection{$\Hall^{\otimes}(\Vectone)$} 

We restrict attention to the skeleton $\C \subset \Vectone$ consisting of the standard pointed
sets $\{\ast, 1,\dots,n\}$, $n \ge 0$. Thus, we have 
\[
	\S_1(\C) \simeq \coprod_{n\ge 0} B S_n
\]
so that an object of $\Hall^{\otimes}(\C)$ is given by a sequence $(\rho_n)_{n
\ge 0}$ of representations of $S_n$ in $\Vect_{\CC}$ where only finitely many
representations are nonzero. We have
\[
	(\rho_n)_{n \ge 0} \cong \bigoplus_{n \ge 0} \rho_n
\]
where $\rho_n$ is interpreted as an object of $\Hall^{\otimes}(\C)$ which is zero on all groupoids
$B S_m$, $m \ne n$. The tensor product of $\Hall^{\otimes}(\C)$ is additive so that it suffices to
describe $\rho_n \otimes \rho_m$. This is obtained by pull-push along the span of groupoids
\[
	\S_1 \times \S_1 \longleftarrow \S_2 \longrightarrow \S_1
\]
which factors through the pull-push along the span
\[
	B S_n \times B S_m \overset{\cong}{\longleftarrow} B (S_n \times S_m) \longrightarrow B S_{n+m}.
\]
This is obtained by restricting to the subgroupoid of $\S_2$ spanned by a fixed chosen short exact sequence
\begin{equation}\label{eq:sesp}
		\{\ast, 1, \dots, n\} \hookrightarrow \{\ast, 1, \dots, n+m\} \twoheadrightarrow \{\ast, 1, \dots, m\}.
\end{equation}
and noting that this choice determines an embedding of the automorphism group $S_n \times S_m$ of
$\eqref{eq:sesp}$ into the automorphism group $S_{n+m}$ of $\{\ast, 1, \dots, n+m\}$. The tensor
product $\rho_n \otimes \rho_m$ in $\Hall^{\otimes}(\C)$ is therefore given by the induced represention 
of the external tensor product $\rho_n \boxtimes \rho_m$ along the embedding $S_n \times S_m \subset
S_{n+m}$.

The resulting monoidal category plays an important role in classical representation theory. It is
canonically monoidally equivalent to the category of {\em polynomial functors}: functors $F:
\Vect_{\CC} \to
\Vect_{\CC}$ satisfying the following condition
\begin{itemize}
	\item for every collection of morphisms $f_i: V \to W$, $1 \le i \le n$, between fixed
		vector spaces, the expression $F(\lambda_1 f_1 + \dots + \lambda_r f_r)$, $\lambda_i
		\in \CC$, is a function polynomial with coefficients in $\Hom(F(V),F(W))$.
\end{itemize}
The polynomial functor corresponding to the representation $\rho_n$ is given by 
\[
	F: \Vect_{\CC} \to \Vect_{\CC}, V \mapsto (X_n \otimes V^{\otimes n})^{S_n}.
\]
The Hall monoidal structure induces the structure of an associative algebra on
the Grothendieck group $K_0(\Hall^{\otimes}(\C))$. Using the interpretation via polynomial functors
we can canonically identify $K_0(\Hall^{\otimes}(\C))$ with the algebra $\Lambda$ of symmetric functions. 
Under this identification, the basis given by isomorphism classes of irreducible representations
gets identified with the basis of $\Lambda$ given by the Schur functions. Therefore, we obtain yet
another Hall algebraic construction of the algebra of symmetric functions which naturally exhibits an
interesting basis. For a detailed exposition of this theory (due to Schur) we refer the reader to
\cite{macdonald}.

\subsubsection{$\Hall^{\otimes}(\Vectq)$}

We consider the skeleton $\C \subset \Vectq$ consisting of the standard objects
$\Fq^n$ so that we have
\[
	\S_1(\C) \simeq \coprod_{n\ge 0} B \GL_n(\Fq).
\]
An object of $\Hall^{\otimes}(\C)$ is therefore given by a sequence $(\rho_n)_{n
\ge 0}$ of representations of $\GL_n(\Fq)$ in $\Vect_{\CC}$ with only finitely many
nonzero components. 
The tensor product $\rho_n \otimes \rho_m$ is obtained by pull-push along the span of groupoids
\[
	\S_1 \times \S_1 \longleftarrow \S_2 \longrightarrow \S_1
\]
which factors through the pull-push along the span
\[
	B \GL_n(\Fq) \times B \GL_m(\Fq) \overset{\cong}{\longleftarrow} B P_{n,m}(\Fq) \longrightarrow B
	\GL_{n+m}(\Fq)
\]
where $P_{n,m}$ is the parabolic subgroup of $\GL_{n+m}$ given by the automorphism group of a fixed
short exact sequence
\[
	\Fq^n \hookrightarrow \Fq^{n+m} \twoheadrightarrow \Fq^m.
\]
Therefore, the tensor product $\rho_n \otimes \rho_m$ in $\Hall^{\otimes}(\C)$ is given by first
pulling back the representation of $\GL_n(\Fq) \times \GL_m(\Fq)$ along $P_{n,m} \to \GL_n(\Fq)
\times \GL_m(\Fq)$ and then forming the induced representation along $P_{n,m} \subset
\GL_{n+m}(\Fq)$.

Green \cite{green-gln, macdonald} has developed a $q$-analog of Schur's theory which uses the associative 
algebra given by the Grothendieck group of $\Hall^{\otimes}(\Vectq)$ to construct all irreducible 
characters of the groups $\GL_n(\FF_q)$. 
The monoidal category $\Hall^{\otimes}(\Vectq)$ itself features in the work of Joyal-Street
\cite{joyal-street} who explain that the commutativity of Green's algebra comes from a 
(partial) braided structure.

\newpage
\section{Derived Hall algebras via $\infty$-groupoids}
\label{section:derived}

The idea of constructing the Hall algebra via the simplicial groupoid of flags $\S_{\bullet}$ is a
very flexible one. We explain how it can be adopted to construct Hall algebras of derived
categories or, more generally, stable $\infty$-categories. This section is a translation of
\cite{toen} (also cf. \cite{bergner}) into the language of $\infty$-categories which makes the
analogy to proto-abelian categories immediate. We use \cite{lurie.htt} as a standard reference.

\subsection{Coherent diagrams in differential graded categories}

An {\em $\infty$-category} $\C$ is a simplicial set such
that, for every $0 < i < n$ and every $\Lambda^n_i \to \C$, there exists a commutative diagram
\[
	\xymatrix{
		\Lambda^n_i \ar[r]\ar[d] & \C\\
		\Delta^n \ar@{-->}[ur] & . }
\]
The arrow $\Lambda^n_i \to \C$ represents the boundary of an
$n$-simplex with $i$th face removed, called an {\em inner horn} in $\C$, 
and the condition asks that it can be filled to a full
$n$-simplex in $\C$. 

\begin{exa} The nerve of a small category provides an example of an
$\infty$-category where every inner horn has a {\em unique} filling. This corresponds to the fact
that every $n$-tupel of composable morphisms has a unique composite. In a general $\infty$-category
the composite is not required to be unique. However, the totality of all horn filling conditions
encodes that it is unique up to a coherent system of homotopies.
\end{exa}

\begin{exa} Given $\infty$-categories $\C$, $\D$, we define the simplicial set $\Fun(\C, \D)$
	of {\em functors from $\C$ to $\D$} given by the internal hom in the category of simplicial
	sets. Then $\Fun(\C,\D)$ is an $\infty$-category.
\end{exa}

\begin{exa} A {\em differential graded (dg) category} $T$ is a category enriched over the monoidal category of
	complexes of abelian groups. The collection of dg categories organizes into a category
	$\dgcat$ with morphisms given by enriched functors. Following \cite{lurie.algebra}, we associate to $T$ an
	$\infty$-category called the dg nerve of $T$.
	
	We associate to the $n$-simplex $\Delta^n$ a dg category $\dg(\Delta^n)$ with objects given by
	the set $\{0,1,\dots,n\}$. The graded $\ZZ$-linear category underlying $\dg(\Delta^n)$ is
	freely generated by the morphisms
	\[
		f_I \in \dg(\Delta^n)(i_-, i_+)^{-m}
	\]
	where $I$ runs over the subsets $\{i_- < i_m < i_{m-1} < \dots < i_1 < i_+ \} \subset
	\{0,1,\dots,n\}$, $m \ge 0$. On these generators, the differential is given by the formula
	\[
		d f_I = \sum_{1 \le j \le m} (-1)^j 
		(f_{I \setminus \{i_j\}}  - 
		f_{\{i_j < \dots < i_m < i_+\}} \circ f_{\{i_- < i_1 < \dots < i_j\}})
	\]
	and extended to all morphisms by the $\ZZ$-linear Leibniz rule.
	We have $d^2 = 0$ on generators and therefore on all morphisms. 
	The dg categories $\dg(\Delta^n)$, $n \ge 0$, assemble to form a cosimplicial object in
	$\dgcat$ which allows us to define the dg nerve of $T$ 
	\[
		\Ndg(T) = \Hom_{\dgcat}(\dg(\Delta^{\bullet}), T).
	\]
	It is shown in \cite[1.3.1.10]{lurie.algebra} that $\Ndg(T)$ is in fact an $\infty$-category.
\end{exa}

It is instructive to analyze the low-dimensional simplices of the dg nerve $\Ndg(T)$:
\begin{itemize}
	\item The $0$-simplices are the objects of $T$.
	\item A $1$-simplex in $\Ndg(T)$ is a morphism $f_{\{0,1\}}: a_0 \to a_1$ of degree $0$ which is closed,
		i.e., $df = 0$.
	\item A $2$-simplex in $\Ndg(T)$ is given by objects $a_0,a_1,a_2$, closed morphisms
		$f_{\{0,1\}}: a_0 \to
		a_1$, $f_{\{1,2\}}: a_1 \to a_2$, $f_{\{0,2\}}: a_0 \to a_2$, and a morphism
		$f_{\{0,1,2\}}: a_0 \to a_2$ of degree $-1$
		which satisfies
		\[
			d f_{\{0,1,2\}}= f_{\{0,2\}} - f_{\{1,2\}} \circ f_{\{0,1\}}
		\]
		so that we obtain a triangle in $T$ which commutes up to the chosen homotopy
		$f_{\{0,1,2\}}$. A key point here is that we do not simply require the triangle to commute up to
		homotopy, but the homotopy is part of the data forming the triangle.
	\item A $3$-simplex in $\Ndg(T)$ involves the data of the four
		boundary $2$-simplices as above and, in addition, a morphism $f_{\{0,1,2,3\}}: a_0
		\to a_3$ of degree $-2$ such that
		\[
			d f_{\{0,1,2,3\}}= f_{\{0,1,3\}} - f_{\{2,3\}} \circ f_{\{0,1,2\}}  -
			f_{\{0,2,3\}}  + f_{\{1,2,3\}} \circ f_{\{0,1\}} .
		\]
		We can interpret this data as follows: The boundary of a $3$-simplex in $\Ndg(T)$
		encodes two homotopies between $f_{\{0,3\}}$ and the composite $f_{\{2,3\}} \circ
		f_{\{1,2\}} \circ f_{\{0,1\}}$ given by $f_{\{0,1,3\}} + f_{\{1,2,3\}} \circ f_{\{0,1\}}$ 
		and $f_{\{0,2,3\}} + f_{\{2,3\}} \circ f_{\{0,1,2\}}$, respectively. To
		obtain a full $3$-simplex in $\Ndg(T)$ we have to provide the homotopy
		$f_{\{0,1,2,3\}}$ between these homotopies.
	\item \dots
\end{itemize}

The passage from a dg category $T$ to the $\infty$-category $\Ndg(T)$ allows us (and forces us) to
systematically consider diagrams in $T$ which commute up to specified coherent homotopy: Let $I$ be a
category and $\N(I)$ its nerve. We define a {\em coherent $I$-diagram in $T$} to be a functor, i.e.,
a map of simplicial sets 
\[
	\N(I) \to \Ndg(T).
\]

\begin{exa} Consider the category $I$ given by the universal commutative square: $I$ has four
	objects $1,2,3,4$, morphisms $f_1: 1 \to 2$, $f_2: 2 \to 4$, $f_3: 1 \to 3$, $f_4: 3 \to 4$
	subject to the relation $f_2 \circ f_1 = f_4 \circ f_3$. An $I$-coherent diagram in
	$T$ consists of
	\[
		\xymatrix{
			a_1 \ar[r]^{f_1}\ar[d]_{f_3}\ar[dr]|-g_{h_2}^{h_1} & a_2\ar[d]^{f_2}\\
			a_3 \ar[r]_{f_4} & a_4
		}
	\]
	where the morphisms $f_1,f_2,f_3,f_4$ and $g$ are closed of degree $0$, and we have $d h_1 =
	g - f_2 \circ f_1$, $d h_2 = g - f_4 \circ f_3$.
\end{exa}

One of the main advantage of homotopy coherent diagrams over homotopy commutative ones is the
existence of a good theory of limits. We give an example.

\begin{exa} Let $\A$ be an abelian category with enough projectives and consider the dg category
	$\Ch^-(\A_{\proj})$ of bounded-above cochain complexes of projective objects in $\A$. We
	define the {\em bounded-above derived $\infty$-category of $\A$} as the dg nerve $\D^-(\A) :=
	\Ndg(\Ch^-(\A_{\proj})$. The ordinary bounded-above derived category is obtained by passing
	to the homotopy category $\h(\D^-(\A))$ which is defined as the ordinary category obtained
	by identifying homotopic morphisms. Consider an edge $f: X \to Y$ in $\D^-(\A)$, i.e., a morphism
	between bounded-above complexes of projectives $X$ and $Y$. Conside the cone of $f$, i.e.,
	the complex with
	\[
		\cone(f)_n = X^{n+1} \oplus Y^n
	\]
	and differential given by 
	\[
		d = \left( \begin{array}{rr} -d_X & 0 \\ f & d_Y \end{array} \right).
	\]
	We obtain a coherent square 
	\begin{equation}\label{eq:cone}
		\xymatrix{
				X \ar[r]^{f}\ar[d] \ar[dr]|0^{h}_{0} & Y\ar[d]^{i}\\
				0 \ar[r] & \cone(f)
			}
	\end{equation}
	in $\D^-(\A)$ where 
	\[
		i: Y \to \cone(f), y \mapsto (0,y)
	\]
	and
	\[
		h: X \to \cone(f), x \mapsto (-x,0)
	\]
	so that we have $dh = 0 - i \circ f$. The key fact is that the diagram \eqref{eq:cone} is a
	pushout diagram in the $\infty$-category $\D^-(\A)$ so that the cone is characterized by a
	universal property. This statement becomes wrong if we pass from $\D^-(\A)$ to the {\em homotopy
	category} $\h(\D^-{\A})$: the image of the square \eqref{eq:cone} commutes up to unspecified
	homotopy, but this data is, in general, insufficient to characterize $\cone(f)$ by a universal 
	property. As an extreme case, consider the cone of the zero morphism $X \to 0$ which is the
	translation $X[1]$. The coherent square 
	\[
		\xymatrix{
				X \ar[r]^{0}\ar[d] \ar[dr]|0^{h}_{0} & 0\ar[d]^{i}\\
				0 \ar[r] & X[1]
			}
	\]
	involves the map
	\[
		h: X \to X[1], x \mapsto x
	\]
	considered as a self-homotopy of $0$. For precise definitions and proofs, 
	we refer the reader to \cite[1.3.2]{lurie.algebra}.
\end{exa}

\subsection{Stable $\infty$-categories}

We take for granted the existence of a theory of limits for $\infty$-categories (cf.
\cite{lurie.htt}). 

\begin{defi}\label{defi:stable} An $\infty$-category $\C$ is called {\em stable} if the following conditions hold:
	\begin{enumerate}
		\item The $\infty$-category $\C$ is pointed. 
		\item 
	\begin{enumerate}
		\item Every diagram in $\C$ of the form
		\[
			\xymatrix{
				A \ar[r] \ar[d] & B\\
				C  & 
			}
		\]
		can be completed to a pushout square of the form
		\[
			\xymatrix{
				A \ar[r] \ar[d] & B\ar[d] \\
				C \ar[r] &  D.
			}
		\]
		\item Every diagram in $\C$ of the form
		\[
			\xymatrix{
					& B \ar[d]\\
				C \ar[r]  & D
			}
		\]
		can be completed to a pullback square of the form
		\[
			\xymatrix{
				A \ar[r] \ar[d] & B \ar[d] \\
				C \ar[r] &  D.
			}
		\]
		\end{enumerate}
		\item A square in $\C$ of the form
		\[
			\xymatrix{
				A \ar[r] \ar[d] & B\ar[d] \\
				C \ar[r] &  D
			}
		\]
		is a pushout square if and only if it is a pullback square.
\end{enumerate}
\end{defi}

\begin{rem} Note that these axioms correspond precisely to the axioms defining a proto-abelian category, except
that we drop the conditions on horizontal (resp. vertical) morphisms to be monic (resp. epic). 
\end{rem}

\begin{exa} The bounded-above derived $\infty$-category $\D^-(\A)$ of an abelian category with
	enough projectives is stable.
\end{exa}

\begin{rem} It can be shown (cf. \cite{lurie.algebra}) that the homotopy category of a stable
	$\infty$-category $\C$ has a canonical triangulated structure. The stable $\infty$-category
	$\C$ can be regarded as an {\em enhancement} of $\h \C$ with better properties such as the
	existence of functorial cones.
\end{rem}

To generalize the simplicial groupoid of flags $\S_{\bullet}$ to stable $\infty$-categories we need
to understand what the $\infty$-categorical analog of a groupoid is:
An {\em $\infty$-groupoid} is an $\infty$-category $\C$ whose homotopy category is a groupoid.  

\begin{exa}\label{exa:infty-groupoids}
	\begin{enumerate}
		\item The nerve of a groupoid is an $\infty$-groupoid.
		\item Given an $\infty$-category $\C$, let $\C^{\simeq}$ denote the simplicial
			subset of $\C$ consisting of only those simplices whose edges become
			isomorphisms in the homotopy category $\h \C$. Then $\C^{\simeq}$ is an
			$\infty$-groupoid called the {\em maximal $\infty$-groupoid in $\C$}.
		\item Let $X$ be a topological space. We define the singular simplicial set
			$\Sing(X)$ with $n$-simplices
			\[
				\Sing(X)_n := \Hom_{\Top}(|\Delta^n|, X)
			\]
			given by continuous morphisms from a geometric $n$-simplex to $X$. Then
			$\Sing(X)$ is an $\infty$-groupoid. The functor $\Sing$ from topological
			spaces to simplicial sets has a left adjoint given by the geometric
			realization $|K|$ of a simplicial set. The pair of functors $(|-|,\Sing)$
			defines a Quillen equivalence between suitably defined model categories of
			topological spaces and $\infty$-groupoids so that, from the point of view of
			homotopy theory, the two concepts are equivalent (see, e.g., \cite{lurie.htt}).
	\end{enumerate}
\end{exa}

\subsection{The $S$-construction and the derived Hall algebra}
With these concepts at hand, we can adapt the theory of Section \ref{section:groupoids} from
proto-abelian categories to stable $\infty$-categories. We define the analog of the simplicial groupoid of
flags $\S_{\bullet}$ for a stable $\infty$-category $\C$ as follows: Let $n \ge 0$. We introduce the
category $T^n = \Fun([1],[n])$ where we interpret the linearly ordered sets $[1]$ and $[n]$ as
categories. A functor
\[
	\N(T^n) \to \C
\]
is simply a coherent version of a triangular diagram of shape \eqref{eq:sdiag}.
We define 
\[
	\S_n \subset \Fun(\N(T^n), \C)^{\simeq}
\]
to be the $\infty$-groupoid of coherent diagrams so that
\begin{enumerate}
	\item the diagonal objects are $0$,
	\item all squares are pushout squares and hence, by stability, also pullback squares.
\end{enumerate}
Using Example \ref{exa:infty-groupoids}(3), we typically interpret the resulting simplicial
$\infty$-groupoid $\S_{\bullet}$ as a simplicial space.

\begin{rem} In other words, besides the idea to use coherent diagrams, the only substantial
	modification in comparison to the case when $\C$ is proto-abelian is to
	allow arbitrary chains of morphisms as opposed to flags given by chains of
	monomorphisms. 
\end{rem}

\begin{theo}[\cite{dk-segal}]\label{theo:2seg} Let $\C$ be a stable $\infty$-category. 
	The simplicial space $\S_{\bullet}(\C)$ is $2$-Segal.
\end{theo}

\begin{rem} For the simplicial groupoids of flags in proto-abelian categories, the pullback
	conditions on the squares \eqref{eq:segal} and \eqref{eq:unital} have to be interpreted in
	the $2$-category of groupoids: the squares have to be $2$-pullback squares.
	In the context of Theorem \ref{theo:2seg}, the pullback conditions have to be interpreted in 
	the $\infty$-category of $\infty$-groupoids. Using the equivalence between
	$\infty$-groupoids and topological spaces this can be made quite explicit: the squares have
	to be homotopy pullback squares.
\end{rem}

The construction of the Hall algebra of a stable $\infty$-category from $\S_{\bullet}(\C)$ can also
be adapted to our new context: given a topological space $X$, we pass to the vector space $\F(X)$ of
functions $\varphi: X \to \QQ$ which are constant along connected components and only supported on
finitely many connected components. All definitions of Section \ref{subsec:groupfunc} admit natural
generalizations to this context. The central idea is to replace the groupoid cardinality
\[
	|\A| = \sum_{[a] \in \pi_0(\A)} \frac{1}{|\Aut(a)|}
\]
by the homotopy cardinality 
\[
	|X| = \sum_{[x] \in \pi_0(X)}
	\frac{1}{|\pi_1(X,x)|}\frac{|\pi_2(X,x)|}{1}\frac{1}{|\pi_3(X,x)|}\dots
\]
as introduced by Baez-Dolan \cite{baez-dolan}, and $2$-pullbacks by homotopy pullbacks. 
In particular, we obtain natural pushforward and pullback operations for maps $X \to Y$ of
topological spaces satisfying suitable finiteness conditions. 
Assume that the stable $\infty$-category $\C$ is {\em finitary}: for every pair of objects $X,Y$,
the groups $\Hom(X,Y[i])$ of morphisms in the homotopy category are finite and non-zero for only
finitely many $i$. Then we may apply the constructions of Section \ref{subsec:groupfunc}, adapted
to the current situation, to obtain the Hall algebra of $\C$.

\begin{exa} Let $\A$ be a finitary abelian category of finite global dimension with enough
	projective objects. Then the dg nerve $\D^b(\A)$ of the full dg subcategory of $\Ch^-(\A_{\proj})$
	consisting of those complexes with bounded cohomology objects is finitary. 
	In this example, we obtain the derived Hall algebra as defined in \cite{toen}.  
\end{exa}

In complete analogy to \cite{toen}, we obtain an explicit description of the structure constants of
the derived Hall algebra of a stable $\infty$-category $\C$. Given objects $X,Y,Z$, we have 
\[
	g_{X,Y}^Z = \frac{|\Hom(X,Z)_Y|\prod_{i>0} |\Hom(X[i],Z)|^{(-1)^i}}{|\Aut(X)| \prod_{i>0}
	|\Hom(X[i],X)|^{(-1)^i}}
\]
where $\Hom$ denotes the morphisms in the homotopy category of $\h\C$ of $\C$ and $\Hom(X,Z)_Y$
denotes the subset given by those morphisms whose cone is isomorphic to $Y$. Note that the structure
constants only depend on the triangulated category $\h\C$

\newpage

\section{Triangulated surfaces in triangulated categories}
\label{section:triangulated}

The abstract Hall algebra construction of Section \ref{subsec:spans} can be generalized to
simplicial $2$-Segal spaces and allows us to interpret the lowest $2$-Segal conditions involving
$\S_{\le 3}$ as associativity and unitality. There are several ways to express the relevance of the
remaining higher $2$-Segal conditions: One can generalize the construction of the Hall monoidal
category of Section \ref{section:groupoids} to define a monoidal $\infty$-category. Alternatively,
one construct a variant of the abstract Hall algebra as an algebra object in a monoidal
$\infty$-category of spans of spaces (cf. \cite{dk-segal}).

In this section, we provide another construction where the higher $2$-Segal constraints are crucial:
We exhibit a cyclic symmetry on the $\S_{\bullet}$-construction of a pretriangulated differential
$\Zt$-graded category and show how, in combination with the $2$-Segal conditions, this can be
utilized to construct invariants of marked oriented surfaces. More details are contained in
\cite{dk-triangulated}. 

\subsection{State sums in associative algebras} 
\label{subsec:algebras}

We sketch a combinatorial construction of surface invariants which form what is known as a
$2$-dimensional open oriented topological field theory (\cite{lazaroiu, moore}). These can be
abstractly defined as monoidal functors from a certain $2$-dimensional noncompact oriented bordism category
into the category of vector spaces. The central result of the theory classifies such functors: they
correspond to symmetric {\em Frobenius algebras}. While there are much more
elegant and intrinsic constructions of the topological field theory associated to such an algebra,
the one we describe is the closest analog to the generalization given in Section
\ref{subsec:triangulated} below. 

Let ${\mathbf k}$ be a field, and let $A$ be an associative finite dimensional ${\mathbf k}$-algebra with chosen basis
$E = \{e_1, e_2, \dots, e_r\}$. The multiplication law of $A$ is numerically encoded in the
structure constants $\lambda_{ij}^{k} \in {\mathbf k}$ defined via $e_i e_j = \sum_k \lambda_{ij}^{k} e_k$. 
Associativity is then expressed by the equations
\begin{equation}\label{dy-eq:assoc}
	\sum\nolimits_{t} \lambda_{ij}^{t} \lambda_{tk}^l = \lambda_{ijk}^{l} = \sum\nolimits_{t} \lambda_{it}^{l}
	\lambda_{jk}^{t}
\end{equation}
where the generalized structure constants $\{\lambda_{ijk}^l\}$ are given by $e_i e_j e_k = \sum_l \lambda_{ijk}^l e_l$. 
We can think of the numbers $\{\lambda_{ij}^k\}$ and $\{\lambda_{ijk}^l\}$ as numerical invariants
attached to triangles and squares, respectively, where the set of vertices is ordered and the edges
are labeled by $E$ as illustrated in 

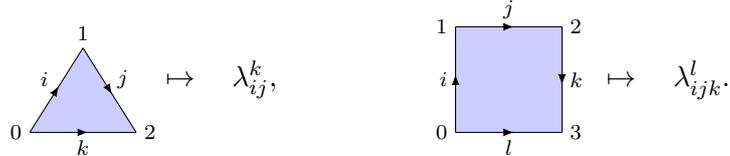
{ \centering 
\begin{tikzpicture}[>=latex,scale=1.0, baseline=(current  bounding  box.south)]
\begin{scope}[scale=1.4]
\coordinate (A0) at (0,0);
\coordinate (A1) at (0.5,0.8);
\coordinate (A2) at (1,0);

\path[fill opacity=0.4, fill=blue!50] (A0) -- (A2) -- (A1) -- cycle;

\begin{scope}[decoration={
    markings,
    mark=at position 0.55 with {\arrow{>}}}
    ] 
\draw[postaction={decorate}] (A0) -- (A1);
\draw[postaction={decorate}] (A1) -- (A2);
\draw[postaction={decorate}] (A0) -- (A2);
\end{scope}

{\scriptsize
\draw (0.25,0.5) node[anchor=east] {$i$};
\draw (0.75,0.5) node[anchor=west] {$j$};
\draw (0.5,0) node[anchor=north] {$k$};
\draw (A0) node[anchor=east] {$0$};
\draw (A1) node[anchor=south] {$1$};
\draw (A2) node[anchor=west] {$2$};
}
\draw (1.8,0.5) node {$\mapsto \quad \lambda_{ij}^k$,};
\end{scope}

\begin{scope}[scale=1.4, xshift=4cm]
\coordinate (A0) at (0,0);
\coordinate (A1) at (0,1);
\coordinate (A2) at (1,1);
\coordinate (A3) at (1,0);

\path[fill opacity=0.4, fill=blue!50] (A0) -- (A3) -- (A2) -- (A1) -- cycle;

\begin{scope}[decoration={
    markings,
    mark=at position 0.55 with {\arrow{>}}}
    ] 
\draw[postaction={decorate}] (A0) -- (A1);
\draw[postaction={decorate}] (A1) -- (A2);
\draw[postaction={decorate}] (A2) -- (A3);
\draw[postaction={decorate}] (A0) -- (A3);
\end{scope}

{\scriptsize
\draw (0,0.5) node[anchor=east] {$i$};
\draw (0.5,1) node[anchor=south] {$j$};
\draw (1,0.5) node[anchor=west] {$k$};
\draw (0.5,0) node[anchor=north] {$l$};
\draw (A0) node[anchor=east] {$0$};
\draw (A1) node[anchor=east] {$1$};
\draw (A2) node[anchor=west] {$2$};
\draw (A3) node[anchor=west] {$3$};
}
\draw (2.0,0.5) node {$\mapsto \quad \lambda_{ijk}^l.$};
\end{scope}
\end{tikzpicture}\\
}
\noindent
Equation \eqref{dy-eq:assoc} is then geometrically reflected by the fact that 
$\{\lambda_{ijk}^l\}$ can be computed in terms of $\{\lambda_{ij}^k\}$ via two different formulas
corresponding to the two possible triangulations of the square.
Similarly, this observation extends to yield numerical invariants of planar convex polygons with ordered
vertices and $E$-labeled edges which can be computed in terms of $\{\lambda_{ij}^k\}$ via any chosen triangulation.

Assume now that $A$ carries a Frobenius structure: a $\k$-linear map ${\operatorname{tr}}: A
\to {\mathbf k}$, called {\em trace}, such that
\begin{enumerate}
	\item $\tr$ is {\em non-degenerate}: the association $a \mapsto \tr(a - )$ defines an
		isomorphism $A \to A^*$,
	\item $\tr$ is {\em symmetric}: for every $a,b \in A$, we have ${\operatorname{tr}}(ab) = {\operatorname{tr}}(ba)$. 
\end{enumerate}
Then we can introduce a dual basis $E^* = \{e_1^*, e_2^*, \ldots, e_r^*\}$ of $A$ which is defined
by the requirement
\[
	\tr(e_i e_j^*) = \begin{cases} 1 & \text{if $i = j$,}\\
		0 & \text{else.} \end{cases}
\]
This allows us to enlarge the range of definition of the above system of invariants 
to include planar polygons with oriented $E$-labeled edges such as
\begin{equation*}
\begin{tikzpicture}[>=latex,scale=1.0, baseline=(current  bounding  box.center)]
\begin{scope}[scale=0.2]
\coordinate (A0) at (1,4);
\coordinate (A1) at (5,0);
\coordinate (A2) at (3,-4);
\coordinate (A3) at (-1,-4);
\coordinate (A4) at (-3,0);

\path[fill opacity=0.4, fill=blue!50] (A0) -- (A1) -- (A2) -- (A3) -- (A4) -- cycle;

\begin{scope}[decoration={
    markings,
    mark=at position 0.55 with {\arrow{>}}}
    ] 
\draw[postaction={decorate}] (A1) -- (A0);
\draw[postaction={decorate}] (A1) -- (A2);
\draw[postaction={decorate}] (A3) -- (A2);
\draw[postaction={decorate}] (A3) -- (A4);
\draw[postaction={decorate}] (A4) -- (A0);
\end{scope}

{\scriptsize
\draw (3.5,2) node[anchor=south] {$i$};
\draw (4.5,-2.5) node[anchor=west] {$j$};
\draw (1,-4) node[anchor=north] {$k$};
\draw (-2.5,-2.5) node[anchor=east] {$l$};
\draw (-2,2) node[anchor=south] {$m$};
}
\draw (1,0) node {$\circlearrowright$};
\draw (15.0,0.5) node {$\mapsto \quad {\operatorname{tr}}(e_me_i^*e_je_k^*e_l)$};
\end{scope}

\end{tikzpicture}
\end{equation*}
where, due to the cyclic invariance of the trace expression, no linear ordering of the vertices is
needed as long as we remember the orientation of the polygon. Again, these invariants can be
computed by choosing any triangulation involving the vertices of the polygon. 

With this new flexibility at hand, we can enlarge the range of definition of our invariants quite
drastically: Let $S$ be a compact oriented surface, possibly with boundary, equipped with a finite
nonempty set $M$ of marked points where we assume that there is at least one marked point on every
boundary component. Unless the pair $(S,M)$ is one of the {\em unstable} cases
\begin{enumerate}
	\item $S$ is a sphere and $|M| < 3$,
	\item $S$ is a boundary marked disk with $|M| < 3$ 
\end{enumerate}
it is always possible to find a triangulation of the surface $S$ with $M$ as vertices. Further, it
is known that any two triangulations of $(S,M)$ are related by a sequence of {\em Pachner moves}
given by replacing the diagonal of a chosen local square by the opposite diagonal. 

Let $(S,M)$ be a stable compact oriented marked surface. Fix an orientation and an $E$-labelling of all
boundary arcs between the boundary marked points. We associate a number to this datum as follows:
\begin{enumerate}
	\item Choose a triangulation $\Delta(S,M)$ of $(S,M)$ and an orientation of all internal edges.  
	\item Given an $E$-labelling $L$ of all internal edges, we define $\lambda_l$ to be
		the product of the numerical invariants of all $E$-labelled triangles
		involved in the triangulation.
	\item The invariant is given by the {\em state sum}
		\begin{equation}\label{eq:statesum}
				\sum_l \lambda_l
		\end{equation}
		where $l$ ranges over all possible $E$-labellings of the internal edges of
		$\Delta(S,M)$.
\end{enumerate}
From the above discussion it is straightforward to show that the resulting number is independent of
the auxiliary data chosen to compute it -- it only depends on $(S,M)$ as well as the orientation and
$E$-labelling of the boundary arcs. Varying over all possible $E$-labels of the boundary arcs, these
numbers form the matrix entries of a $\k$-linear map 
\[
	A^{\otimes^i} \lra A^{\otimes^o}
\]
where $i$ denotes the number of boundary arcs whose orientation is compatible with the surface
orientation, and $o$ denotes the number of remaining boundary arcs. The resulting $\k$-linear maps
assemble to provide a functor from a suitably defined $2$-dimensional bordism category into the
category of vector spaces -- the topological field theory corresponding to $A$.

\subsection{State sums in stable $\infty$-categories} 
\label{subsec:triangulated}

The goal of this section is to show that certain symmetries present in $2$-periodic pretriangulated
differential graded categories can be exploited to define invariants of oriented surfaces via a 
formalism similar to the one explained in Section \ref{subsec:algebras}. The conceptual explanation for
why this is possible is that, in a certain sense, the abstract Hall algebra associated to a
$2$-periodic pretriangulated differential graded category is a Frobenius algebra.

We begin with a heuristic version of the construction in the context of triangulated categories
which is completely parallel to the discussion in Section \ref{subsec:algebras}:
Let ${\mathcal T}$ be a triangulated category with set of
objects $E = \{A, B, \dots, A', B', \dots \}$. We associate to an $E$-labeled triangle the collection 
\begin{equation*}
\begin{tikzpicture}[>=latex, scale=1.0, baseline=(current  bounding  box.center)]
\begin{scope}[scale=1.4]
\coordinate (A0) at (0,0);
\coordinate (A1) at (0.5,0.8);
\coordinate (A2) at (1,0);

\path[fill opacity=0.4, fill=blue!50] (A0) -- (A2) -- (A1) -- cycle;

\begin{scope}[decoration={
    markings,
    mark=at position 0.55 with {\arrow{>}}}
    ] 
\draw[postaction={decorate}] (A0) -- (A1);
\draw[postaction={decorate}] (A1) -- (A2);
\draw[postaction={decorate}] (A0) -- (A2);
\end{scope}

{\scriptsize
\draw (0.25,0.5) node[anchor=east] {$A$};
\draw (0.75,0.5) node[anchor=west] {$A'$};
\draw (0.5,0) node[anchor=north] {$B$};
\draw (A0) node[anchor=east] {$0$};
\draw (A1) node[anchor=south] {$1$};
\draw (A2) node[anchor=west] {$2$};
}
\draw (2,0.5) node {$\mapsto \quad$};
\end{scope}
\begin{scope}[scale=1.5, xshift=2.7cm, yshift=0.2cm]
	\small
	\node (LB) at (-0.3,0.25){$\Big\{$};
	\node (A0) at (0,0.5){$A$};
	\node (A1) at (1,0.5){$A'$};
	\node (B) at (0.5,-0.2){$B$};
	\node (RB) at (1.3,0.25){$\Big\}$};

\draw[->] (A0) -- (B);
\draw[->] (B) -- (A1);
\draw[->] (A1) -- (A0);

{\tiny
	\draw (0.5,0.5) node[anchor=south] {$+1$};
}
\end{scope}
\end{tikzpicture}
\end{equation*}
of all distinguished triangles in ${\mathcal T}$ involving the objects determined by the edge labels.
To a triangulated square with $E$-labeled edges, we attach the following collections of diagrams
\begin{equation}\label{eq:tft2seg}
\begin{tikzpicture}[>=latex,scale=1.0, baseline=(current  bounding  box.center)]
	\begin{scope}[xshift=-3.2cm, scale=0.95]
	\scriptsize
\begin{scope}[scale=1.4, xshift=4cm]
\coordinate (A0) at (0,0);
\coordinate (A1) at (0,1);
\coordinate (A2) at (1,1);
\coordinate (A3) at (1,0);

\path[fill opacity=0.4, fill=blue!50] (A0) -- (A3) -- (A2) -- (A1) -- cycle;

\begin{scope}[decoration={
    markings,
    mark=at position 0.55 with {\arrow{>}}}
    ] 
\draw[postaction={decorate}] (A0) -- (A1);
\draw[postaction={decorate}] (A1) -- (A2);
\draw[postaction={decorate}] (A2) -- (A3);
\draw[postaction={decorate}] (A0) -- (A2);
\draw[postaction={decorate}] (A0) -- (A3);
\end{scope}

{\scriptsize
\draw (0,0.5) node[anchor=east] {$A$};
\draw (0.5,1) node[anchor=south] {$A'$};
\draw (1,0.5) node[anchor=west] {$A''$};
\draw (0.5,0) node[anchor=north] {$C$};
\draw (A0) node[anchor=east] {$0$};
\draw (A1) node[anchor=east] {$1$};
\draw (A2) node[anchor=west] {$2$};
\draw (A3) node[anchor=west] {$3$};
}
\draw (1.8,0.5) node {\normalsize $\mapsto \quad$};
\end{scope}
\begin{scope}[scale=2.0, xshift=4.4cm, yshift=-0.15cm]
	\small
	\node (LB) at (-0.2,0.5){$\Big\{$};
	\node (A0) at (0,0.5){$A$};
	\node (A1) at (0.5,1){$A'$};
	\node (A2) at (1,0.5){$A''$};
	\node (A3) at (0.5,0){$C$};
	\node (Am) at (0.5,0.5){$B$};
	\node (st1) at (0.35,0.65){$\ast$};
	\node (st2) at (0.65,0.35){$\ast$};
	\node (RB) at (1.2,0.5){$\Big\}$};
	\node (comma) at (1.3,0.4){\normalsize ,};

\draw[->] (A0) -- (Am);
\draw[->] (Am) -- (A1);
\draw[->] (A1) -- (A0);
\draw[->] (A2) -- (Am);
\draw[->] (A2) -- (A1);
\draw[->] (A3) -- (A2);
\draw[->] (A0) -- (A3);
\draw[->] (Am) -- (A3);

{\tiny
	\draw (0.15,0.85) node {$+1$};
	\draw (0.85,0.85) node {$+1$};
\draw (0.7,0.5) node[anchor=south] {$+1$};
}
\end{scope}
\end{scope}
\begin{scope}[xshift=3.2cm, scale=0.95]
	\scriptsize
\begin{scope}[scale=1.4, xshift=4cm]
\coordinate (A0) at (0,0);
\coordinate (A1) at (0,1);
\coordinate (A2) at (1,1);
\coordinate (A3) at (1,0);

\path[fill opacity=0.4, fill=blue!50] (A0) -- (A3) -- (A2) -- (A1) -- cycle;

\begin{scope}[decoration={
    markings,
    mark=at position 0.55 with {\arrow{>}}}
    ] 
\draw[postaction={decorate}] (A0) -- (A1);
\draw[postaction={decorate}] (A1) -- (A2);
\draw[postaction={decorate}] (A2) -- (A3);
\draw[postaction={decorate}] (A1) -- (A3);
\draw[postaction={decorate}] (A0) -- (A3);
\end{scope}

\draw (0,0.5) node[anchor=east] {$A$};
\draw (0.5,1) node[anchor=south] {$A'$};
\draw (1,0.5) node[anchor=west] {$A''$};
\draw (0.5,0) node[anchor=north] {$C$};
\draw (A0) node[anchor=east] {$0$};
\draw (A1) node[anchor=east] {$1$};
\draw (A2) node[anchor=west] {$2$};
\draw (A3) node[anchor=west] {$3$};

\draw (1.8,0.5) node {\normalsize $\mapsto \quad$};
\end{scope}
\begin{scope}[scale=2.0, xshift=4.4cm, yshift=-0.15cm]
	\small
	\node (LB) at (-0.2,0.5){$\Big\{$};
	\node (A0) at (0,0.5){$A$};
	\node (A1) at (0.5,1){$A'$};
	\node (A2) at (1,0.5){$A''$};
	\node (A3) at (0.5,0){$C$};
	\node (Am) at (0.5,0.5){$B'$};
	\node (st1) at (0.65,0.65){$\ast$};
	\node (st2) at (0.35,0.35){$\ast$};
	\node (RB) at (1.2,0.5){$\Big\}$};

\draw[->] (A1) -- (Am);
\draw[->] (Am) -- (A2);
\draw[->] (A2) -- (A1);
\draw[->] (A0) -- (A3);
\draw[->] (A3) -- (Am);
\draw[->] (Am) -- (A0);
\draw[->] (A1) -- (A0);
\draw[->] (A3) -- (A2);

{ \tiny
	\draw (0.15,0.85) node {$+1$};
	\draw (0.85,0.85) node {$+1$};
	\draw (0.3,0.5) node[anchor=south] {$+1$};
}

\end{scope}
\end{scope}
\end{tikzpicture}
\end{equation}
where the $\ast$-marked triangles are distinguished, the unmarked triangles commute, and the objects
$B$ and $B'$ are allowed to vary.
The two types of diagrams correspond to the upper and lower cap of an octahedron. On a
heuristic level, the role of the associativity in Equation \eqref{dy-eq:assoc} will now be played by the {\em octahedral axiom}
which allows us to pass from one triangulation of the square to the other.
More generally, to an $E$-labeled polygon with a chosen triangulation we associate the
collection of certain {\em Postnikov systems} \cite{gelfand-manin} such as
\begin{equation}\label{eq:postnikov}
\begin{tikzpicture}[>=latex,scale=1.0, baseline=(current  bounding  box.center)]
\begin{scope}[scale=0.2]
\coordinate (A2) at (1,4);
\coordinate (A3) at (5,0);
\coordinate (A4) at (3,-4);
\coordinate (A0) at (-1,-4);
\coordinate (A1) at (-3,0);

\path[fill opacity=0.4, fill=blue!50] (A0) -- (A1) -- (A2) -- (A3) -- (A4) -- cycle;

\begin{scope}[decoration={
    markings,
    mark=at position 0.55 with {\arrow{>}}}
    ] 
\draw[postaction={decorate}] (A0) -- (A1);
\draw[postaction={decorate}] (A0) -- (A2);
\draw[postaction={decorate}] (A0) -- (A3);
\draw[postaction={decorate}] (A1) -- (A2);
\draw[postaction={decorate}] (A2) -- (A3);
\draw[postaction={decorate}] (A3) -- (A4);
\draw[postaction={decorate}] (A0) -- (A4);
\end{scope}

{\scriptsize
\draw (4,2) node[anchor=south] {$A''$};
\draw (4,-2.5) node[anchor=west] {$A'''$};
\draw (1,-4) node[anchor=north] {$D$};
\draw (-2.5,-2.5) node[anchor=east] {$A$};
\draw (-2,2) node[anchor=south] {$A'$};
\draw (A0) node[anchor=east] {$0$};
\draw (A1) node[anchor=east] {$1$};
\draw (A2) node[anchor=south] {$2$};
\draw (A3) node[anchor=west] {$3$};
\draw (A4) node[anchor=west] {$4$};
}
\draw (12.0,-0.5) node {$\mapsto \quad$};
\end{scope}
\begin{scope}[scale=1.8, xshift=2cm, yshift=0.2cm,xscale=0.8]
	\small
	\node (LB) at (-0.3,-0.25){$\Big\{$};
	\node (A) at (0,0){$A$};
	\node (A1) at (0.5,-0.5){$A'$};
	\node (B) at (1,0){$B$};
	\node (A2) at (1.5,-0.5){$A''$};
	\node (C) at (2,0){$C$};
	\node (A3) at (2.5,-0.5){$A'''$};
	\node (D) at (3,0){$D$};
	\node (st1) at (0.5,-0.2){$\ast$};
	\node (st2) at (1.5,-0.2){$\ast$};
	\node (st3) at (2.5,-0.2){$\ast$};
	\node (RB) at (3.3,-0.25){$\Big\}.$};

\draw[->] (A) -- (B);
\draw[->] (B) -- (A1);
\draw[->] (A1) -- (A);
\draw[->] (B) -- (C);
\draw[->] (C) -- (A2);
\draw[->] (A2) -- (B);
\draw[->] (C) -- (D);
\draw[->] (D) -- (A3);
\draw[->] (A3) -- (C);
\draw[->] (A2) -- (A1);
\draw[->] (A3) -- (A2);

{\tiny
	\draw (0.1,-0.2) node[anchor=north] {$+1$};
	\draw (1.1,-0.2) node[anchor=north] {$+1$};
	\draw (2.1,-0.2) node[anchor=north] {$+1$};
	\draw (1.0,-0.5) node[anchor=north] {$+1$};
	\draw (2.0,-0.5) node[anchor=north] {$+1$};
}
\end{scope}
\end{tikzpicture}
\end{equation}
The analog of the Frobenius structure in Section \ref{subsec:algebras} turns out to be a $2$-periodic
structure on ${\mathcal T}$: an isomorphism of functors $\Sigma^2 \simeq \operatorname{id}$. This structure allows us to rewrite any distinguished triangle as
\begin{equation*}
\begin{tikzpicture}[>=latex, scale=1.0, baseline=(current  bounding  box.center)]
\begin{scope}[scale=1.5, xshift=2.7cm, yshift=0.2cm]
	\scriptsize
	\node (A0) at (0,0.5){$A$};
	\node (A1) at (1,0.5){$A'$};
	\node (B) at (0.5,-0.2){$B$};

\draw[->] (A0) -- (B);
\draw[->] (B) -- (A1);
\draw[->] (A1) -- (A0);

{\tiny
	\draw (0.5,0.5) node[anchor=south] {$+1$};
}
\draw (2,0.2) node {\normalsize $\sim$};

\end{scope}
\begin{scope}[scale=1.5, xshift=5.7cm, yshift=0.2cm]
	\scriptsize
	\node (A0) at (0,0.5){$A$};
	\node (A1) at (1,0.5){$A'$};
	\node (B) at (0.5,-0.2){$\Sigma B$};

\draw[->] (A0) -- (B);
\draw[->] (B) -- (A1);
\draw[->] (A1) -- (A0);

{\tiny
	\draw (0.5,0.5) node[anchor=south] {$+1$};
	\draw (0.9,0.1) node {$+1$};
	\draw (0.1,0.1) node {$+1$};
}
\end{scope}
\end{tikzpicture}
\end{equation*}
where the right-hand form exhibits a cyclic symmetry analogous to the symmetry of the trace
expression ${\operatorname{tr}}(e_i e_j e_k^*)$ from Section \ref{subsec:algebras}. 

These heuristics suggest the existence of invariants of marked oriented surfaces associated with any
$2$-periodic triangulated category ${\mathcal T}$. Further, state sum formulas should lead to a
description of these invariants in terms of {\em surface Postnikov systems}: collections of
distinguished triangles in ${\mathcal T}$ parametrized by a chosen triangulation of the surface.

We have already defined the concepts needed to provide a rigorous version of the above heuristics:
\begin{enumerate}
		\item to address the issues with the octahedral axiom, we should assume that the
			triangulated category comes with an enhancement: it is the homotopy category
			of a differential graded category $T$,
		\item the invariant associated to a polygon $P_n$ is the space $\S_n$
			of $n$-simplices in the $S_{\bullet}$-construction of the dg nerve of $T$,
		\item the $2$-Segal property allows us to identify the space $\S_n$ with the
			space of Postnikov systems corresponding to a chosen triangulation of $P_n$,
\end{enumerate}
The only missing ingredient is an interpretation of $2$-periodicity. One of the main points of this
section is, that it can be captured in terms of Connes' \cite{connes} {\em cyclic category $\Lambda$}. The category $\Lambda$
has an object $\cn$ for every $n \ge 0$. Let $S^1$ denote the unit circle in $\CC$ and $\mu_n
\subset S^1$ the subset of $n$th roots of unity. The morphisms from $\cm$ to $\cn$ are then given by
homotopy classes of monotone degree $1$ maps $\varphi: S^1 \to S^1$ such that $\varphi(\mu_{n+1})
\subset \mu_{m+1}$.  There is a natural embedding $\Delta \subset \Lambda$ so that every morphism in
$\Lambda$ can be uniquely expressed as the composite of a cyclic rotation and a morphism in
$\Delta$.  A {\em cyclic structure} on a simplicial object $X: \Delta \to \C$ is a lift 
\[
	\xymatrix{ \Delta \ar[d] \ar[r] & \C\\
		\Lambda \ar@{-->}[ur]. & }
\]

\begin{theo}[\cite{dk-triangulated}]\label{dy-ka-thm1}
	Let $T$ be a pretriangulated differential $\Zt$-graded category. Denote
	by $\S_{\bullet}(T)$ the simplicial space given by Waldhausen's $\S_{\bullet}$-construction of
	the dg nerve of $T$. Then $\S_{\bullet}(T)$ admits a canonical cyclic structure.
\end{theo}
\begin{proof}
	Let $\A^n$ denote the $k$-linear envelope of the linearly ordered set $\{1,2,\dots,n\}$. We
	may consider $\A^n$ as a differential $\Zt$-graded category which is concentrated in even
	degrees. Given a differential $\Zt$-graded category $T$, one can construct an equivalence of
	spaces
	\[
		S_{n}(\Ndg(T)) \overset{\simeq}{\lra} \Map(\A^n, T)
	\]
	where $\Map$ denotes the mapping space of the localization of the category $\dgcat^{(2)}$ of
	differential $\Zt$-graded categories along Morita equivalences.

	The key of the proof is to replace $\A^n$ by a Morita equivalent model in which the cyclic
	functoriality is more apparent: the differential $\Zt$-graded category
	\[
		\operatorname{MF}^{\ZZ/(n+1)}(k[z], z^{n+1})
	\] 
	of $\ZZ/(n+1)$-graded matrix factorizations of the polynomial $z^{n+1}$. More precisely, we focus on the
	full dg subcategory $F^n \subset \operatorname{MF}(k[z], z^{n+1})$ spanned by the objects
	\[
		[i,j] := \xymatrix{
			k[z](i) \ar@/_1ex/[r]_{z^{j-i}} & k[z](j)\ar@/_1ex/[l]_{z^{i-j}}
			 }.
	\]
	Here, $i$ and $j$ range over $\ZZ/(n+1)$ and, in the symbols $z^{i-j}$ and $z^{j-i}$, we
	replace the exponent by its representative in $\{0,1,\dots,n\}$. There is a dg functor
	\[
		\A^n \to \F^n, i \mapsto [0,i]
	\]
	which is a Morita equivalence. The dg categories $\F^{\bullet}$ naturally organize into a
	cocyclic object
	\[
		\F^{\bullet}: \Lambda \lra \dgcat^{(2)}
	\]
	whose underlying cosimplicial object is $2$-Segal (where the pullback conditions become
	pushout conditions) after localizing along Morita equivalences. 
	The cyclic structure on $S_{\bullet}(\Ndg(T))$ is now obtained via
	\[
		S_{\bullet}(\Ndg(T)) \simeq \Map(\F^{\bullet}, T).
	\]
\end{proof}

The following version of a result of \cite{dk-triangulated} shows that the expected surface invariants can
indeed be defined and computed in terms of a limit which should be regarded as the
analog of the state sum \eqref{eq:statesum}.

\begin{theo}\label{dy-ka-thm2}
	Let $\C$ be an $\infty$-category with limits and let $X$ be a $2$-Segal
	object in $\C$ equipped with a cyclic structure. 
	Let $(S,M)$ be a stable marked oriented surface. Then there exists an object
	$X_{(S,M)}$ in $\C$ which, for every triangulation $\Delta(S,M)$ of $(S,M)$,
	comes equipped with canonical isomorphism
	\[
		X_{(S,M)} \overset{\simeq}{\longrightarrow} \mathop{\operatorname{lim}}_{\Lambda^n \to
		\Delta(S,M)} X_n.
	\]
	Further, the mapping class group of $(S,M)$ acts coherently on $X_{(S,M)}$ via equivalences in $\C$.
\end{theo}

An application of the theorem to the $S_{\bullet}$-construction of a pretriangulated differential
$\Zt$-graded category $T$ yields the surface invariants predicted heuristically in Section
\ref{subsec:triangulated}. As a remarkable feature of the proof, note that we can give a universal
variant of this construction by applying the theorem to the cocyclic object $\F^{\bullet}$. As a
result, we obtain, for every stable oriented marked surface $(S,M)$ a dg category
\[
	\F^{(S,M)}
\]
which, for every triangulation $\Delta(S,M)$, comes equipped with a universal surface Postnikov
system. We give some examples:

\subsection{Examples}

\subsubsection{Boundary-marked disk}
Consider the marked surface
  \[
  (S,M) = 
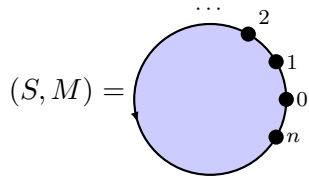
\begin{tikzpicture}[baseline=0,>=latex,scale=0.5, decoration={
    markings,
    mark=at position 0.55 with {\arrow{>}}}]
\begin{scope}

	\fill[fill opacity=0.4, fill=blue!50] (0:2cm) arc (0:360:2cm);
	\draw[thick] (0:2cm) arc (0:360:2cm);
	\fill (-30:2cm) circle [radius=0.2cm];
	\fill (0:2cm) circle [radius=0.2cm];
	\fill (30:2cm) circle [radius=0.2cm];
	\fill (60:2cm) circle [radius=0.2cm];

	\draw[postaction={decorate}] (0:2cm) arc (0:360:2cm);

	\node[anchor=south] at (90:2cm) {\scriptsize $\cdots$};
	\node[anchor=west] at (0:2cm) {\scriptsize $0$};
	\node[anchor=west] at (30:2cm) {\scriptsize $1$};
	\node[anchor=south west] at (60:2cm) {\scriptsize $2$};
	\node[anchor=west] at (-30:2cm) {\scriptsize $n$};

\end{scope}
\end{tikzpicture}
\]
given by a disk with $n$ marked points on its boundary so that, by construction, we have
\[
	\F^{(S,M)} \simeq \F^n \subset \on{MF}^{\ZZ/(n+1)}(\k[z], z^{n+1}).
\]
The universal Postnikov system in $\F^{(S,M)}$ corresponding to the triangulation with all edges starting in 
$0$ (cf. \eqref{eq:postnikov}) is given by 
\[
	\xymatrix@=1ex{[0,1] \ar[rr] & & [0,2]\ar[dl] \ar[rr] & & [0,3]\ar[dl] \ar[r]& \cdots\ar[r]&
		[0,n-1]\ar[rr] & & [0,n]\ar[dl]\\
		& [1,2]\ar[ul]^{+1} & & [2,3] \ar[ul]^{+1} & & &&[n-1,n] \ar[ul]^{+1} & }
\]
The action of the mapping class group $\on{Mod}(S,M) \cong \ZZ/(n+1)$ is given by shifting the
$\ZZ/(n+1)$-grading and is part of the cocyclic structure on $\F^{\bullet}$.

\subsubsection{Disk with two marked points}
  \[
  (S,M) = 
\begin{tikzpicture}[baseline=0,>=latex,scale=0.5, decoration={
    markings,
    mark=at position 0.55 with {\arrow{>}}}]
\begin{scope}

	\fill[fill opacity=0.4, fill=blue!50] (0:2cm) arc (0:360:2cm);
	\draw[thick] (0:2cm) arc (0:360:2cm);
	\fill (0:0cm) circle [radius=0.2cm];
	\fill (0:2cm) circle [radius=0.2cm];

	\draw[postaction={decorate}] (0:2cm) arc (0:360:2cm);
	\draw[postaction={decorate}] (0:0cm) -- (0:2cm);

	\node[anchor=east] at (0:0cm) {\scriptsize $0$};
	\node[anchor=south east] at (0:2cm) {\scriptsize $2$};
	\node[anchor=north east] at (0:2cm) {\scriptsize $1$};

	\node at (0:8cm) {$\leadsto\quad \F^{(S,M)} = \on{D}^b(\on{coh} \mathbb A^1)^{(2)}$};
\end{scope}
\end{tikzpicture}
\]
The universal Postnikov system in $\F^{(S,M)}$, corresponding to the indicated triangulation, is given by 
\[
\begin{tikzpicture}[>=latex, scale=1.0, baseline=(current  bounding  box.center)]
\begin{scope}[scale=1.4]
\coordinate (A0) at (0,0);
\coordinate (A1) at (0.5,0.8);
\coordinate (A2) at (1,0);

\path[fill opacity=0.4, fill=blue!50] (A0) -- (A2) -- (A1) -- cycle;

\begin{scope}[decoration={
    markings,
    mark=at position 0.55 with {\arrow{>}}}
    ] 
\draw[postaction={decorate}] (A0) -- (A1);
\draw[postaction={decorate}] (A1) -- (A2);
\draw[postaction={decorate}] (A0) -- (A2);
\end{scope}

{\scriptsize
\draw (0.25,0.5) node[anchor=east] {$X$};
\draw (0.75,0.5) node[anchor=west] {$Y$};
\draw (0.5,0) node[anchor=north] {$X$};
\draw (A0) node[anchor=east] {$0$};
\draw (A1) node[anchor=south] {$1$};
\draw (A2) node[anchor=west] {$2$};
}
\draw (2,0.5) node {$\leadsto \quad$};
\end{scope}
\begin{scope}[scale=1.5, xshift=2.7cm, yshift=0.2cm]
	\small
	\node (A0) at (0,0.5){$\k[x]$};
	\node (A1) at (1,0.5){$\k$};
	\node (B) at (0.5,-0.2){$\k[x]$};
	\node (st2) at (0.1,0.1){$x$};

	\draw[->] (A0) -- (B);
\draw[->] (B) -- (A1);
\draw[->] (A1) -- (A0);

{\tiny
	\draw (0.5,0.5) node[anchor=south] {$+1$};
}
\end{scope}
\end{tikzpicture}
\]
The mapping class group $\on{Mod}(S,M)$ is trivial.

\subsubsection{Annulus with two marked points}
  \[
  (S,M) = 
\begin{tikzpicture}[baseline=0,>=latex,scale=0.5, decoration={
    markings,
    mark=at position 0.55 with {\arrow{>}}}]
\begin{scope}

	\fill[fill opacity=0.4, fill=blue!50] (0:1cm) -- (0:3cm)
	arc (0:360:3cm) -- (0:1cm)
	arc (360:0:1cm) -- cycle;

	\draw[thick] (0:3cm) arc (0:360:3cm);
	\draw[thick] (0:1cm) arc (0:360:1cm);
	\fill (0:1cm) circle [radius=0.2cm];
	\fill (0:3cm) circle [radius=0.2cm];

	\draw[postaction={decorate}] (0:1cm) arc (0:-360:1cm);
	\draw[postaction={decorate}] (0:1cm) -- (0:3cm);
	\draw[domain=0:6.28,samples=200,smooth] plot (canvas polar
	cs:angle=\x r,radius={3cm - 2cm*sin(0.25*\x r )});
	\draw[postaction={decorate}] (0:3cm) arc (0:360:3cm);

	\node[anchor=south west] at (0:1cm) {\scriptsize $1$};
	\node[anchor=south east] at (0:3cm) {\scriptsize $2$};
	\node[anchor=north east] at (0:1cm) {\scriptsize $0$};
	\node[anchor=north west] at (0:1cm) {\scriptsize $0'$};
	\node[anchor=north east] at (0:3cm) {\scriptsize $2'$};
	\node[anchor=south west] at (0:3cm) {\scriptsize $1'$};

	\node at (0:10cm) {$\leadsto\quad \F^{(S,M)} = \on{D}^b(\on{coh} \mathbb P^1)^{(2)}$};
\end{scope}
\end{tikzpicture}
\]

The universal Postnikov system in $\F^{(S,M)}$, corresponding to the indicated triangulation, is given by 
\[
\begin{tikzpicture}[>=latex, scale=1.0, baseline=(current  bounding  box.center)]
\begin{scope}[scale=1.4]
\coordinate (A0) at (0,0);
\coordinate (A1) at (0.5,0.8);
\coordinate (A2) at (1,0);

\path[fill opacity=0.4, fill=blue!50] (A0) -- (A2) -- (A1) -- cycle;

\begin{scope}[decoration={
    markings,
    mark=at position 0.55 with {\arrow{>}}}
    ] 
\draw[postaction={decorate}] (A0) -- (A1);
\draw[postaction={decorate}] (A1) -- (A2);
\draw[postaction={decorate}] (A0) -- (A2);
\end{scope}

{\scriptsize
\draw (0.25,0.5) node[anchor=east] {$X$};
\draw (0.75,0.5) node[anchor=west] {$Z'$};
\draw (0.5,0) node[anchor=north] {$Y$};
\draw (A0) node[anchor=east] {$0'$};
\draw (A1) node[anchor=south] {$1'$};
\draw (A2) node[anchor=west] {$2'$};
}
\end{scope}
\begin{scope}[scale=1.4, xshift=2cm]
\coordinate (A0) at (0,0);
\coordinate (A1) at (0.5,0.8);
\coordinate (A2) at (1,0);

\path[fill opacity=0.4, fill=blue!50] (A0) -- (A2) -- (A1) -- cycle;

\begin{scope}[decoration={
    markings,
    mark=at position 0.55 with {\arrow{>}}}
    ] 
\draw[postaction={decorate}] (A0) -- (A1);
\draw[postaction={decorate}] (A1) -- (A2);
\draw[postaction={decorate}] (A0) -- (A2);
\end{scope}

{\scriptsize
\draw (0.25,0.5) node[anchor=east] {$Z$};
\draw (0.75,0.5) node[anchor=west] {$Y$};
\draw (0.5,0) node[anchor=north] {$X$};
\draw (A0) node[anchor=east] {$0$};
\draw (A1) node[anchor=south] {$1$};
\draw (A2) node[anchor=west] {$2$};
}
\draw (2,0.5) node {$\leadsto \quad$};
\end{scope}
\begin{scope}[scale=2, xshift=3.3cm, yshift=-.15cm]
	\small
	\node (B) at (0.5,1) {$\k_0$};
	\node (A0) at (0,0.5) {$\O$};
	\node (A1) at (1,0.5) {$\O(1)$};
	\node (C) at (0.5,0) {$\k_{\infty}$};
	\draw[transform canvas={yshift=0.5ex},->] (A0) -- (A1) node[above,midway] {\tiny $x$};
	\draw[transform canvas={yshift=-0.5ex},->] (A0) -- (A1)node[below,midway] {\tiny $y$};
      \draw[->] (A1) -- (B);
      \draw[->] (B) -- (A0) node[above,midway] {\tiny $+1$};
      \draw[->] (A1) -- (C);
      \draw[->] (C) -- (A0) node[below,midway] {\tiny $+1$};

\end{scope}
\end{tikzpicture}
\]
The generator of $\on{Mod}(S,M) \cong \ZZ$ acts via $- \otimes \O(1)$.

\subsubsection{Sphere with $3$ marked points} 
  \[
    (S,M) = 
  \begin{tikzpicture}[scale=1.3,baseline=0,>=latex, decoration={
    markings,
    mark=at position 0.55 with {\arrow{>}}}]
    \draw (-1,0) arc (180:360:1cm and 0.5cm);
    \draw[dashed] (-1,0) arc (180:0:1cm and 0.5cm);
    \draw (0,0) circle (1cm);
    \shade[ball color=blue!50,opacity=0.20] (0,0) circle (1cm);
    \node at (3.9,0) {$\leadsto \; \F^{(S,M)} = \on{D}^b(\on{coh} \underbrace{\k[x,y]/(xy)}_{S} )^{(2)}$};
	\fill (-1,0) circle [radius=0.04cm];
	\fill (0,-0.5) circle [radius=0.04cm];
	\fill (1,0) circle [radius=0.04cm];
	{\scriptsize
      \draw (-1,0) node[anchor=south east] {$2'$};
      \draw (0,-0.5) node[anchor=south] {$1'$};
      \draw (1,0) node[anchor=south west] {$0'$};
      \draw (-1,0) node[anchor=north east] {$0$};
      \draw (0,-0.5) node[anchor=north] {$1$};
      \draw (1,0) node[anchor=north west] {$2$};
    }
\end{tikzpicture}
\]

The universal Postnikov system is given by
\[
\begin{tikzpicture}[>=latex, scale=1.0, baseline=(current  bounding  box.center)]
\begin{scope}[scale=1.4]
\coordinate (A0) at (0,0);
\coordinate (A1) at (0.5,0.8);
\coordinate (A2) at (1,0);

\path[fill opacity=0.4, fill=blue!50] (A0) -- (A2) -- (A1) -- cycle;

\begin{scope}[decoration={
    markings,
    mark=at position 0.55 with {\arrow{>}}}
    ] 
\draw[postaction={decorate}] (A0) -- (A1);
\draw[postaction={decorate}] (A1) -- (A2);
\draw[postaction={decorate}] (A0) -- (A2);
\end{scope}

{\scriptsize
\draw (0.25,0.5) node[anchor=east] {$X$};
\draw (0.75,0.5) node[anchor=west] {$Z$};
\draw (0.5,0) node[anchor=north] {$Y$};
\draw (A0) node[anchor=east] {$0'$};
\draw (A1) node[anchor=south] {$1'$};
\draw (A2) node[anchor=west] {$2'$};
}
\end{scope}
\begin{scope}[scale=1.4, xshift=2cm]
\coordinate (A0) at (0,0);
\coordinate (A1) at (0.5,0.8);
\coordinate (A2) at (1,0);

\path[fill opacity=0.4, fill=blue!50] (A0) -- (A2) -- (A1) -- cycle;

\begin{scope}[decoration={
    markings,
    mark=at position 0.55 with {\arrow{>}}}
    ] 
\draw[postaction={decorate}] (A0) -- (A1);
\draw[postaction={decorate}] (A1) -- (A2);
\draw[postaction={decorate}] (A0) -- (A2);
\end{scope}

{\scriptsize
\draw (0.25,0.5) node[anchor=east] {$\Sigma Z$};
\draw (0.75,0.5) node[anchor=west] {$\Sigma X$};
\draw (0.5,0) node[anchor=north] {$\Sigma Y$};
\draw (A0) node[anchor=east] {$0$};
\draw (A1) node[anchor=south] {$1$};
\draw (A2) node[anchor=west] {$2$};
}
\draw (2,0.5) node {$\leadsto \quad$};
\end{scope}
\begin{scope}[scale=2, xshift=3.3cm, yshift=.1cm]
	\small
	\node (A0) at (0,0.5){$S/(x)$};
	\node (A1) at (1,0.5){$S/(y)$};
	\node (B) at (0.5,-0.2){$S$};

  \draw[transform canvas={xshift=0.5ex},->] (A0) -- (B) node[right,midway] {\tiny $y$};
  \draw[transform canvas={xshift=-0.5ex},->] (B) -- (A1);
  \draw[transform canvas={yshift=-0.5ex},->] (A1) -- (A0) node[below,midway] {\tiny $+1$};
  \draw[transform canvas={xshift=-0.5ex},->] (B) -- (A0); 
  \draw[transform canvas={xshift=+0.5ex},->] (A1) -- (B) node[right,midway] {\tiny $x$};
  \draw[transform canvas={yshift=+0.5ex},->] (A0) -- (A1)node[above,midway] {\tiny $+1$};

\end{scope}
\end{tikzpicture}
\]
There is an action of $\on{Mod}(S,M) \cong S_3$ on $\F^{(S,M)}$ which permutes the objects $\Sigma
S$, $S/(x)$ and $S/(y)$.

\begin{rem} As a final remark, we conclude by mentioning an interpretation of the dg categories
	$\F^{(S,M)}$ as purely topological Fukaya categories which is due to Kontsevich
	\cite{kontsevich-fukaya}. In this context, the state sum formula given by the limit in
	Theorem \ref{dy-ka-thm2} can then be regarded as implementing a $2$-dimensional instance of
	Kontsevich's proposal on localizing the Fukaya category along a singular Lagrangian spine
	(given in our context as the dual graph of the chosen triangulation). 
\end{rem}

\bibliographystyle{halpha} 
\bibliography{biblio}

\end{document}